\documentclass[11pt]{article}
\usepackage{amsmath,amsthm,amsfonts,amssymb}
\usepackage[all]{xy}
\usepackage[usenames]{color} 
\usepackage{makeidx}

\usepackage{xr}

\newcommand\brefRK[1]{\ref{RK#1}}
\newcommand\brefWA[1]{\ref{WA#1}}

\newcommand\refRK[1]{II.\brefRK{#1}}
\newcommand\refWA[1]{III.\brefWA{#1}}

\newcounter{enumitemp}
\newenvironment{enumeratecontinue}{
 \setcounter{enumitemp}{\value{enumi}}
 \begin{enumerate}
 \setcounter{enumi}{\value{enumitemp}}
}
{
 \end{enumerate}
}

\newcommand\pref[1]{(\ref{#1})}

\newtheorem{thm}{Theorem}[section]

\newtheorem{theorem}[thm]{Theorem}
\newtheorem*{theorem*}{Theorem}

\newtheorem{lemma}[thm]{Lemma}

\newtheorem{corollary}[thm]{Corollary}
\newtheorem{proposition}[thm]{Proposition}
\newtheorem*{proposition*}{Proposition}

\newtheorem{fact}[thm]{Fact}

\theoremstyle{definition}

\newtheorem{definition}[thm]{Definition} 

\newtheorem*{defn*}{Definition}

\newtheorem{notn}[thm]{Notation}

\newtheorem{remark}[thm]{Remark}
\theoremstyle{remark}

\newcounter{remarks}
{\paragraph*{Remarks}\ skip
 \begin{list}{\arabic{remarks}. }{\usecounter{remarks}%
 \setlength{\leftmargin}{0in}%
 \setlength{\rightmargin}{0in}%
 \setlength{\labelsep}{0pt}%
 \setlength{\labelwidth}{0pt}%
 \setlength{\listparindent}{0pt}%
 }
}
{
\end{list}
}

\newcommand\from\colon
\newcommand\inv{{-1}}
\newcommand\subgroup{<}
\newcommand\infinity\infty

\newcommand\supp{\text{supp}}
\newcommand\disjunion\amalg
\newcommand\act\curvearrowright

\DeclareMathOperator{\Fix}{Fix}
\DeclareMathOperator{\Per}{Per}

\DeclareMathOperator\core{core}
\DeclareMathOperator\Link{Link}
\DeclareMathOperator\IA{IA}
\DeclareMathOperator\closure{cl}

\DeclareMathOperator\Isom{Isom}

\newcommand{\Z}{{\mathbb Z}}
\newcommand{\C}{{\mathcal C}}

\renewcommand\P{{\mathcal P}}

\newcommand\U{{\mathcal U}}

\newcommand\K{{\mathcal K}}

\newcommand{\Out}{\mathsf{Out}}
\newcommand{\Aut}{\mathsf{Aut}}
\newcommand{\Inn}{\mathsf{Inn}}

\newcommand{\Stab}{\mathsf{Stab}}

\newcommand{\F}{\mathcal F}
\newcommand{\rtt}{relative train track map}

\renewcommand\L{\mathcal L}

\def\B{\mathcal B}
\newcommand{\A}{\mathcal A}

\newcommand{\ti} {\tilde}

\newcommand{\eg}{EG}
\newcommand{\noneg}{NEG}
\newcommand\Lim{\text{Lim}}
\renewcommand\neg\noneg

\newcommand{\wt}{\widetilde}

\newcommand{\ct}{CT}
\newcommand{\cts}{CTs}
\newcommand\free{{\text{f}}}
\newcommand\nonfree{{\text{nf}}}

\newcommand\BH{\cite{BestvinaHandel:tt}}
\newcommand\BookZero{\cite{BFH:laminations}}
\newcommand\BookOne{\cite{BFH:TitsOne}}

\newcommand\recognition{\cite{FeighnHandel:recognition}}

\newcommand\Intro{\cite{HandelMosher:SubgroupsIntro}}

\newcommand\PartTwo{Part II \cite{HandelMosher:SubgroupsII}}
\newcommand\PartThree{Part III \cite{HandelMosher:SubgroupsIII}}
\newcommand\PartFour{Part IV \cite{HandelMosher:SubgroupsIV}}

 \newcommand\rc{}

\DeclareMathOperator\interior{int}
\newcommand\bdy\partial

\newcommand\intersect\cap
\newcommand\union\cup
\newcommand\<\langle
\renewcommand\>\rangle
\newcommand\meet\wedge
\newcommand\composed{\circ}
\newcommand\cross\times
\newcommand\restrict{\bigm |}
\newcommand\wh{\widehat}
\newcommand\inject\hookrightarrow
\newcommand\reals{\mathbf{R}}
\DeclareMathOperator\Length{Length}
\newcommand\abs[1]{\left|#1\right|}

\DeclareMathOperator\rank{rank}
\DeclareMathOperator\BCC{BCC} 
\DeclareMathOperator\Hull{\mathcal{H}}
 
 \newcommand\surjection\twoheadrightarrow
 
\newcommand\suchthat{\bigm|}
\newcommand\hyp{\mathbf{H}}
\DeclareMathOperator\MCG{\mathcal{MCG}}
\DeclareMathOperator\Homeo{Homeo}

\newcommand{\Lambdapp}{\Lambda^+_\phi}

\newcommand{\Lambdapmp}{\Lambda^{\pm}_\phi}

\renewcommand\int{{\text{int}}}

\newcommand\inner{\iota}

\newcommand\norm[1]{\left| #1 \right|}

\newcommand\GeomModelsDef{Definitions~\ref{DefGeomModel} and~\ref{DefWeakGeomModel}}

\title{Subgroup decomposition in $\Out(F_n)$\\ Part I: Geometric Models}

\author{Michael Handel and Lee Mosher}

\makeindex

\begin{document}

\maketitle

\begin{abstract}
This is the first in a series of four papers, announced in \cite{HandelMosher:SubgroupsIntro}, that together develop a decomposition theory for subgroups of $\Out(F_n)$.

In this paper we develop further the theory of geometric \eg\ strata of relative train track maps originally introduced in \BookOne\ Section~5, with our focus trained on certain 2-dimensional models of such strata called ``geometric models'' and on the interesting properties of these models. A secondary purpose of this paper is to serve as a central reference for the whole series regarding basic facts of the general theory for elements of $\Out(F_n)$.
\end{abstract}

An outer automorphism $\phi \in \Out(F_n)$ is said to be \emph{geometric} if there exists a compact surface $S$ with nonempty boundary $\bdy S$, a homeomorphism $\Phi \from S \to S$, and an isomorphism between $\pi_1 S$ and $F_n$, such that the outer automorphism of $F_n$ induced by the homeomorphism $\Phi$ equals $\phi$. Any such surface $S$ deformation retracts to a finite graph $G$ called a \emph{spine}. The homeomorphism $\Phi$ of $S$ induces a homotopy equivalence $h$ of $G$, and either of $\Phi$ or $h$ may be considered as a topological representative of the outer automorphism $\phi$. If $\bdy S$ is connected and the mapping class of $\Phi$ is pseudo-Anosov then $\phi$ is \emph{fully irreducible}, meaning that for each proper, nontrivial free factor $A \subgroup F_n$, the conjugacy class of $A$ in $F_n$ is not $\phi$-periodic. Conversely, as proved in \cite{BestvinaHandel:tt} using train track maps, if $\phi$ is fully irreducible and if $\phi$ has a periodic conjugacy class then $\phi$ is geometric.

The concept of geometricity also arises in the context of strata of relative train track maps, as follows. Consider a compact surface $S$ with nonempty boundary $\bdy S$ and a homeomorphism $\Psi \from S \to S$ with pseudo-Anosov mapping class. We make no assumption that $\bdy S$ is connected, although we single out one component of $\bdy S$ as the ``top boundary'', the remaining ones being called the ``lower boundary components''. Consider also a graph $G'$, a homotopy equivalence $f' \from G' \to G'$, and a homotopically nontrivial gluing map from each lower boundary component to $G'$ such that the gluing maps are equivariant with respect to the action of $\Psi$ on the lower components of $\bdy S$ and the action of $h$ of $H$. The result of this gluing is a 2-complex $Y$ and a self homotopy equivalence $h \from Y \to Y$. Any such 2-complex deformation retracts to a spine $G'$ that includes $G''$ as a subcomplex, and the homotopy equivalence $h \from Y \to Y$ induces a homotopy equivalence $f'' \from G'' \to G''$ which extends~$f'$. Following \BookOne\ Definition~5.1.4, in Definition~\ref{DefGeometricStratum} below we use this construction to define geometricity of an \eg\ stratum of a relative train track map $f \from G \to G$ representing some $\phi \in \Out(F_n)$: the subgraphs $G'=G_{r-1} \subset G_r=G''$ will be adjacent filtration elements, the maps $f',f''$ are the restrictions of $f$, and the stratum $H_r = G_r \setminus G_{r-1}$ is, by definition, a geometric \eg\ stratum of~$f$. The 2-complex $Y$ and the homotopy equivalence $h \from Y \to Y$ define what we shall call a \emph{weak geometric model for $f$ with respect to the stratum $H_r$}, the adjective ``weak'' emphasizing that it is a topological model only for the restricted homotopy equivalence $f \restrict G_r$. A (full) \emph{geometric model} for $f$ with respect to $H_r$, which is indeed a topological model for all of $f$, is then defined by attaching the higher strata $G \setminus G_r$ to $Y$ to form a 2-complex $X$ homotopy equivalent to $G$, and extending $h \from Y \to Y$ appropriately to a homotopy equivalence of $X$.

Geometricity of \eg\ strata plays an explicit role in the proof of the Tits alternative for $\Out(F_n)$ in \BookOne, with weak geometric models often playing an implicit supporting role. In this paper we give weak geometric models and (full) geometric models a central role, developing their theory in some depth, and obtaining applications needed for later in this series. 

\paragraph{Geometricity as an invariant of laminations.} The decomposition theory for $\Out(F_n)$ developed on \BookOne\ associates to each $\phi \in \Out(F_n)$ a finite set $\L(\phi)$ of ``attracting laminations'' (see Sections~\ref{SectionAttractingLams} and~\ref{SectionLams} for a review). After passing to a power of $\phi$ which is ``rotationless'', and using the very nicest kind of topological representative $f \from G \to G$ called a ``completely split relative train track map'' or just a \ct, there is a natural correspondence between the set $\L(\phi)$ and the set of \eg\ strata of $f$, where $\Lambda \in \L(\phi)$ corresponds to $H_r$ if and only if iterating the edges of $H_r$ produces the leaves of $\Lambda$ in the limit. The elements of the $\L(\phi)$ can be distinguished by their ``free factor supports'', and the elements of the two sets $\L(\phi)$ and $\L(\phi^\inv)$ can be put in bijective correspondence where $\Lambda^+ \in \L(\phi)$ corresponds to $\Lambda^- \in \L(\phi^\inv)$ if and only if $\Lambda^+,\Lambda^-$ have the same free factor supports, in which case we say that this pair forms a \emph{dual lamination pair}. 

Although geometricity is defined in \BookOne\ and here in Definition~\ref{DefGeometricStratum} as an invariant of \eg\ strata of relative train tracks, the following result says that it is actually an invariant of laminations, and furthermore it is an invariant of laminations pairs.

\begin{proposition*}(Proposition~\ref{PropGeomEquiv}) For any $\phi \in \Out(F_n)$ and $\Lambda^+ \in \L(\phi)$, if $\Lambda^+$ corresponds to a geometric \eg\ stratum in some \ct\ representing some rotationless power of $\phi$, then $\Lambda^+$ corresponds to a geometric \eg\ stratum in \emph{every} \ct\ representing \emph{every} rotationless power of $\phi$. Furthermore, if this holds, and if $\Lambda^- \in \L(\phi^\inv)$ is dual to $\Lambda^+$, then $\Lambda^-$ corresponds to a geometric \eg\ stratum of every \ct\ representing every rotationless power of $\phi^\inv$.
\end{proposition*}

The full statement of this proposition gives an explicit property, expressed solely in terms of $\phi$ and $\Lambda$ without regard to a representative \ct, which nonetheless characterizes geometricity of the \eg\ stratum associated to $\Lambda$ in any \ct\ representing any rotationless power of $\phi$. This property says that there is a finite $\phi$-invariant set of conjugacy classes in $F_n$ whose free factor support equals the free factor support of~$\Lambda$.

\paragraph{Complementarity properties of geometric models.} Consider a closed surface $S$ decomposed as a union of two compact connected subsurfaces $S = S_1 \union S_2$ such that $S_1 \intersect S_2 = \bdy S_1 = \bdy S_2 = c$ is connected and neither of $S_1,S_2$ is a disc. In this setting the subgroup of $\MCG(S)$ that stabilizes $S_1$ (up to isotopy) is the same as the subgroup that stabilizes $S_2$ and the subgroup that stabilizes $\C$ preserving transverse orientation. Furthermore, one can understand this subgroup rather explicitly: it is the quotient of direct sum of the relative mapping class groups $\MCG(S_1,\bdy S_1)$ and $\MCG(S_2,\bdy S_2)$, modulo identification of Dehn twists supported in collar neighborhoods of $c$ in $S_1$ and in $S_2$.

By contrast, such behavior fails in the setting of $\Out(F_n)$ when using a free factorization $F_n = A * B$ in place of the decomposition $S = S_1 \union S_2$:
the subgroup of $\Out(F_n)$ that preserves the conjugacy class of the free factor~$A$ is not equal to the subgroup that preserves the conjugacy class of $B$. One way to think of this is that there is no well-defined way to determine a complement of a free factor of $F_n$ amongst other free factors (up to conjugacy).

We study a version of complementarity which works in the context of geometric models. For purposes of brevity, in this introduction we describe this phenomenon only in the special case that $H_r$ is the top stratum of a \ct\ $f \from G \to G$ representing a rotationless $\phi \in \Out(F_n)$, in which case the full geometric model $X$ and the weak geometric model $Y$ are the same. In this special case, under the quotient map $G_{r-1} \union S \mapsto X$, the interior of the surface $S$ embeds as an open subset of $X$ whose complement is a finite graph $L$ which is the disjoint union of $G_{r-1}$ and a circle which is as a copy of the top boundary circle $\bdy_0 S$; we refer to the graph $L$ as the \emph{complementary graph} of the geometric model. One may reconstruct $X$ as the quotient obtained from the disjoint union $L \union S$ by gluing the lower boundary circles to $G_{r-1}$ as before and identifying the upper boundary circle $\bdy_0 S$ in $S$ with its copy in $L$. 

Like subgraphs of marked graphs, the inclusion of $L$ into $X$ is $\pi_1$-injective on each noncontractible component, and the image subgroups are each malnormal and are ``mutually'' malnormal (see Lemma~\ref{LemmaLImmersed}). The collection of conjugacy classes of these image subgroups forms what we call a ``subgroup system'' in $F_n \approx \pi_1(X)$, denoted $[\pi_1 L]$ (see Section~\ref{SectionSSAndFFS} for ``subgroup systems''). Also, the restriction to $S$ of the quotient map $L \union S \mapsto X$ is $\pi_1$-injective and the conjugacy class of the image of $\pi_1 S$ is a subgroup system in $F_n$ denoted $[\pi_1 S]$. In general the subgroup systems $[\pi_1 L]$ and $[\pi_1 S]$ do not behave as nicely as free factor systems, and in fact $[\pi_1 L]$ is \emph{never} a free factor system in $F_n$, although $[\pi_1 S]$ may or may not be a free factor system depending on the example.

The subgroup systems $[\pi_1 L]$ and $[\pi_1 S]$, and the subgroups of $\Out(F_n)$ that stabilize them, will play an important role in later applications in this series of papers. For example, in the present context where we have assumed $H_r$ to be the top stratum, the conjugacy classes of $F_n$ represented by circuits in $L$ are precisely the conjugacy classes that are not weakly attracted to the lamination $\Lambda^+$ under iteration of~$\phi$; this fact is a key observation underlying the Weak Attraction Theorem of \BookOne, and its generalization dropping the assumption on $H_r$ will play a role in our general study of weak attraction theory in \PartThree. Also, in \PartTwo\ we shall need to study nonrotationless elements of $\Out(F_n)$ which stabilize $[\pi_1 L]$, such as certain nonrotationless roots of $\phi$. Also, regarding $[\pi_1 S]$, an important step in the proof of Proposition~\ref{PropGeomEquiv} stated above is to prove that the free factor support of the lamination $\Lambda^+$ equals the free factor support of $[\pi_1 S]$. 

The following result, which is perhaps the most significant new result of this paper,  says roughly that if an outer automorphism preserves the complementary subgraph $L$ then it preserves the surface $S$. To state it, we use that the natural action of $\Aut(F_n)$ on subgroups induces natural actions of $\Out(F_n)$ on conjugacy classes of subgroups and on subgroup systems, with respect to which we may define the subgroup $\Stab(\A) \subgroup \Out(F_n)$ which stabilizes a subgroup system~$\A$.

\begin{proposition*}[Proposition \ref{PropVertToFree}, special case]
For any $\theta \in \Out(F_n)$, if $\theta[\pi_1 L]=[\pi_1 L]$ then $\theta[\pi_1 S]=[\pi_1 S]$,  inducing a well-defined outer automorphism $\theta \restrict \pi_1 S \in \Out(\pi_1 S)$. Furthermore, this outer automorphism preserves the set of conjugacy classes associated to the components of $\bdy S$ and so, by the Dehn-Nielsen-Baer theorem, $\theta \restrict \pi_1 S$ lives in the natural $\MCG(S)$ subgroup of $\Out(\pi_1 S)$. One obtains by this process a homomorphism $\Stab[\pi_1 L] \mapsto \MCG(S)$.
\end{proposition*}

This proposition is based on Lemma~\ref{LemmaGeomModelHE} we can be regarded as a topological version of the proposition, stated in terms of self homotopy equivalences of a geometric model. In Lemma~\ref{LemmaFreeDefRetr} we prove an important addendum to Lemma~\ref{LemmaGeomModelHE}, which will be used in \PartTwo\ when we analyze elements of $\Out(F_n)$ that act trivially on homology with $\Z/3$ coefficients. In the proof of Proposition~\ref{PropVertToFree}, the application of the Dehn-Nielsen-Baer theorem comes after showing that $\theta \restrict \pi_1 S$ preserves the conjugacy classes of $\pi_1 S$ associated to the boundary circles of $S$. Although the gluing of lower boundary circles of $S$ to $G_{r-1}$ will in general not be 1--1, it may happen that certain components $c \subset \bdy S$ remain as ``free boundary circles'' in the geometric model, by which we mean that some half-open annulus neighborhood of $c$ in $S$ embeds with image an open subset of the geometric model. For example in the case where $H_r$ is a top stratum: the top boundary circe $\bdy_0 S$ is always free; and a lower boundary circle is free if it is identified homeomorphically to a component of $G_{r-1}$ and has no point identified with any point on another lower boundary circle. Lemma~\ref{LemmaFreeDefRetr} says that $\theta \restrict \pi_1 S \in \MCG(S)$ preserves the conjugacy classes of $\pi_1 S$ associated to the subset of \emph{free} boundary circles. 

\paragraph{Vertex group systems.} As mentioned above, the subgroup system $[\pi_1 L]$ associated to complementary subgroup $L$ of a geometric model $X$ is never be a free factor system, as we show in Lemma~\ref{LemmaScaffoldFFS}~\pref{ItemFFSLowerBdys}. In the special case of a geometric outer automorphism, where $X=S$ is a compact surface and $L=\bdy S \ne \emptyset$, this follows from a result due to Stallings \cite{Stallings:transversality} which shows that for any free factorization $\pi_1 S = A_1 * \cdots * A_K$ there exists a component of $\bdy S$ which is not conjugate into any factor $A_k$, and our argument is an adaptation of Stallings' proof. 

Subgroup systems in $F_n$ can be ill-behaved in general. We shall show that the subgroup system $[\pi_1 L]$ has an important special property. Given a minimal action of $F_n$ on an $\reals$-tree $T$ such that the stabilizer of each arc is trivial, there are only finitely many $F_n$-orbits of points $x \in T$ such that the stabilizer of $x$ is nontrivial \cite{GaboriauLevitt:rank}. The collection of nontrivial point stabilizers of $T$ is closed under conjugacy, and the corresponding subgroup system is called the \emph{vertex group system} of~$T$. Vertex group systems, while not quite as well-behaved as free factor systems, nonetheless satisfy some useful properties; in Proposition~\ref{PropVDCC} we give a proof that we learned from Mark Feighn of a descending chain condition on vertex group systems.

Our main result relating geometric models and vertex group systems is a more general version of the following proposition, and it will be used for various purposes in later parts of this series.

\begin{proposition*}(Proposition~\ref{PropGeomVertGrSys}, special case)
For any geometric model $X$, the subgroup system $[\pi_1 L]$ is the vertex group system of some minimal action of $F_n$ on an $\reals$-tree with trivial arc stabilizers.
\end{proposition*}

\vfill\break

\noindent\textbf{Description of the contents.} Section~\ref{SectionPrelim} is a detailed but terse review of the decomposition theory for individual elements of $\Out(F_n)$ as developed in \BookOne\ and in \recognition. This section is designed as a central reference, and a reader who is conversant in the language of relative train track and outer automorphisms as presented in \BookOne\ and \recognition\ may only need to refer to Section~\ref{SectionPrelim} occasionally. 

Section~\ref{SectionGeometric} contains the results about geometricity of \eg\ strata and of attracting laminations and about the geometric models.The dichotomy between geometric and nongeometric \eg\ strata is used in the statements and proofs of many results in \BookOne, and the same is true in this series; see for example in Theorem~C of the introductory paper of the series \Intro. 

Section~\ref{SectionVertexGroups} is a study of vertex group systems and their relation to geometric strata.

\setcounter{tocdepth}{3}
\tableofcontents

\bigskip

\section{Preliminaries: Decomposition of outer automorphisms of free groups}
\label{SectionPrelim}

In this section we give a self-contained account of basic results about $\Out(F_n)$ needed throughout this series of papers, citing outside sources, particularly \BookOne\ and \recognition, for the reader to verify the correctness of statements. With that intent, we try to be as succinct as possible. We shall state numerous ``Facts'' which are either citations from the literature, or are quickly proved using results from the literature. We state also some ``Lemmas'' whose proofs require a bit more care. 

A reader familiar with the results of \BookOne\ and \recognition\ may find it convenient to use this section more as a central reference for this whole series of papers, rather than as a prerequisite for reading the series. 

\subsection{$F_n$ and its subgroups, marked graphs, and lines.}
\label{SectionTheBasics}

For $n \ge 2$ let $R_n$ denote the rose with $n$ directed edges labelled $E_1,\ldots,E_n$. Let $F_n$ denote the rank $n$ free group $\pi_1(R_n)$, with $E_1,\ldots,E_n$ as a free basis, and fix a lifted base point in the universal cover $\wt R_n$ thereby identifying $F_n$ as the group of deck transformations, and so $\wt R_n$ is the Cayley graph of $F_n$ with respect to that free basis. 

Let $\Aut(F_n)$ be the automorphism group of $F_n$, let $\Inn(F_n)$ be the group of inner automorphisms of $F_n$, and let $\Out(F_n) = \Aut(F_n) / \Inn(F_n)$ be the outer automorphism group. The action of $\Aut(F_n)$ on elements and subgroups of $F_n$ induces actions of $\Out(F_n)$ on conjugacy classes of elements and of subgroups. 

We use the notation $[\cdot]$ to denote conjugacy classes in $F_n$ of elements and of subgroups, i.e.\ $F_n$-orbits under the action of $F_n$ on itself by inner automorphisms. We use $\C=\C(F_n)$ to denote the set of conjugacy classes of nontrivial elements of $F_n$.

\subsubsection{The geometry of $F_n$ and its subgroups.} 
\label{SectionFnGeometry}
Recall that in a finitely generated group $G$, a finitely generated subgroup $H \subgroup G$ is \emph{undistorted} if the inclusion map $H \inject G$ is a quasi-isometric embedding with respect to the word metrics $d_H$, $d_G$.

\begin{fact}\label{FactFiniteRankSubgroup}
For any finite rank subgroup $A \subgroup F$ the following hold:
\begin{enumerate}
\item\label{ItemFiniteRankInNormalizer}
$A$ has finite index in its own normalizer.
\item\label{ItemFiniteRankUndistorted}
$A$ is undistorted in $F$.
\end{enumerate}
\end{fact}

\begin{proof}[Proof]
Under the action $F_n \act \wt R_n$, the restricted action of $A$ on its minimal subtree $T$ is cocompact. It follows that $A$ has finite index in the subgroup of $F_n$ that stabilizes $T$, but that subgroup is the normalizer of $A$, proving~\pref{ItemFiniteRankInNormalizer}. The injection $T \inject \wt R_n$ is an isometric embedding, and for any $x \in T$ the orbit map $A \to T$ defined by $a \mapsto a \cdot x$ is a quasi-isometric embedding, proving~\pref{ItemFiniteRankUndistorted}.
\end{proof}

This fact has the following consequences which we use henceforth without comment, regarding the natural identification between the Gromov boundary $\bdy F_n$ and the Cantor set of ends of the tree~$\wt R_n$. For any finite rank subgroup $A \subgroup F_n$, it follows by item~\pref{ItemFiniteRankUndistorted} that the Gromov boundary $\bdy A$ is naturally identified with the closure in $\bdy F_n$ of the set of fixed point pairs for the action of nontrivial elements of $A$; in particular when $A=\<\gamma\>$ is infinite cyclic, $\bdy A$ is a 2-point subset of~$\bdy F_n$, namely the fixed point set for the action of $\gamma$. Also, under the action of $F_n$ on $\wt R_n$, or under any properly discontinuous cocompact action of $F_n$ on a simplicial tree $T$ whose boundary $\bdy T$ has been identified with $\bdy F_n$, it follows by item~\pref{ItemFiniteRankInNormalizer} that the action of $A$ on the unique minimal $A$-invariant subtree $T^A$ of $T$ is cocompact; the Gromov boundary of $T^A$ is identified with $\bdy A$, and $T^A$ equals the \emph{convex hull} of $\bdy A$ denoted $\Hull(\bdy A)$, which is just the union of all lines in $T$ whose endpoints are in $\bdy A$.

The following is a standard fact about undistorted torsion free subgroups of word hyperbolic groups, but in our context it has a rather low technology proof.

\begin{fact}\label{FactBoundaries}  For any finitely generated subgroups $A_1, A_2 \subgroup F_n$, their intersection $A_1 \intersect A_2$ is the trivial subgroup if and only if $\bdy A_1 \intersect \bdy A_2 = \emptyset$. More generally, $\bdy(A_1 \intersect A_2) = \bdy A_1 \intersect \bdy A_2$.
\end{fact}

\begin{proof} Let $p \in \wt R_n$ be the base vertex of the Cayley graph. Let $\norm{\cdot}$ denote the word norm in $F_n$ and let $d(\cdot,\cdot)$ denote the word metric, so $d(g,h) = d(g(p),h(p)) = \norm{h^\inv g}$. The inclusion $\bdy(A_1 \intersect A_2)  \subset \bdy A_1 \intersect \bdy A_2$ being obvious, consider $\xi \in \partial A_1 \cap \partial A_2$. In $\wt R_n$ the intersection $\Hull(\bdy A_1) \intersect \Hull(\bdy A_2)$ contains a ray $[x,\xi)$. Choose a base vertex $p_i \in \Hull(\bdy A_i)$, $i=1,2$ and $g_i \in F_n$ so that $g_i(p)=p_i$. By cocompactness of the action of $A_i$ on $\Hull(\bdy A_i)$, $i=1,2$, there exists~$C \ge 0$, vertex sequences $x_k \in \Hull(\bdy A_1)$ and $y_k \in \Hull(\bdy A_2)$, and sequences $h_{1k} \in A_1$, $h_{2k} \in A_2$, such that $h_{1k}(p_1)=x_k$, $h_{2k}(p_2)=y_k$, $d(x_k,y_k) \le C$ in~$\wt R_n$, $d(x_k,[x,\xi)) , d(y_k,[x,\xi)) \le C$, and as $k \to +\infinity$ both of the sequences $x_k,y_k$ approach $\xi$ in the compactified Cayley graph $\wt R_n \disjunion \bdy F_n$. It follows that 
\begin{align*}
\norm{h_{1k}^\inv h_{2k}^{\vphantom{\inv}}} &= d(h_{1k}(p),h_{2k}(p)) \\
 & \le d(h_{1k}(p),h_{1k}(p_1)) + d(h_{1k}(p_1), h_{2k}(p_2)) + d(h_{2k}(p_2), h_{2k}(p)) \\
 &= d(p,p_1) +  d(x_k,y_k) + d(p_2,p)
\end{align*}
is bounded independent of $k$. Since $F_n$ has only finitely many elements with a given bound on word length, passing to a subsequence we may assume that $h_{1k}^\inv h_{2k}^{\vphantom{\inv}}$ is constant. In particular $h_{1,0}^\inv h_{2,0}^{\vphantom{\inv}} = h_{1k}^\inv h_{2k}^{\vphantom{\inv}}$ and so $h_{1k}^{\vphantom{\inv}} h_{1,0}^\inv = h_{2k}^{\vphantom{\inv}} h_{2,0}^\inv = f_k \in A_1 \cap A_2$ for all $k$. The distance between $x_k = h_{1k}(p_1)$ and $f_k(p_1)=h_{1k}(h_{1,0}(p_1))$ is uniformly bounded for all $k$, and so $f_k(p_1)$ converges to $\xi$ in $\wt R_n \union \bdy F_n$. Truncating an initial sequence we may assume $f_k$ is nontrivial for all $k$. Letting $\eta_k^+ \in  \bdy(A_1 \intersect A_2)$ be the attracting endpoint of the axis of $f_k$, it follows that $\eta_k^+ \to \xi$ as $k \to +\infinity$. Since $\bdy(A_1 \intersect A_2)$ is closed it follows that $\xi \in \bdy(A_1 \intersect A_2)$.
\end{proof}

\subsubsection{Subgroup systems and free factor systems.} 
\label{SectionSSAndFFS}
We recall from \BookOne\ Section 2.6 the concepts of free factor systems. We also introduce a generalized notion called ``subgroup systems'', to lay the ground for the ``vertex group systems'' introduced in Section~\ref{SectionVertexGroups} and for the ``fixed subgroup systems'' that arise in the study of rotationless outer automorphisms and that are used in \PartTwo.

Define a \emph{subgroup system} in $F_n$ to be a finite set $\K$ each of whose elements is a conjugacy class of finite rank subgroups of $F_n$. The individual elements of the set $\K$ are called its \emph{components}. Note that a subgroup system $\K$ is completely determined by the set of subgroups $K \subgroup F_n$ such that $[K] \in \K$. With this in mind, we often abuse notation by writing $K \in \K$ when we really mean $[K] \in \K$. There is a partial order on subgroup systems denoted $\K_1 \sqsubset \K_2$ and defined by requiring that for each subgroup $K \in \K_1$ there exists a subgroup $K' \in \K_2$ such that $K \subgroup K'$. We refer to the relation $\K_1 \sqsubset \K_2$ by saying for example that ``$\K_1$ is contained in $\K_2$'' or that ``$\K_2$ carries $\K_1$''.\index{carrying!of subgroup systems} In particular we say that \emph{$\K_1$ is properly contained in $\K_2$} if $\K_1 \sqsubset \K_2$ and $\K_1 \ne \K_2$. The action of $\Out(F_n)$ on the set of conjugacy classes of subgroups induces an action on the set of subgroup systems. The \emph{stabilizer} of a subgroup system $\K$ is the subgroup of all $\theta \in \Out(F_n)$ such that $\theta(\K)=\K$.

A \emph{free factor system} in $F_n$ is a (possibly empty) subgroup system of the form $\F = \{[A_1],\ldots,[A_k]\}$ such that there exists a free factorization $F_n = A_1 * \ldots * A_k * B$ where each $A_i$ is nontrivial; the subgroup $B$ may or may not be trivial. The partial order $\sqsubset$ defined generally on subgroup systems is consistent with the partial order on free factor systems given in the introduction. The action of $\Out(F_n)$ on the set of subgroup systems restricts to an action on the set of free factor systems. 

We often take the liberty of referring to a free factor system $\{[A]\}$ with a single component as a \emph{free factor}, rather than the lengthier but more appropriate locution ``free factor conjugacy class''. This should not cause any confusion since each element of a free factor system is a conjugacy class of subgroups of $F_n$ as opposed to a subgroup itself.

The following fact and definition taken from Section~2 of \BookOne\ is essentially a refinement of the Kurosh Subgroup Theorem:

\begin{fact}[Definition of the meet of free factor systems]
\label{FactGrushko}
Every collection $\{\F_i\}_{i \in I}$ of free factor systems has a well-defined \emph{meet}\index{meet} $\meet\{\F_i\}$, which is the unique maximal free factor system $\F$ such that $\F \sqsubset \F_i$ for each~$i \in I$. Furthermore, for any free factor $A \subgroup F_n$ we have $[A] \in \meet\{\F_i\}$ if and only if there exists an indexed collection of subgroups $\{A_i\}_{i \in I}$ such that $[A_i] \in \F_i$ for each $i$ and $A = \cap_{i \in I} A_i$. In particular it follows that the relation $\sqsubset$ satisfies the descending chain condition.
\end{fact}

\subparagraph{Malnormal subgroup systems.} Recall that a subgroup $K \subgroup F_n$ is \emph{malnormal} if $K^x \cap K$ is trivial for all $x \in F_n - K$; note in particular that a malnormal subgroup is its own normalizer.

More generally, define a subgroup system $\K$ to be \emph{malnormal} if each of its subgroups $K \in \K$ is malnormal and for all $K,K' \in \K$, if $K \intersect K'$ is nontrivial then $K=K'$. Every free factor $A \in F_n$ is malnormal, and every free factor system $\F$ is malnormal; this can be recovered by applying Fact~\ref{FactGrushko} to the meet of the collection~$\{\A\}$. By Fact~\ref{FactBoundaries}, if a subgroup system $\K$ is malnormal then for any two subgroups $K,K' \in \K$, if $K \ne K'$ then $\bdy K \intersect \bdy K' = \emptyset$.

\subsubsection{Restrictions of outer automorphisms.} 
\label{SectionRestrictedOuts}
Consider $\theta \in \Out(F_n)$ and a finite rank subgroup $A \subgroup F_n$ whose conjugacy class $[A]$ is fixed by the action of $\theta$. We may choose $\Theta \in \Aut(F_n)$ representing $\theta$ so that $\Theta(A)=A$, and we obtain $\Theta \restrict A \in \Aut(A)$. If~the outer automorphism class of $\Theta \restrict A$ is well-defined in the group $\Out(A)$ independent of the choice of $\Theta$, then we denote this by $\theta \restrict A \in \Out(A)$ and we say that \emph{$\theta \restrict A$ is well-defined}. This need not always happen, but the following gives a sufficient condition:

\begin{fact}\label{FactMalnormalRestriction} 
Given a finite rank subgroup $A \subgroup F_n$ which is its own normalizer in $F_n$, for any $\theta \in \Out(F_n)$ such that $\theta[A]=[A]$, the restriction $\theta \restrict A \in \Out(A)$ is well-defined, and the map $\theta \mapsto \theta \restrict A$ defines a homomorphism $\Stab[A] \mapsto \Out(A)$. If moreover $A$ is malnormal then the conjugacy class in $A$ of an element or subgroup of $A$ is fixed by $\theta \restrict A$ if and only if the conjugacy class in $F_n$ determined by that element or subgroup is fixed by $\theta$. 
\end{fact}

\begin{proof} Let $\Theta' \in \Aut(F_n)$ be another representative of $\theta$ such that $\Theta'(A)=A$. Choose $d \in F_n$ with associated inner automorphism $\inner_d(\gamma)=d\gamma d^\inv$ so that $\Theta' = \inner_d \composed \Theta$. It follows that $A = \Theta'(A) = \inner_d(\Theta(A)) = \inner_d(A)$, and so $d$ normalizes $A$, implying that $d \in A$. The equation $\Theta'=\inner_d \composed \Theta$ therefore implies that $\Theta,\Theta'$ represent the same outer automorphism of~$A$. Given $\theta_1,\theta_2 \in \Stab[A]$, and choosing representatives $\Theta_1,\Theta_2 \in \Aut(F_n)$ each preserving the subgroup $A$, it follows that $\Theta_1 \cdot \Theta_2 \in \Aut(F_n)$ also preserves $A$; since $\Theta_1 \cdot \Theta_2$ represents the element $\theta_1 \cdot \theta_2$ in $\Stab[A]$, and since it also represents the element $(\theta_1 \restrict A) \cdot (\theta_2 \restrict A)$ in $\Out(A)$, it follows that the map $\Stab[A] \mapsto \Out(A)$ takes $\theta_1 \cdot \theta_2$ to $(\theta_1 \restrict A) \cdot (\theta_2 \restrict A)$ and is therefore a homomorphism.

Malnormality of $A$ implies that the injection $A \inject F_n$ induces injections from the sets of conjugacy classes \emph{in $A$} of elements and subgroups to the sets of conjugacy classes \emph{in $F_n$} of elements and subgroups, and these injections are natural with respect to the actions of $\theta \restrict A$ on conjugacy classes in $A$ and of $\theta$ on conjugacy classes in~$F_n$. The final ``If moreover\ldots'' sentence follows.
\end{proof}

\subsubsection{Marked graphs, paths, and circuits.} 
\label{SectionGraphsPathsCircuits}
The subset theoretic difference operator, for any set $A$ and subset $B$, is denoted throughout as $A-B = \{x \in A \suchthat x \not\in B\}$. We also need a subgraph difference operator, and more generally a subcomplex difference operator. Given a CW complex $G$ and subcomplex $H \subset G$, let $G \setminus H$ denote the subcomplex of $G$ which is the closure of $G-H$, alternatively $G \setminus H$ is the union of those closed cells whose interiors are in $G-H$.

A \emph{marked graph} is a graph $G$ having no vertices of valence~1 equipped with a path metric and a homotopy equivalence $\rho \from G \to R_n$. The fundamental group $\pi_1(G)$ is therefore identified with $F_n$ up to inner automorphism. There is a natural identification between the set of circuits in $G$, the set of conjugacy classes of $\pi_1(G)$, and the set~$\C$. Each marked graph $G$ determines a unique ``length function'' in $\reals^\C$ which assigns to an element of $\C$ the length of the corresponding circuit of~$G$. Two marked graphs $\rho \from G \to R_n$, $\rho' \from G' \to R'_n$ are equivalent if there is a homeomorphism $\theta \from G \to G'$ that multiplies path metric by a constant such that $\rho' \theta$ is homotopic to $\rho$, and this happens if and only if the two corresponding elements of $P\reals^\C$ are equal.

Henceforth for any marked graph $G$ the group of deck transformations of the universal covering space $\wt G$ will be identified by $F_n$, under suitable choice of base points (different choice of base points only changes the action by an inner automorphism of $F_n$). As in Section~\ref{SectionFnGeometry} there is an induced $F_n$-equivariant homeomorphism between $\bdy F_n$ and the Cantor set of ends $\bdy\wt G$.

For any marked graph $G$ and any subgraph $H \subset G$, the fundamental groups of the noncontractible components of $H$ determine a free factor system in $F_n$ that (by a moderate abuse of notation) we denote~$[\pi_1 H]$ and that has one component for each noncontractible component of the subgraph $H$. A~subgraph of $G$ is called a \emph{core graph}\index{core graph} if it has no valence~1 vertices. Every subgraph $H \subset G$ has a unique core subgraph $\core(H) \subset H$ which is a deformation retract of the union of the noncontractible components of $H$, implying that \hbox{$[\pi_1 \core(H)] = [\pi_1 H]$}. Conversely any free factor system $\F$ can be realized as $[\pi_1 H]$ for some nontrivial core subgraph $H$ of some marked graph $G$.

A \emph{path} in a graph $K$ is a locally injective, continuous map $\gamma \from J \to K$ defined on a closed, connected, nonempty subset $J \subset \reals$ such that either $\gamma$ is \emph{trivial} meaning that $J$ is a point, or $\gamma$ can be expressed by concatenating oriented edges at vertices. Such a concatenation expression of $\gamma$, called an \emph{edge path} representing $\gamma$, can be written in one of the following forms: a \emph{finite path} is a finite concatenation $E_0 E_1 \ldots E_n$ for $n \ge 0$; a \emph{positive ray} is a singly infinite concatenation $E_0 E_1 E_2 \ldots$; a \emph{negative ray} is a singly infinite concatenation $\ldots E_{-2} E_{-1} E_0$; and a \emph{line} is a bi-infinite concatenation $\ldots E_{-2} E_{-1} E_0 E_1 E_2 \ldots$. In each case we have $E_i \ne \overline E_{i+1}$ by the locally injective requirement. For any path $\gamma$ let $\bar\gamma$ or $\gamma^\inv$ denote $\gamma$ with its orientation reversed; note in particular that the reversal of a positive ray is a negative ray. We regard two paths as equivalent if they have the same concatenation expression up to translation of the index in the case of a line, and up to inversion in the case of a finite path or a line, or what is the same if they differ up to a homeomorphism of their domains. 

Given a continuous function $\gamma \from J \to K$ where $J,K$ are as above, and where $\gamma$ takes each endpoint of $J$ to a vertex of $K$, under certain conditions $\gamma$ can be \emph{tightened} to a path denoted $[\gamma]$. This is always possible when $J$ is compact, because $\gamma$ is homotopic rel endpoints either to a nontrivial path $[\gamma]$ or to a constant path in which case $[\gamma]$ denotes the corresponding trivial path. If $J$ is noncompact then the straightened path $[\gamma]$ is defined under the additional assumption that a lift to the universal cover $\ti\gamma \from J \to \wt K$ is a quasi-isometric embedding, in which case $[\gamma]$ is the unique path which lifts to a path in $\wt K$ that is properly homotopic to $\ti\gamma$, where ``properness'' of the homotopy means here that the inverse image of a bounded diameter set is compact; equivalently, $[\gamma]$ is the unique path in $K$ such that $\gamma$ and $[\gamma]$ have lifts to $\wt K$ with the same finite endpoints and the same infinite ends.

A \emph{circuit} in $K$ is a continuous, locally injective map of an oriented circle into~$K$.  As with paths, we will not distinguish between two circuits that differ only by a homeomorphism of domains. Any circuit can be written as a concatenation of edges which is unique up to cyclic permutation and inversion. Every continuous, homotopically nontrivial function $c \from S^1 \to K$ may be tightened, meaning in this case that $c$ is freely homotopic to a unique circuit denoted $[c]$. 

A circuit represented by a local injection $\gamma \from S^1 \to K$ is \emph{root-free}\index{root-free} if $\gamma$ does not factor through a covering map $S^1 \mapsto S^1$ of positive degree. We also use the root-free terminology for a closed path $\gamma \from J \to K$ to mean that $\gamma$ is not a concatenation of more than one copy of the same closed path.

A path or circuit \emph{crosses} or \emph{contains} an edge $E$ if $E$ or $\overline E$ occurs in its edge path expression.

\subsubsection{Lines and rays.} 
\label{SectionLineDefs}

\textbf{Lines.} Given any finite graph $K$ let $\wh \B(K)$ be the compact space of circuits and paths in $K$, where this space is given the \emph{weak topology} with one basis element $\wh V(K,\alpha)$ for each finite path $\alpha$ in $K$, consisting of all paths and circuits having $\alpha$ as a subpath. Let $\B(K) \subset \wh\B(K)$ be the compact subspace consisting of all lines in $K$, with basis elements $V(K,\alpha) = \wh V(K,\alpha) \intersect \B(K)$.

Following \BookOne\ we next turn to a coordinate free description of lines. We give the definitions in the context of an arbitrary free group $A$ of finite rank~$\ge 2$, in order to apply it to subgroups of~$F_n$. Let $\wt \B(A)$ denote the set of two-element subsets of $\bdy A$, equivalently
$$\wt\B(A) = (\bdy A \cross \bdy A - \Delta) / (\Z/2)
$$
where $\Delta$ is the diagonal and $\Z/2$ acts by interchanging factors. We put the \emph{weak topology} on $\wt\B(A)$, induced from the usual Cantor topology on $\bdy A$. The group $A$ acts on $\wt\B(A)$ with compact but non-Hausdorff quotient space $\B(A) = \wt\B(A) / A$; we also refer to this quotient topology as the \emph{weak topology}. A \emph{line of $A$} is, by definition, an element of $\B(A)$. A \emph{lift} of a line $\gamma \in \B(A)$ is an element $\ti\gamma \in \wt\B(A)$ that projects to $\gamma$, and the two elements of $\partial A$ that determine $\ti\gamma$ are called its \emph{endpoints}. 

Every nontrivial conjugacy class $[b]$ in $A$ determines a well-defined line $\gamma \in \B(A)$ called the \emph{axis}\index{axis!of a conjugacy class} of $[b]$, characterized by saying that the lifts of the axis are precisely those $\ti\gamma \in \wt B(A)$ which are fixed by representatives of~$[b]$. Note that two nontrivial conjugacy classes $[b]$, $[b']$ have the same axis if and only if there exists $a \in A$ and integers $i,j \ne 0$ such that $a^i,a^j$ are conjugate to $b,b'$ respectively. The set of axes therefore forms a subset of $\B(A)$ in natural 1--1 correspondence with the set of unordered pairs $\{[b],[b^\inv]\}$ where $b \in A$ is \emph{root free} meaning that $b = c^k$ implies $k = \pm 1$.

The natural action of $\Aut(A)$ on $\bdy A$ induces an action of $\Aut(A)$ on $\wt\B(A)$ and hence an action of $\Out(A)$ on $\B(A)$.

We also use $\wt\B$ and $\B$ as shorthand for the spaces $\wt\B(F_n)$ and $\B(F_n)$ associated to the free group $F_n$. 

Given any marked graph $G$, its space of lines $\B(G)$ is naturally homeomorphic to $\B(F_n)$, via the natural identification between the $\bdy F_n$ and the space of ends of the universal covering~$\wt G$. Given $\gamma \in \B(F_n)$ the corresponding element of $\B(G)$ is called the \emph{realization} of $\gamma$ in $G$. Keeping this homeomorphism in mind, we use the term ``line'' ambiguously, to refer either to an element of $\B(F_n)$ or an element of $\B(G)$ in a context where a marked graph $G$ is under discussion. Sometimes to allay confusion an element of $\B(F_n)$ is called an \emph{abstract~line}. 

\smallskip\textbf{Notation and terminology.} We usually abbreviate $\wt\B(F_n)$ to $\wt\B$ and $\B(F_n)$ to $\B$, although for emphasis we sometimes do not abbreviate. We use the adjective ``weak'' to refer to concepts associated to the weak topology, although in contexts where this is clear we sometimes drop the adjective. In particular, although $\wh\B(G)$ is not Hausdorff and limits of sequences are not unique, we use the phrase ``weak limit'' of a sequence to refer to any point to which that sequence weakly converges. As a special case, in any context where $\phi \in \Out(F_n)$ or $f \from G \to G$ are given, if $\alpha,\beta \in \B$ and $\phi^k(\alpha)$ converges weakly to $\beta$ as $k \to +\infinity$, or if $\alpha,\beta \in \wh\B(G)$ and $f^k_\#(\alpha)$ converges weakly to~$\beta$, then we say that \emph{$\alpha$ is weakly attracted to $\beta$}\index{weakly attracted} (under iteration by $\phi$ or by $f_\#$).

\medskip\noindent
\textbf{Rays.} A \emph{ray of $F_n$}\index{ray} is by definition an element of the orbit set $\bdy F_n / F_n$ (we will have no occasion to use the quotient topology on this set). Rays of $F_n$ may be realized in a marked graph $G$ as follows. Two (positive) rays in $G$ are \emph{asymptotic} if they have equal subrays, and this is an equivalence relation on the set of rays in $G$. The set of asymptotic equivalence classes of rays $\rho$ in $G$ is in natural bijection with $\bdy F_n / F_n$ where $\rho$ in $G$ corresponds to end $\xi \in \bdy F_n / F_n$ if there is a lift $\ti\rho \subset \wt G$ of $\rho$, and a lift $\ti\xi \in \bdy F_n$ of $\xi$, such that $\ti\rho$ converges to $\ti\xi$ in the Gromov compactification of $\wt G$. A ray $\rho$ in $G$ corresponding to an abstract ray $\xi$ is said to be a \emph{realization of $\xi$ in $G$}. Again we sometimes use the term \emph{abstract ray} to distinguish a ray of $F_n$ from a ray in $G$.

The \emph{weak accumulation set} of a ray $\rho$ in $G$ is the set of lines $\ell \in \B(G)$ which are elements of the weak closure of $\rho$ in the space $\wh\B(G)$; equivalently, every finite subpath of $\ell$ occurs infinitely often as a subpath of $\rho$. Closely related is the \emph{accumulation set} of an abstract ray $\xi \in \bdy F_n / F_n$, defined as the set of lines $\ell \in \B$ which are in the weak closure of every line $\mu \in \B$ that has a lift $\ti\mu$ with an endpoint~$\ti\xi$ projecting to~$\xi$. We may also speak of the accumulation set of any element of $\bdy F_n$, well-defined to be the accumulation set of its projection in $\bdy F_n / F_n$. The accumulation set of a ray $\rho$ in $G$ is well-defined in its asymptotic equivalence class and is equal to the accumulation set of the corresponding abstract ray $\xi \in \bdy F_n / F_n$.

\subsubsection{The path maps $f_\#$ and $f_{\#\#}$, and bounded cancellation.}
 \label{def:DoubleSharp} 
Consider a $\pi_1$-injective map $f \from K \to G$ between two finite graphs which we fix for the moment. 

There is an induced continuous map $f_\# \from \wh \B(K) \to \wh\B(G)$ defined for each path or circuit $\gamma$ in $K$ by $f_\#(\gamma) = [f(\gamma)]$, that is, $f_\#(\gamma)$ is obtained from the $f$-image of $\gamma$ by straightening. If $f$ is a homotopy equivalence then the restricted map $f_\# \from \B(K) \to \B(G)$ is a homeomorphism. If $K,G$ are both marked graphs then, with respect to the identifications $\B(K) \approx \B \approx \B(G)$, the homeomorphism $\B \approx \B(K) \xrightarrow{f_\#} \B(G) \approx \B$ equals the self-homeomorphism of $\B$ induced by the outer automorphism associated to $f$, and in particular if $f$ preserves marking then this self-homeomorphism is the identity. 

Also define the map $f_{\#\#} \from \wh\B(K) \to \wh\B(G)$ as follows. Intuitively $f_{\#\#}(\gamma)$ is the largest common subpath of all paths $f_\#(\delta)$ as $\delta$ varies over all paths containing $\gamma$ as a subpath. To be precise, choose a lifted path $\ti\gamma$ and a lifted map $\ti f \from \wt K \to \wt G$, define the path $\ti f_{\#\#}(\ti\gamma) \subset \wt G$ to be the intersection of all paths $\ti f_\#(\ti\delta)$ as $\ti\delta$ ranges over all paths in $\wt K$ that contain $\ti\gamma$ as a subpath, and define $f_{\#\#}(\gamma)$ to be the projection of $\ti f_{\#\#}(\ti\gamma)$, which is well-defined independent of the choice of $\ti\gamma$. \emph{It may happen that $f_{\#\#}(\gamma)$ is empty,} although by Lemma~\ref{LemmaDoubleSharpFacts}~\pref{ItemOneToTwoSharps} below this happens only when $f_\#(\gamma)$ is short.

\begin{fact}[Bounded Cancellation Lemma \cite{Cooper:automorphisms}, \BookZero]
\label{FactBCCSimplicial}
For any $\pi_1$-injective map $f \from K \to G$ of finite graphs there exists a constant $\BCC(f)$, called a \emph{bounded cancellation constant} for $f$, such that for any lift $\ti f \from \wt K \to \wt G$ to universal covers and any path $\gamma$ in $\wt G$, the path $[f_\#(\gamma)]$ is contained in the $\BCC(f)$ neighborhood of the image $f(\gamma)$. \qed
\end{fact}
 
The following lemma will be used extensively throughout Parts I--IV. Recall the notation $\wh V(\gamma,G)$ for the basis element of the weak topology on $\wh\B(G)$ associated to a finite path $\gamma$.


\begin{lemma}[The $\#\#$ Lemma]
\label{LemmaDoubleSharpFacts}
For any $\pi_1$-injective map $f \from K \to G$ of finite graphs, and for any finite path $\gamma$ in $K$, the following hold:
\begin{enumerate}
\item \label{ItemOneToTwoSharps}
The path $f_{\#\#}(\gamma)$ is obtained from the path $f_\#(\gamma)$ by truncating initial and terminal segments of length at most $\BCC(f)$. 
\item\label{ItemSharpNhd}
$f_\#(\wh V(\gamma;K)) \subset \wh V(f_{\#\#}(\gamma);G)$.
\end{enumerate}
Furthermore, for any $\pi_1$-injective maps $f \from G_1 \to G_2$ and $g \from G_2 \to G_3$ of finite graphs the following hold: 
\begin{enumeratecontinue}
\item \label{ItemDblSharpContain} If $\alpha$ is a subpath of $\beta$ in $G_1$ then $f_{\#\#}(\alpha)$ is a subpath of $f_{\#\#}(\beta)$ in $G_2$. 
\item \label{ItemDblSharpComp} If $\beta$ is a path in $G_1$ then $g_{\#\#}(f_{\#\#}(\beta))$ is a subpath of $(g \composed f)_{\#\#}(\beta)$.
\item \label{ItemDisjointCopies} $f_{\#\#}$ preserves order in the following sense: if $\sigma$ is a path in $G_1$ that decomposes into (possibly trivial) subpaths as $\sigma = \alpha_1 \, \beta_1 \, \alpha_2 \, \beta_2\, \ldots \, \beta_m \, \alpha_{m+1}$ then there is a decomposition 
$$f_\#(\sigma) = \gamma_1 \, f_{\#\#}(\beta_1) \, \gamma_2 \, f_{\#\#}(\beta_2) \, \ldots \, f_{\#\#}(\beta_m) \, \gamma_{m+1}
$$
for some (possibly trivial) subpaths $\gamma_i$. A similar statement holds for singly infinite and for bi-infinite decompositions of $\sigma$.
\item\label{ItemCircuit} If $\alpha$ is a subpath of a circuit $\sigma \subset G_1$  then $f_{\#\#}(\alpha)$ is a subpath of $f_\#(\sigma)$.
\end{enumeratecontinue}
 \end{lemma}

\begin{proof} Item~\pref{ItemOneToTwoSharps} follows from the Bounded Cancellation Lemma. Item~\pref{ItemSharpNhd} says that if $\delta$ is a path containing $\gamma$ as a subpath in $K$ then $f_\#(\delta)$ contains $f_{\#\#}(\gamma)$ as a subpath in $G$; this is just a rewriting of the definition of $f_{\#\#}$. Item~\pref{ItemDblSharpContain} is an immediate consequence of the definitions. To prove \pref{ItemDblSharpComp}, if $\beta$ is a subpath of~$\gamma$ then $f_{\#\#}(\beta)$ is a subpath of $f_{\#\#}(\gamma)$ which is a subpath of $f_\#(\gamma)$, and similarly $g_{\#\#}(f_{\#\#}(\beta))$ is a subpath of $g_{\#}(f_{\#}(\gamma)) = (g\composed f)_\#(\gamma)$; since this is true for all such~$\gamma$, \pref{ItemDblSharpComp} follows. 

For \pref{ItemDisjointCopies}, choose a lift $\ti \sigma = \ti \alpha_1 \ti \beta_1 \ti \alpha_2 \ti \beta_2\ldots \ti \beta_m \ti \alpha_{m+1}$. Given $i < j$, write $\ti \sigma = \ti \sigma_1 \ti \sigma_2$ where $\ti \sigma_1$ is the initial subpath of $\ti \sigma$ that terminates with $\ti \beta_i$. It is an immediate consequence of the definition of $f_{\#\#}$ that $\ti f_\#(\ti \sigma_1)$ contains $\ti f_{\#\#}(\ti \beta_i)$ and that $\ti f(\ti \sigma_1)$ is disjoint from the interior of $\ti f_{\#\#}(\ti \beta_j)$. Similarly $\ti f_\#(\ti \sigma_2)$ contains $\ti f_{\#\#}(\ti \beta_j)$ and $\ti f(\ti \sigma_2)$ is disjoint from the interior of $\ti f_{\#\#}(\ti \beta_i)$. It follows that $\ti f_{\#\#}(\ti \beta_i)$ and $\ti f_{\#\#}(\ti \beta_j)$ are subpaths of $\ti f_\#(\ti \sigma)$ with disjoint interiors and that former precedes the latter. Since $i < j$ are arbitrary this proves \pref{ItemDisjointCopies}.

To prove~\pref{ItemCircuit} write $\sigma$ as a concatenation $\alpha\beta$ and let $\tau = f_\#(\sigma)$. Any line $\ti\sigma$ in the universal cover $\wt G_1$ lifting $\sigma$ can be written as a bi-infinite concatenation $\ti\sigma = \ldots \ti\alpha_{-1} \, \ti\beta_{-1} \, \ti\alpha_0 \, \ti\beta_0 \, \ti\alpha_1 \, \ti\beta_1 \ldots$ of lifts of $\alpha$ and~$\beta$. Let $T$ be the covering transformation that increases indices by $1$, taking $\ti\alpha_i$ to $\ti\alpha_{i+1}$ etc. By \pref{ItemDisjointCopies} and equivariance of the $\#\#$ construction, the line $\ti f_\#(\ti\sigma)$ in $\wt G_2$ decomposes as a bi-infinite concatenation
$$\ti f_\#(\ti\sigma) = \ldots \ti f_{\#\#}(\ti\alpha_{-1}) \, \ti \gamma_{-1} \, \ti f_{\#\#}(\ti\alpha_0) \, \ti \gamma_0 \, \ti f_{\#\#}(\ti\alpha_1) \, \ti \gamma_1\ldots
$$ 
and $T$ increases indices by~$1$. It follows that $f_\#(\sigma) = f_{\#\#}(\alpha) \gamma$ as desired.
\end{proof}

\subsection{Subgroup systems carrying lines and other things.}  
\label{SectionSubgroupsCarryingThings}
Given a set of lines $\Lambda \subset \B$, in \BookOne\ Section~2.6 one finds the basic definitions and facts regarding free factor systems carrying $\Lambda$, and regarding the free factor support of $\Lambda$ which means the free factor system $\F$ which carries $\Lambda$ and which is the smallest such free factor system with respect to the partial ordering~$\sqsubset$. 

We shall need to generalize these concepts in two ways: allowing general subgroup systems; and allowing more general sets of objects beyond sets of lines. In all cases we use the operator $\F_\supp(\cdot)$ to denote the free factor support.

\subsubsection{Subgroup systems carrying lines and rays.} 
\label{SectionSubgroupLinesAndEnds}
For any a finite rank subgroup $K \subgroup F_n$, the induced $K$-equivariant inclusion $\bdy K \subset \bdy F_n$ induces in turn maps of rays $\bdy K / K \mapsto \bdy F_n / F_n$ and of lines $\B(K) \mapsto \B$. Note that the images of these two maps depend only on the conjugacy class of the subgroup $K$ in the group $F_n$.

Consider a subgroup system $\K$, an ray $\xi \in \bdy F_n / F_n$, and a line $\ell \in \B$. We say that \emph{$\K$ carries $\xi$} if there exists a subgroup $K \subgroup F_n$ such that $[K] \in \K$ and $\xi$ is in the image of the induced map $\bdy K / K \mapsto \bdy F_n / F_n$. We define \emph{$\K$ carries $\ell$} similarly, requiring instead that $\ell$ is in the image of the induced map $\B(K) \mapsto \B=\B(F_n)$. Let $\B(\K) \subset \B$ denote the subset all lines carried by $\K$, equivalently the union of the images of the maps $\B(K) \mapsto \B$ taken over all components $[K] \in \K$. We do not introduce any special notation for the subset of rays carried by $\K$.

Marked graphs can be used to detect carrying of lines and rays by free factor systems. This is an evident consequence of the definitions, as noted already for lines in \BookOne\ preceding Lemma~2.6.4:


\begin{fact} 
\label{FactLineRealizedCarried}
Consider a free factor system $\F$. For any marked graph $G$ and subgraph $H$ with $[H]=\F$ the following hold:
\begin{enumerate}
\item  A line $\ell \in \B$ is carried by $\F$ if and only if the realization of $\ell$ in $G$ is contained in $H$.
\item \label{ItemRayRealizedCarried}
An ray $\eta \in \bdy F_n / F_n$ is carried by $\F$ if and only if for any ray $\rho$ in $G$ realizing $\eta$, some subray of $\rho$ is contained in $H$.\qed
\end{enumerate}
\end{fact}

Given a conjugacy class $[g]$ of $F_n$ we say that \emph{$\K$ carries $[g]$} if $\K$ carries the axis of $[g]$; by Fact~\ref{FactBoundaries} this is equivalent to the existence of a subgroup $K \in \K$ such that $g^m \in K$ for some $m \ge 1$. In the special case that $\K=\F$ is a free factor system we have: $\F$ carries a conjugacy class $[g]$ if and only if $g \in A$ for some $A \in \F$; and $\F$ carries a line $\ell$ if and only if $\ell$ is a weak limit of a sequence of conjugacy classes carried by $\F$, if and only if for some (any) marked graph $G$ and a subgraph $H \subset G$ with $[\pi_1 H]=\F$, the realization of $\ell$ in $G$ is contained in $H$. 

The following facts about subgroups carrying lines will be useful in several places.


\begin{fact} \label{FactLinesClosed} \quad \hfill
\begin{enumerate}
\item\label{ItemLinesClosedOneGp}
For any finite rank subgroup $K \subgroup F_n$ the set $\wt\B(K)$ is closed in $\wt\B$ and the set $\B(K)$ is closed in $\B$.
\item \label{ItemLinesLocFin}
For any subgroup system $\K$ the collection of subsets $\{\wt\B(K) \suchthat K \in \K\}$ is uniformly locally finite in~$\wt\B$: there is a constant $C$ and an $F_n$-equivariant open cover $\U$ of $\wt\B$ such that each element of $\U$ intersects at most~$C$ of the subsets $\wt\B(K)$.
\item \label{ItemLinesClosedManyGps}
For any subgroup system $\K$ consisting of subgroups $\{K_i\}_{i \in I}$, and for any subset $J \subset I$, the set $\union_{i \in J}\wt\B(K_i)$ is closed in $\wt\B$.
\end{enumerate}
\end{fact}

\begin{proof} Item~\pref{ItemLinesClosedOneGp} is an immediate consequence of the fact that $\bdy K$ is a compact subset of~$\bdy F_n$ and the fact that the action of $K$ on its minimal subtree in $\wt R_n$ is cocompact. Items~\pref{ItemLinesClosedOneGp} and~\pref{ItemLinesLocFin} together clearly imply~\pref{ItemLinesClosedManyGps}. 

To prove~\pref{ItemLinesLocFin} we work in the Cayley graph $\wt R_n$. For any finite rank subgroup $A \subgroup F_n$ let $T(A) \subset \wt R_n$ be the minimal subtree for the action of $A$ on $\wt R_n$, which equals the convex hull in $\wt R_n$ of the Gromov boundary $\bdy A \subset \bdy R_n = \bdy F_n$. For each edge $e \subset \wt R_n$ let $\wt \B(e) \subset \B$ be the open subset consisting of lines whose realization in $\wt R_n$ contains $e$, and let $\U = \{\wt \B(e)\}$. Pick subgroups $K_1,\ldots,K_L \subgroup F_n$ such that $\K = \{[K_1],\ldots,[K_L]\}$. For each $l = 1,\ldots,L$ and each edge $e$ we shall bound the size of any subset $\Gamma \subset F_n$ representing pairwise distinct left cosets of $K_l$ such that if $\gamma \in \Gamma$ then $e \subset \gamma(T(K_l)) = T(\gamma K_l \gamma^\inv)$, equivalently $\wt\B(e) \intersect \wt\B(\gamma K_l \gamma^\inv) \ne \emptyset$. Pick a finite subtree $D_l \subset T(K_l)$ which is a fundamental domain for the action of $K_l$ on $T(K_l)$. For each $\gamma \in \Gamma$, after replacing $\gamma$ by $\gamma k$ for appropriate $k \in K_l$ we may assume that $e \subset \gamma D_l$, and by proper discontinuity of the action of $F_n$ on $\wt R_n$ the set $\Gamma$ has bounded cardinality $C_{l,e}$ depending only on $l$ and~$e$. Noticing that $C_{l,e}$ depends only on the orbit of $e$ under the action of $F_n$, by restricting to edges $e$ in a fundamental domain for the action of $F_n$ on $\wt R_n$ we obtain a finite maximum $C = \max C_{l,e}$, finishing the proof.
\end{proof}

\subsubsection{Free factor supports of lines, rays, and subgroup systems.}
\label{SectionFFSupp}
The \emph{free factor support} of a subset of lines $B \subset \B$, denoted $\F_\supp(B)$, is defined by the following result from \BookOne:

\begin{fact}[\BookOne, Corollary 2.6.5]
\label{FactFFSupport}
For any subset $B \subset \B$ the meet of all free factor systems that carry $B$ is a free factor system that carries $B$ that is denoted $\F_\supp(B)$. If $B$ is a single line then $\F_\supp(B)$ has one component. $\F_\supp(B)$ is the unique free factor system carrying $B$ that is minimal with respect to the partial ordering~$\sqsubset$.
\end{fact}

\noindent
Fact~\ref{FactFFSupport} is also sufficient to incorporate the \emph{free factor support} of a set of conjugacy classes~$C$, denoted $\F_\supp(C)$, by which we simply mean the free factor support of the set of axes for~$C$.

We need to define the free factor support $\F_\supp(X)$ of a polyglot set $X$ each of whose elements is a line, conjugacy class, ray, or subgroup system of $F_n$. This is done using the following generalization of Fact~\ref{FactFFSupport}:

\begin{fact}
\label{FactFFSPolyglot}
For any set $X$ each of whose elements is either a line, conjugacy class, ray, or subgroup system of $F_n$, the following hold:
\begin{enumerate}
\item \label{ItemFFSPolyDefined}
The meet of all free factor systems that carry each element of $X$ is the unique free factor system that carries each element of $X$ and is minimal with respect to $\sqsubset$. This free factor system is called the \emph{free factor support} of~$X$, denoted  $\F_\supp(X)$. 
\item \label{ItemFFSPolyNatural}
$\F_\supp(X)$ is natural in the sense that for each $\phi \in \Out(F_n)$ we have $\F_\supp(\phi(X)) = \phi(\F_\supp(X))$. In particular if $X$ is $\phi$-invariant then $\F_\supp(X)$ is $\phi$-invariant.
\item \label{ItemFFSSingle}
If $X$ is a single line or ray then $\F_\supp(X)$ has one component.
\end{enumerate}
\end{fact}

\begin{proof} For the proof of~\pref{ItemFFSPolyDefined}, as in the proof of Fact~\ref{FactFFSupport} given in \BookOne\ Corollary 2.6.5, because the meet of an arbitrary collection of free factor systems equals the meet of some finite subcollection, and because meet is associative, it suffices to consider just two free factor systems $\F_1,\F_2$, and to prove that if a line, conjugacy class, ray, or subgroup system is carried by $\F_1$ and by $\F_2$ then it is carried by $\F_1 \meet \F_2$. For conjugacy classes the definition of carrying is a special case of the definition for lines and so need not be considered further. Recall the proof for a line $\ell$ carried by $\F_1,\F_2$: choosing a lift $\ti\ell \in \wt\B$, there exist $A_1,A_2 \subgroup F_n$ such that $[A_i] \in \F_i$ and $\bdy\ell \subset \bdy A_i$, and so $\bdy\ell \subset \bdy A_1 \intersect \bdy A_2 = \bdy(A_1 \intersect A_2)$ (see Fact~\ref{FactBoundaries}); by Fact~\ref{FactGrushko} we have $[A_1 \intersect A_2] \in \F_1 \meet \F_2$, so $\ell$ is carried by $\F_1 \meet \F_2$.  Similarly, for an ray $\xi \in X$ carried by $\F_1,\F_2$, choosing a lift $\ti\xi \in \bdy F_n$, again there exist $A_i$ as above with $\ti\xi \in \bdy A_i$ and so $\ti\xi \in \bdy(A_1 \intersect A_2)$ and so $\xi$ is carried by $\F_1 \meet \F_2$. For a subgroup system $\K \sqsubset \F_n$, it suffices to consider each component $[K] \in \K$ one at a time. For $i=1,2$ there exists $A_1,A_2 \subgroup F_n$ such that $K \subgroup A_i$, and so $K \subgroup A_1 \intersect A_2$ implying that $[K] \sqsubset \F_1 \meet \F_2$.

The naturally clause~\pref{ItemFFSPolyNatural} is evident. The proof of~\pref{ItemFFSSingle} when $X$ is a single line is given shown in \BookOne, Corollary 2.6.5, and the same proof works if $X$ is a single ray $\rho$, using the evident fact that if a free factor system $\F$ carries $\rho$ then some component of $\F$ carries~$\rho$.
\end{proof}

The following fact relates the free factor support of a subgroup system $\K$ to the free factor support of the set of lines carried by $\K$:

\begin{fact}
\label{FactSubgroupSupport}
Given a subgroup system $\K$ the following hold:
\begin{enumerate}
\item \label{ItemSubgpCarryingEquiv}
For any free factor system $\F$, the inclusion $\K \sqsubset \F$ holds if and only if each line carried by $\K$ is carried by $\F$.
\item \label{ItemSubgpSupport}
$\F_\supp(\B(\K)) = \F_\supp(\K)$.
\end{enumerate}
\end{fact}

\begin{proof}
The implication~\pref{ItemSubgpCarryingEquiv}$\implies$~\pref{ItemSubgpSupport} is obvious, as is the ``only if'' direction of~\pref{ItemSubgpCarryingEquiv}.

For the ``if'' direction of~\pref{ItemSubgpCarryingEquiv}, given $K \subgroup F_n$ so that $[K] \in \K$, it suffices to assume that $\B(K)$ is carried by $\F$ and that $[K] \not\sqsubset \F$, and to derive a contradiction. Choose a nontrivial $\gamma_1 \in K$. Choose $A_1$ so that $[A_1] \in \F$ and $\bdy\gamma_1 \subset \bdy A_1$. Since $[K] \not\sqsubset \F$ it follows that $K \not\subgroup A_1$. Choosing $\gamma_2 \in K - A_1$ it follows that $\bdy\gamma_2 \not\subset \bdy A_1$, for otherwise $\bdy A_1 \intersect \bdy(\gamma_2 A_1 \gamma_2^\inv) = \bdy A_1 \intersect \gamma_2(\bdy A_1) \ne \emptyset$ contradicting malnormality of $A_1$ and Fact~\ref{FactBoundaries}. Choose $A_2 \in \F$ such that $\bdy\gamma_2 \subset \bdy A_2$. Clearly $\bdy\gamma_1 \ne \bdy\gamma_2$ and $A_1 \ne A_2$. Pick $\xi_1 \in \bdy\gamma_1$ and $\xi_2 \in \bdy\gamma_2$ such that $\xi_1 \ne \xi_2$, let $\ti\ell = \{\xi_1,\xi_2\} \in \wt\B$, and let $\ell \in \B$ be its projection. By construction we have $\ell \in \wt\B(K)$ so we may choose $A_3 \in \F$ such that $\bdy\ti\ell \subset \bdy A_3$. It follows that $\bdy A_1 \intersect \bdy A_3 \ne \emptyset$ and $\bdy A_2 \intersect \bdy A_3 \ne \emptyset$. But since one of $A_1$ or $A_2$ is not equal to $A_3$ this contradicts malnormality of $\F$ and Fact~\ref{FactBoundaries}.
\end{proof}

\subparagraph{Remark.} By similar proofs we can show that the equivalent statements in Fact~\ref{FactSubgroupSupport}~\pref{ItemSubgpCarryingEquiv} are also equivalent to the statement that each ray carried by $\K$ is carried by $\F$, and that the subgroup system in~\pref{ItemSubgpSupport} is also equal to the free factor support of the set of rays carried by~$\K$. We do not need these results so we omit the proofs.

\medskip

The following fact gives a bit more useful information relating free factor systems to the lines and rays that they carry.

\begin{fact}
\label{FactWeakLimitLines} 
For any free factor system $\F$,
\begin{enumerate}
\item \label{ItemLineSequenceLimits}
For every sequence of lines $\ell_i \in \B$, if every weak limit of every subsequence of $\ell_i$ is carried by $\F$ then $\ell_i$ is carried by $\F$ for all sufficiently large $i$.
\item \label{ItemEndLineSupportRelated}
The weak accumulation set of every ray not carried by $\F$ contains a line not carried by $\F$.
\item \label{ItemWeakLimitRayAccSame}
The free factor support of every ray is equal to the free factor support of its weak accumulation set.
\end{enumerate}
\end{fact}

\begin{proof} If the conclusion of~\pref{ItemLineSequenceLimits} fails then some subsequence $\ell_{i_n}$ is not carried by $\F$. Choose a marked graph $G$ and a core subgraph $H$ representing $\F$. The realization of $\ell_{i_n}$ in $G$ is not contained in $H$ and so, after passing to a subsequence, each $\ell_{i_n}$ contains some edge $E \subset G-H$. It follows that some weak limit of some subsequence is a line that contains $E$, and so is not contained in $H$ and is not carried by $\F$. 

To prove \pref{ItemEndLineSupportRelated}, if the ray $r$ in $G$ realizes an abstract ray not carried by $\F$ then there is an edge $E \subset G-H$ that the ray $r$ crosses infinitely many times, and we can proceed as in~\pref{ItemLineSequenceLimits}.

To prove \pref{ItemWeakLimitRayAccSame}, given an ray $\xi \in \bdy F_n / F_n$ with weak accumulation set $\Lambda \subset \B$, the inclusion $\F_\supp(\xi) \sqsubset \F_\supp(\Lambda)$ is a consequence of~\pref{ItemEndLineSupportRelated}, and the reverse inclusion $\F_\supp(\Lambda) \sqsubset \F_\supp(\xi)$ follows by Fact~\ref{FactLineRealizedCarried}~\pref{ItemRayRealizedCarried}.
\end{proof}

\subsection{Attracting laminations.} 
\label{SectionAttractingLams}

We recall here the basic definitions and properties of attracting laminations proved in Section~3 of \BookOne.

A closed subset of $\B=\B(F_n)$ is called a \emph{lamination}, or an \emph{$F_n$-lamination} when emphasis is needed. An element of a lamination is called a \emph{leaf}. The action of $\Out(F_n)$ on $\B$ induces an action on the set of laminations. 

For each marked graph $G$ the homeomorphism $\B \approx \B(G)$ induces a bijection between $F_n$-laminations and closed subsets of $\B(G)$. The closed subset of $\B(G)$ corresponding to a lamination $\Lambda \subset \B$ is called the \emph{realization} of $\Lambda$ in~$G$; we often conflate $\Lambda$ with its realizations in marked graphs. Also, we occasionally use the term lamination to refer to $F_n$-invariant, closed subsets of $\wt \B$; the orbit map $\wt\B \mapsto \B$ puts these in natural bijection with laminations in~$\B$. 

A line $\ell \in \B$ is \emph{birecurrent} if $\ell$ is in the weak accumulation set of every subray of $\ell$. A lamination $\Lambda$ is \emph{minimal} if each of its leaves is dense in $\Lambda$. Every leaf of a minimal lamination is birecurrent, and a connected lamination in which every leaf is birecurrent is minimal.

\begin{definition}[Attracting laminations]
\label{DefAttractingLaminations}
Given $\phi \in \Out(F_n)$ and a lamination $\Lambda \subset \B$, we say that $\Lambda$ is an \emph{attracting lamination} for $\phi$ if there exists a leaf $\ell \in \Lambda$ satisfying the following: $\Lambda$ is the weak closure of $\ell$; $\ell$ is a birecurrent; $\ell$ is not the axis of the conjugacy class of a generator of a rank one free factor of $F_n$; and there exists $p \ge 1$ and a weak open set $U \subset \B$ such that $\phi^p(U) \subset U$ and such that $\{\phi^{kp}(U) \suchthat k \ge 1\}$ is a weak neighborhood basis of $\ell$. Any such leaf $\ell$ is called a \emph{generic leaf} of $\Lambda$, and any such neighborhood $U$ is called an \emph{attracting neighborhood of $\Lambda$ for the action of $\phi^p$}. Let $\L(\phi)$ be the set of attracting laminations for~$\phi$. 
\end{definition}

\textbf{Remark.} Given $\Lambda \in \L(\phi)$, if $\phi(\Lambda)=\Lambda$ then in Definition~\ref{DefAttractingLaminations} we may assume $p=1$, because if $p>1$ then by Fact~\ref{FactLamsBasicProps}~\pref{ItemLphiFiniteAndInv} we may replace $U$ by $U \intersect \phi(U) \intersect\cdots\intersect \phi^{p-1}(U)$. In general we may assume that $p$ is the period of~$\Lambda$, meaning the least integer for which $\phi(p)=\Lambda$. 


\begin{fact}
\label{FactLamsBasicProps}
For each $\phi \in \Out(F_n)$ the following hold:
\begin{enumerate}
\item\label{ItemLphiFiniteAndInv}
$\L(\phi)$ is finite and $\phi$-invariant.
\item\label{ItemLamFFSConnected}
For all $\Lambda \in \L(\phi)$ and each generic leaf $\ell$ of $\Lambda$ we have $\F_\supp(\Lambda)=\F_\supp(\ell)$, and this free factor system has one component.
\item For all $\Lambda,\Lambda' \in \L(\phi)$ we have $\Lambda=\Lambda'$ if and only if $\F_\supp(\Lambda)=\F_\supp(\Lambda')$.
\item\label{ItemDualityDef}
There is a bijection $\L(\phi) \leftrightarrow \L(\phi^\inv)$ defined by $\Lambda_+ \leftrightarrow \Lambda_-$ if and only if $\F_\supp(\Lambda_+) = \F_\supp(\Lambda_-)$.\qed
\end{enumerate}
\end{fact}
\noindent

\begin{definition}[Dual pairs of laminations]
\label{DefinitionDualLams}
The bijection defined in item~\pref{ItemDualityDef} is called \emph{duality} between $\L(\phi)$ and $\L(\phi^\inv)$, a pair $\Lambda^\pm = (\Lambda^+,\Lambda^-)$ corresponding in this manner is called a \emph{dual pair} of laminations of $\phi$, and the set of dual pairs is denoted $\L^\pm(\phi)$. 
\end{definition}

\textbf{Remark.}
In general the action of $\phi$ on the finite set $\L(\phi)$ may be a nontrivial permutation. But in certain circumstances one sees that each element of $\L(\phi)$ is fixed: this holds for $\phi \in \IA_n(\Z/3)$ by Lemma~\refRK{LemmaFFSComponent}\footnote{``Lemma II.X.Y'' or ``Section II.V.W'' refers to Lemma X.Y or Section V.W of \PartTwo.}, and for rotationless $\phi$ (see Definition~\ref{DefPrinicipalAndRotationless} below) by Lemma~3.30 of \recognition.

\subsection{Principal automorphisms and rotationless outer automorphisms.} 
\label{SectionPrincipalRotationless}
The normal subgroup $\Inn(F_n)$ acts by conjugation on $\Aut(F_n)$, and the orbits of this action define an equivalence relation on $\Aut(F_n)$ called \emph{isogredience}, which is a refinement of the relation of being in the same outer automorphism class. In this section we shall define various isogredience invariants of elements of $\Aut(F_n)$.

\begin{notn}  One inevitably abbreviates `forward rotationless' to `rotationless'   so we will do that from the start. Throughout this series of papers, forward rotationless outer automorphisms are referred to as rotationless outer automorphisms.
\end{notn}      

\begin{definition}\label{DefHatAndFix}
The group $\Aut(F_n)$ acts on $\bdy F_n$. Let $\wh\Phi \from \bdy F_n \to \bdy F_n$ denote the continuous extension of the action of $\Phi \in \Aut(F_n)$, and $\Fix(\wh\Phi) \subset \bdy F_n$ its set of fixed points. Let $\Fix(\Phi) \subgroup F_n$ denote the subgroup of elements fixed by $\Phi$, shown first to be of finite rank in \cite{Cooper:automorphisms} and then to be of rank~$\le n$ in \cite{BestvinaHandel:tt}. Its boundary $\bdy \Fix(\Phi) \subset \bdy F_n$ is empty when $\Fix(\Phi)$ is trivial, 2 points when $\Fix(\Phi)$ has rank~1, and a Cantor set when $\Fix(\Phi)$ has rank~$\ge 2$. Let $\Fix_+(\wh\Phi)$ denote the set of attractors in $\Fix(\wh\Phi)$, a discrete subset consisting of points $\xi \in \Fix(\wh\Phi)$ such that for some neighborhood $U \subset \bdy F_n$ of $\xi$  we have $\wh\Phi(U) \subset U$ and the sequence $\wh\Phi^n(\eta)$ converges to $\xi$ for each $\eta \in U$. Let $\Fix_-(\wh\Phi) = \Fix_+(\wh\Phi^\inv)$ denote the set of repellers in $\Fix(\wh\Phi)$. For a nontrivial $\gamma \in F_n$ the associated homeomorphism at infinity, denoted $\hat\gamma \from \bdy F_n \to \bdy F_n$, is identical to continuous extension of the action of the associated inner automorphism $i_\gamma \in \Inn(F_n)$, and we denote the 2-point fixed set as $\Fix(\hat\gamma) = \{\Fix_-(\hat\gamma),\Fix_+(\hat\gamma)\}$. The following lemma combines results of \cite{GJLL:index} and \cite{BFH:Solvable}; see \recognition\ Lemma~2.1 for detailed citations.
\end{definition}

\begin{fact} \label{FactFPBasics} 
For any $\Phi \in \Aut(F_n)$, and for any nontrivial $\gamma \in F_n$ with the following are equivalent:

(1) $\Fix(\hat\gamma) \intersect \Fix(\wh\Phi) \ne \emptyset$

(2) $\Fix(\hat\gamma) \subset \Fix(\wh\Phi)$ 

(3) $\gamma \in \Fix(\Phi)$

(4) $\Phi$ commutes with $i_\gamma$.\qed
\end{fact}

The following simple corollary of Fact~\ref{FactFPBasics} will be used often without comment: 


\begin{fact}\label{FactTwoLifts}  
If $\Phi_1 \ne \Phi_2 \in \Aut(F_n)$ represent the same outer automorphism then $\Fix(\wh\Phi_1) \intersect \Fix(\wh\Phi_2)$ is either empty or equal to $\Fix(\hat\gamma)$ where $\Phi_1^\inv \Phi_2 = i_\gamma$.\qed
\end{fact}


The following statement is a consequence of \cite{GJLL:index}, Proposition I.1. 

\begin{fact}
\label{LemmaFixPhiFacts} 
For each $\Phi \in \Aut(F_n)$ we have 
$$\Fix(\wh\Phi) = \bdy\Fix(\Phi) \union \Fix_-(\wh\Phi) \union \Fix_+(\wh\Phi) \qquad (*)
$$
and each term on both sides of $(*)$ is invariant under the action of $\Fix(\Phi)$ on $\bdy F_n$. Also, if $\Fix(\wh\Phi) \ne \Fix(\hat\gamma)$ for each nontrivial $\gamma \in \F_n$ then the following hold:
\begin{enumerate}
\item\label{ItemFixDisjUnion}
The expression $(*)$ is a disjoint union.
\item\label{ItemStrongAttracting}
Each $\xi \in \Fix_+(\wh\Phi)$ is an attracting point for the action of $\overline\Phi = \Phi \union \wh\Phi$ on the Gromov compactification $\overline F_n = F_n \union \bdy F_n$. 
\item\label{ItemFixPhiInvariance}
The points of $\Fix_+(\wh\Phi) \union \Fix_-(\wh\Phi)$ are the isolated points of the set $\Fix(\wh\Phi)$.
\end{enumerate}
\end{fact}

\begin{proof} Given $\gamma \in \Fix(\Phi)$, by Fact~\ref{FactFPBasics} the actions of $\wh\Phi$ and $\hat\gamma$ on $\bdy F_n$ commute; invariance of the terms in $(*)$ under $\hat\gamma$ follows immediately.

Letting $\bdy_+\overline\Phi \subset \bdy F_n$ be the set of fixed attractors for the action of $\overline\Phi$ on $\overline F_n$, and letting $\bdy_-\overline\Phi = \bdy_+ \overline\Phi^\inv$, clearly $\bdy_+\overline\Phi \subset \Fix_+(\wh\Phi)$ and $\bdy_-\overline\Phi \subset \Fix_-(\wh\Phi)$. Applying \cite{GJLL:index}, Proposition I.1 we have the equation and the disjoint union given in the first line of the following display:
\begin{align*}
\Fix(\wh\Phi) &= \, \bdy\Fix(\Phi) \union \bdy_- \overline\Phi \union \bdy_+ \overline\Phi \\
   &\subset \bdy\Fix(\Phi) \union \Fix_-(\overline\Phi) \union  \Fix_+(\overline\Phi) \, \subset \, \Fix(\wh\Phi)
\end{align*}
and Equation~$(*)$ follows. Since $\Fix_+(\wh\Phi)$ and $\Fix_-(\wh\Phi)$ are clearly disjoint and their points are clearly isolated in $\Fix(\wh\Phi)$, what remains in proving~(\ref{ItemFixDisjUnion}--\ref{ItemFixPhiInvariance}) is to verify disjointness of $\bdy\Fix(\Phi)$ from $\Fix_\pm(\wh\Phi) = \Fix_-(\wh\Phi) \union \Fix_+(\wh\Phi)$ and nonisolation of the points of $\bdy\Fix(\Phi)$ in the set $\Fix(\wh\Phi)$. If $\Fix(\Phi)$ is trivial then $\bdy\Fix(\Phi)$ is empty. If $\rank(\Fix(\Phi)) \ge 2$ then $\bdy\Fix(\Phi)$ is a Cantor subset of $\Fix(\wh\Phi)$, whereas $\Fix_\pm(\Phi)$ is a discrete subset. If $\rank(\Fix(\Phi))=1$ and if $\Fix_\pm(\wh\Phi) \not\subset \bdy\Fix(\Phi)$ then the $\Fix(\Phi)$ orbit of any point of $\Fix_\pm(\wh\Phi) - \bdy\Fix(\Phi)$ is a subset of $\Fix(\wh\Phi)$ that accumulates on both points of $\bdy\Fix(\Phi)$. Disjointness and nonisolation evidently follows in all of these cases. The only remaining case is when $\rank(\Fix(\Phi))=1$ and $\Fix_\pm(\wh\Phi) \subset \bdy\Fix(\Phi) = \Fix(\hat\gamma)$ where $\gamma$ generates $\Fix(\Phi)$, and it follows that $\Fix(\wh\Phi) = \Fix(\hat\gamma)$.
%
\end{proof}

\subparagraph{Remark.} From the proof of Lemma~\ref{LemmaFixPhiFacts}, our set $\Fix_+(\wh\Phi)$ equals the set of attracting fixed points under the action on $\overline \Phi = \Phi \union \wh\Phi$ on $F_n \union \bdy F_n$, as defined in \cite{GJLL:index}, outside of the exceptional situation described in the statement of Lemma~\ref{LemmaFixPhiFacts}. But in that exceptional situation this equality of sets may fail, for instance if we fix $\Phi' \in \Aut(F_n)$ and let $\Phi = i_\alpha^k \Phi'$ for some inner automorphism $i_\alpha$ and some sufficiently large~$k$. 

Let $\Fix_N(\wh\Phi) = \Fix(\wh\Phi) - \Fix_-(\wh\Phi)$, which equals $\bdy\Fix(\Phi) \union \Fix_+(\wh\Phi)$ as long as $\Fix(\wh\Phi)$ is not the fixed point pair of a nontrivial element of $F_n$. Let $\Per(\wh\Phi) = \union_{k \ge 1} \Fix(\wh\Phi^k)$, and similarly for $\Per_+(\wh\Phi)$, $\Per_-(\wh\Phi)$, and $\Per_N(\wh\Phi)$.

\begin{fact}[\cite{FeighnHandel:recognition} Lemmas~3.23 and 3.26, or \cite{LevittLustig:PeriodicDynamics} Theorem I]\rc\
\label{FactPeriodicNonempty}
For every $\Phi \in \Aut(F_n)$ there exists $m > 0$ such that $\Fix_N(\wh\Phi^m) \ne \emptyset$. \qed
\end{fact}

\bigskip

The following combines Definitions~3.1 and~3.13 of \recognition:

\begin{definition}[Principal automorphisms and rotationless outer automorphisms.]
\label{DefPrinicipalAndRotationless}
We say that $\Phi \in \Aut(F_n)$ is a \emph{principal automorphism} if $\Fix_N(\wh\Phi)$ contains at least three points, or $\Fix_N(\wh\Phi)$ contains exactly two points which are neither the endpoints of some axis nor the endpoints of a lift of a generic leaf of an element of $\L(\phi)$. Note that if $\Phi$ is principal then by Fact~\ref{LemmaFixPhiFacts} we have a disjoint union $\Fix_N(\wh\Phi)  = \bdy\Fix(\Phi) \union \Fix_+(\wh\Phi)$. The set of principal automorphisms representing $\phi \in \Out(F_n)$ is denoted $P(\phi)$, it is invariant under isogredience, and by \recognition\ Remark~3.9 there are only finitely many is isogredience classes.

We say that $\phi \in \Out(F_n)$ is \emph{(forward) rotationless} if $\Fix_N(\wh\Phi) = \Per_N(\wh\Phi)$ for all $\Phi \in P(\phi)$, and if for each $k \ge 1$ the map $\Phi \mapsto \Phi^k$ induces a bijection between $P(\phi)$ and $P(\phi^k)$.  
\end{definition}

\begin{fact}[\recognition\ Lemma 4.42]\rc\
\label{FactRotationlessPower}
There exists $K$ depending only on the rank $n$ such that for each $\phi \in \Out(F_n)$, $\phi^K$ is rotationless. \hfill\qed
\end{fact}


\begin{fact}[\recognition\ Lemma~3.30 and Corollary~3.31]\rc
 \label{FactPeriodicIsFixed}
If $\phi \in \Out(F_n)$ is rotationless then
\begin{enumerate}
\item \label{ItemPeriodicClassFixed}
Every conjugacy class in $F_n$ which is $\phi$-periodic is fixed by $\phi$.
\item $\phi$ fixes every element of $\L(\phi)$.
\item \label{ItemFreeFactorFixed}
Every free factor system in $F_n$ which is $\phi$-periodic is fixed by $\phi$.
\item If $F$ is a $\phi$-invariant free factor then $\phi \restrict F$ is rotationless. \hfill\qed
\end{enumerate}
\end{fact}

\subsection {Relative train track maps and \ct s}
\label{SectionRTTandCT}

In this section we review general relative train track maps, as well as the more specialized \cts\ or completely split relative train track maps.

\subsubsection{Definitions of relative train track maps and \cts}
\label{SectionRTTDefs}

\subparagraph{Topological representatives and Nielsen paths.} 
\label{SectionTopReps}
Given $\phi \in \Out(F_n)$ a \emph{topological representative} of $\phi$ is a map $f \from G \to G$ such that $\rho \from R_n \to G$ is a marked graph, $f$ is a homotopy equivalence, $f$ takes vertices to vertices and edges to paths, and $\bar\rho \composed f \composed \rho \from R_n \to R_n$ represents $\phi$. 

A nontrivial path $\gamma$ in $G$ is a \emph{periodic Nielsen path} if there exists $k$ such that $f^k_\#(\gamma)=\gamma$. The minimal such $k$ is the \emph{period}, and if $k=1$ then $\gamma$ is a \emph{Nielsen path}. A periodic Nielsen path is \emph{indivisible} if it cannot be written as a concatenation of two nontrivial periodic Nielsen paths. We will often abuse our notion of path equivalence when talking about Nielsen paths: many statements asserting the uniqueness of an indivisible Nielsen path $\rho$, such as Fact~\ref{FactEGNPUniqueness}, should really assert uniqueness of $\rho$ \emph{up to reversal of direction}; since $\rho$ determines its reversal $\bar\rho$ and vice versa, this abuse is easily detectable and harmless.

\subparagraph{Filtrations.} A \emph{filtration}\index{filtration} of a marked graph $G$ is a strictly increasing sequence of subgraphs $G_0 \subset G_1 \subset \cdots \subset G_k = G$, each with no isolated vertices. The individual terms $G_k$ are called \emph{filtration elements}, and if $G_k$ is a core graph then it is called a \emph{core filtration element}.\index{core filtration element} The subgraph $H_k = G_k \setminus G_{k-1}$ is called the \emph{stratum of height $k$}.\index{stratum} The \emph{height}\index{height} of subset of $G$ is the minimum $k$ such that the subset is contained in $G_k$. The height of a map to $G$ is the height of the image of the map. The height of an abstract ray $\xi \in \bdy F_n / F_n$ equals $s$ if and only if every ray in $G$ realizing $\xi$ has a subray contained in $G_s$ and not contained in $G_{s-1}$; equivalently, for any ray $R$ in $G$ realizing $\xi$, the highest stratum having edges crossed infinitely often by $R$ is $H_s$. A \emph{connecting path} of a stratum $H_k$ is a nontrivial finite path $\gamma$ of height $< k$ whose endpoints are contained in $H_k$.

Given a topological representative $f \from G \to G$ of $\phi \in \Out(F_n)$, we say that $f$ \emph{respects} the filtration or that the filtration is \emph{$f$-invariant} if $f(G_k) \subset G_k$ for all $k$. If this is the case then we also say that the filtration is \emph{reduced} if for each free factor system $\A$ which is invariant under $\phi^i$ for some $i \ge 1$, if $[\pi_1 G_{r-1}] \sqsubset \A \sqsubset [\pi_1 G_r]$ then either $\A = [\pi_1 G_{r-1}]$ or $\A = [\pi_1 G_r]$.

Given an $f$-invariant filtration, for each stratum $H_k$ with edges $\{E_1,\ldots,E_m\}$, define the \emph{transition matrix} of $H_k$ to be the square matrix whose $j^{\text{th}}$ column records the number of times $f(E_j)$ crosses the other edges. If $M_k$ is the zero matrix then we say that $H_k$ is a \emph{zero stratum}. If $M_k$ irreducible --- meaning that for each $i,j$ there exists $p$ such that the $i,j$ entry of the $p^{\text{th}}$ power of the matrix is nonzero --- then we say that $H_k$ is \emph{irreducible}; and if one can furthermore choose $p$ independently of $i,j$ then $H_k$ is \emph{aperiodic}. Assuming that $H_k$ is irreducible, by  \cite{Hawkins:PerronFrobenius} the matrix $M_k$ a unique eigenvalue $\lambda \ge 1$, called the \emph{Perron-Frobenius eigenvalue}, for which some associated eigenvector has positive entries: if $\lambda>1$ then we say that $H_k$ is an \emph{exponentially growing} or \eg\ stratum; whereas if $\lambda=1$ then $H_k$ is a \emph{nonexponentially growing} or \neg\ stratum.

\subparagraph{Directions and turns.}  A \emph{direction} of a marked graph $G$ at a point $x \in G$ is a germ of finite paths with initial vertex~$x$. If $x$ is not a vertex then the number of directions at $x$ equals~2. If $x$ is a vertex then the number of directions equals the valence of $x$, in fact there is a natural bijection between the directions at $x$ and the oriented edges with initial vertex $x$, and we shall often identify a direction at $x$ with its corresponding oriented edge. Let $T_x G$ be the set of directions of $G$ at~$x$ and let $TG$ be the union of $T_x G$ over all vertices $x$ of $G$.  A \emph{turn} at $x \in G$ is an unordered pair of directions $\{d,d'\} \subset T_x G$. If $d \ne d'$ then the turn is \emph{nondegenerate}, otherwise it is \emph{degenerate}. 

If a filtration of $G$ is given, the \emph{height} of a direction $d$ is well-defined as the height of any sufficiently short path representing $d$, equivalently the height of the oriented edge representing $d$. A nondegenerate turn is said to have height $r$ if \emph{each} of its directions has height $r$.

Given a marked graph $G$ and a homotopy equivalence $f \from G \to G$ that takes edges to paths, a turn $\{d,d'\}$ in $G$ is said to be \emph{legal} with respect to the action of $f$ if the turn $\{(Df)^k(d),(Df)^k(d')\}$ is nondegenerate for all $k \ge 0$. For any path $\gamma$ and any turn $\{E,E'\}$ at a vertex, if $\bar E E'$ or its inverse $\bar E' E$ is a subpath of $\gamma$ then we say that the turn $\{E,E'\}$ is \emph{taken} by $\gamma$. If in addition an $f$-invariant filtration is given, a path $\gamma$ is \emph{$r$-legal} if it has height $\le r$ and each turn of height $r$ taken by $\gamma$ is legal.

\begin{definition}\label{DefRelTT}
\textbf{Relative train track maps.} We define relative train track maps using the method of \recognition\ Section~2.6, which imposes certain mild preconditions on a filtered topological representative, and which explains how an arbitrarily filtered topological representative can be mildly altered, by subdivision of edges and/or refinement of the filtration, so as to satisfy these preconditions.

Consider $\phi \in \Out(F_n)$ and a topological representative $f \from G \to G$ of $\phi$ with $f$-invariant filtration $G_0 \subset G_1 \subset \cdots \subset G_k = G$ satisfying the following preconditions:
\begin{enumerate} 
\item\label{ItemStrataZeroOrIrr}
Every stratum is either a zero stratum or irreducible.
\item\label{ItemNEGNotUEV}
The edges of each \neg\ stratum $H_k$ may be assigned orientations called \emph{\neg\ orientations},\index{\neg\ orientation} and they may be numbered as $E_1,\ldots,E_N$, so that for all $i \in \Z/N\Z$ we have $f(E_i) = E_{i+1} u_i$ where $u_i \subset G_{k-1}$ is a (possibly empty) path. 
\end{enumerate}
Note that while item~\pref{ItemNEGNotUEV} is stable under passage to a positive power $f^k \from G \to G$, item~\pref{ItemStrataZeroOrIrr} need not be. However there always exists \emph{some} positive power so that, after further refinement of the invariant filtration, each irreducible stratum is aperiodic, which forces item~\pref{ItemStrataZeroOrIrr} to become stable under passage to any further positive power.

We say that $f$ is a \emph{relative train track} representative of $\phi$ if for each \eg\ stratum $H_r$ the following hold:
\begin{description}
\item[RTT-(i)] $Df$ maps the set of directions of height $r$ to itself. In particular, each turn consisting of a direction of height $r$ and one of height $<r$ is legal.
\item[RTT-(ii)] For each connecting path $\gamma$ of $H_r$, $f_\#(\gamma)$ is a connecting path of $H_r$.
\item[RTT-(iii)] For each $r$-legal path $\alpha$, the path $f_\#(\alpha)$ is $r$-legal.
\end{description}

For each \neg\ stratum $H_k$ we use the following additional terminology. If each $u_i$ is the trivial path then we say that $H_k$ is a \emph{periodic stratum} and each edge $E \subset H_k$ is a \emph{periodic edge}; if furthermore $N=1$ then $H_k$ is a \emph{fixed stratum} and $E$ is a \emph{fixed edge}.\index{fixed edge} If each $u_i$ is either trivial or a periodic Nielsen path, with at least one periodic Nielsen path, then we say that $H_k$ is a \emph{linear stratum} and each $E \subset H_k$ is a \emph{linear edge}.\index{linear edge} An \neg\ edge $E$ which is neither a fixed edge nor a linear edge is said to be a \emph{superlinear edge}.\index{superlinear edge} This completes the definition of a relative train track map.
\end{definition}

For the most part we will be interested in the special class of relative train track maps called \cts, recounted in Definition~\ref{DefCT} after some further preliminaries.


\begin{definition} \textbf{Principal vertices.} 
\label{DefPrincipalVertices}
Consider a relative train track map $f \from G \to G$ with filtration $\emptyset = G_0 \subset G_1 \subset \cdots \subset G_N$. Two periodic points $x \ne y \in G$ are \emph{Nielsen equivalent} if there exists a periodic Nielsen path with endpoints $x,y$. A periodic point $x$ is \emph{principal} if and only if \emph{neither} of the following hold:
\begin{enumerate}
\item\label{ItemEGNonprincipal} $x$ is the unique periodic point in its Nielsen class, there are exactly two periodic directions based at $x$, and both of those directions are in the same \eg\ stratum.
\item\label{ItemBadCircle} $x$ is contained in a topological circle $C \subset G$ such that each point of $C$ is periodic and has exactly two periodic directions (but see Remark following Fact~\ref{FactPrincipalVertices} which says that no such periodic point $x$ exists in the case that $f$ is a \ct).
\end{enumerate}
\end{definition}

\begin{definition} \textbf{Splittings.} 
\label{DefSplittings}
Let $f \from G \to G$ be a relative train track map with invariant filtration $\emptyset = G_0 \subset G_1 \subset \cdots \subset G_N = G$. A \emph{splitting} of a path or circuit $\sigma$ in $G$ is a decomposition of $\sigma$ into a concatenation of subpaths $\sigma = \sigma_1 \cdot \ldots \cdot \sigma_k$ such that for all $i \ge 1$ the path or circuit $f^i_\#(\sigma)$ decomposes as $f^i_\#(\sigma_1) \ldots f^i_\#(\sigma_k)$. The single dot notation ``$\cdot$'' indicates that the decomposition is a splitting, and the subpaths $\sigma_1,\ldots,\sigma_k$ are called the \emph{terms} of the splitting. In other words, a decomposition of $\sigma$ into a concatenation of subpaths is a splitting of $\sigma$ if one can tighten the image of $\sigma$ under any iterate by simply tightening the images of the subpaths.

Certain paths will be atomic terms of splittings, including the following two types. 

Given a zero stratum $H_i$, a path $\tau$ in $H_i$, and an irreducible stratum $H_j$ with $j>i$, we say that $\tau$ is \emph{$j$-taken} if there exists an edge $E \in H_j$ and $k \ge 1$ such that $\tau$ is a maximal subpath in $H_i$ of the path $f^k_\#(E)$. We also say that $\tau$ is \emph{taken} if it is $j$-taken for some $j$. 

Given two linear edges $E_i, E_j$, a root-free closed Nielsen path $w$, and integers $d_i,d_j$ of the same sign such that $f_\#(E_i) = E_i w^{d_i}$ and $f_\#(E_j) = E_j w^{d_j}$, each path of the form $E_i w^p \overline E_j$ for $p \in \Z$ is called an \emph{exceptional path}.
\end{definition}

\begin{definition} \textbf{Complete Splittings.}
\label{DefCompleteSplitting}
A splitting of a path or circuit $\sigma = \sigma_1 \cdot \ldots \cdot \sigma_k$ is called a \emph{complete splitting} if each term $\sigma_i$ satisfies one of the following: $\sigma_i$ is an edge in an irreducible stratum; $\sigma_i$ is an indivisible Nielsen path; $\sigma_i$ is an exceptional path; there is a zero stratum $H_j$ such that $\sigma_i$ is a maximal subpath of $\sigma$ in $H_j$, and $\sigma_i$ is taken. We say that $\sigma$ is \emph{completely split} if it has a complete splitting.
\end{definition}

\paragraph{Enveloping of zero strata.} Consider a relative train track map $f \from G \to G$ with filtration $\emptyset = G_0 \subset G_1 \subset \cdots \subset G_N=G$ and two strata $H_u,H_r$ with $1 \le u<r \le N$. Suppose that the following hold:
\begin{enumerate}
\item $H_u$ is irreducible.
\item $H_r$ is EG and each component of $G_r$ is noncontractible.
\item Each $H_i$ with $u<i<r$ is a zero stratum that is a component of $G_{r-1}$, and each vertex of $H_i$ has valence $\ge 2$ in $G_r$. 
\end{enumerate}
In this case we say that the zero strata $H_i$ with $u<i<r$ are \emph{enveloped} by the \eg\ stratum $H_r$, and we write $H^z_r = \union_{i=u+1}^r H_i$.

%
%
%
%
%

\begin{definition}[CT --- Completely split relative train track map. \recognition\ Definition~4.7]  
\label{DefCT}
A \ct\ is a \rtt\ with particularly nice properties. Of the nine defining properties, all but the ninth involve terms that are fully defined in this article. The statement of the ninth, which we shall not ever use, involves terms for definitions of which we refer the reader to \recognition.

Consider a relative train track map $f \from G \to G$ with filtration $\emptyset = G_0 \subset G_1 \subset G_K=G$. The following properties define what it means for $f$ to be a \ct:
\begin{enumerate}
\item \textbf{(Rotationless)} Each principal vertex is fixed by $f$ and each periodic direction at a principal vertex is fixed by $Df$ (c.f.\ \recognition\ Definition~3.18).
\item \textbf{(Completely Split)} For each edge $E$ in each irreducible stratum, the path $f(E)$ is completely split. For each taken connecting path $\sigma$ in each zero stratum, the path $f_\#(\sigma)$ is completely split.
\item \textbf{(Filtration)} The filtration $\emptyset = G_0 \subset G_1 \subset \cdots \subset G_N = G$ is reduced. For each~$i$ there exists $j \le i$ such that $G_i = \core(G_j)$.
\item \textbf{(Vertices)} The endpoints of all indivisible Nielsen paths are vertices. The terminal endpoint of each nonfixed \neg\ edge is principal.
\item \textbf{(Periodic Edges)} Each periodic edge is fixed and each endpoint of a fixed edge is principal. If $E = H_r$ is a fixed edge and $E$ is not a loop then $G_{r-1}$ is a core graph and both ends of $E$ are contained in $G_{r-1}$.
\item \label{ItemZeroStrata}
\textbf{(Zero Strata)} Each zero stratum $H_i$ is enveloped by some \eg\ stratum $H_r$, each edge in $H_i$ is $r$-taken, and each vertex in $H_i$ is contained in $H_r$ and has link contained in $H_i \union H_r$.
\item \textbf{(Linear Edges)} For each linear edge $E_i$ there exists a closed root-free Nielsen path $w_i$ such that $f(E_i) = E_i w_i^{d_i}$ for some $d_i \ne 0$. If $E_i,E_j$ are distinct linear edges, and if the closed path $w_i$ is freely homotopic to the closed path $w_j$ or its inverse, then $w_i=w_j$ and $d_i \ne d_j$.
\item \label{ItemNEGNielsenPaths}
\textbf{(\neg\ Nielsen Paths)} If $H_i = \{E_i\}$ is an \neg\ stratum and $\sigma$ is an indivisible Nielsen path of height $i$ then $E_i$ is linear and, with $w_i$ as in (Linear Edges), there exists $k \ne 0$ such that $\sigma = E_i w^k_i \overline E_i$.
\item\label{ItemEGNielsenPaths}
\textbf{(\eg\ Nielsen Paths)} If $H_i$ is an \eg\ stratum and $\rho$ is an indivisible Nielsen path of height $i$ then $f \restrict G_i = \theta \composed f_{i-1} \composed f_i$ where: the map $f_i \from G_i \to G^1$ is a composition of proper extended folds defined by iteratively folding $\rho$; the map $f_{i-1} \from G^1 \to G^2$ is a composition of folds involving edges in $G_{i-1}$; and $\theta \from G^2 \to G_i$ is a homeomorphism.
\end{enumerate}
\end{definition}

The main existence theorem for \ct s is the following. Given a nested sequence $\C = (\F_1 \sqsubset \cdots \sqsubset \F_L)$ of free factor systems and a topological representative $f \from G \to G$ with $f$-invariant filtration $\emptyset = G_0 \subset G_1 \subset G_K=G$, we say that $f$ \emph{realizes} $\C$ if for each $\ell = 1,\ldots,L$ there exists $k_\ell \in 1,\ldots,K$ such that $G_{k_\ell}$ is a core subgraph that realizes $\F_\ell$.


\begin{theorem}[\recognition\ Theorem~4.28]
\label{TheoremCTExistence}
For each rotationless $\phi \in \Out(F_n)$ and each nested sequence $\C$ of $\phi$-invariant free factor systems, there exists a \ct\ $f \from G \to G$ that represents $\phi$ and realizes~$\C$. \qed
\end{theorem}

\subsubsection{Facts about \cts, their Nielsen paths, and their zero strata.}
\label{SectionCTFacts}
For Facts~\ref{FactEGAperiodic}--\ref{FactEdgeToZeroConnector}, let $f \from G \to G$ be a \ct\ with filtration $\emptyset = G_0 \subset G_1 \subset \cdots \subset G_k$, representing a rotationless $\phi \in \Out(F_n)$. 

\smallskip

\begin{fact}\label{FactEGAperiodic} If $H_r$ is an \eg\ stratum then $H_r$ is aperiodic and $G_r$ is a core graph.
\end{fact}

\begin{proof} By \recognition\rc\ Lemma~3.19, $H_r$ contains a principal vertex $v$ whose link contains a principal direction in $H_r$. By (Rotationless), that direction is fixed. The transition matrix of $H_r$ therefore has a nonzero element on the diagonal, which implies that $H_r$ is aperiodic. By (Filtration) and \recognition\rc\ Lemma~2.20~(2), $G_r$ is a core graph.
\end{proof}

\begin{fact}[\recognition, Remark 4.9]\rc\
\label{FactUsuallyPrincipal} A vertex of $G$ whose link contains edges in more than one irreducible stratum is principal. \qed
\end{fact}

\smallskip

\begin{fact}[\cite{FeighnHandel:recognition}, Lemma 4.11]\rc\
\label{FactUniqueCompSp}   
Every completely split path or circuit has a unique complete splitting. \qed
\end{fact}

The next fact is used repeatedly without reference:

\begin{fact}[\cite{FeighnHandel:recognition}, Lemma 4.6]\rc\
\label{FactComplSplitStable}
If $\sigma$ is a path in $G$ with endpoints at vertices and if $\sigma$ is completely split then $f_\#(\sigma)$ is completely split. Moreover, if $\sigma = \sigma_1\ldots \sigma_k$ is a complete splitting then $f_\#(\sigma)$ has a complete splitting that refines $f_\#(\sigma) = f_\#(\sigma_1)\ldots f_\#(\sigma_k)$.\qed
\end{fact}

\smallskip

\begin{fact}
\label{FactEvComplSplit}
If $\sigma$ is a circuit in $G$ or a path in $G$ with endpoints at vertices then $f^k_\#(\sigma)$ is completely split for all sufficiently large~$k \ge 0$. 
\end{fact}

\begin{proof}
When $\sigma$ is a path this is \cite{FeighnHandel:recognition}\rc\ Lemma 4.25. When $\sigma$ is a circuit then, by \BookOne\ Lemma 4.1.2, $\sigma$ may be written as a based closed path $\sigma_1$ which splits at its base point, and so the path $f^k_\#(\sigma_1)$ completely splits for sufficiently large~$k$. By \BookOne\ Lemma 4.1.1 any complete splitting of the path $f^k_\#(\sigma_1)$ gives a splitting of the circuit $f^k_\#(\sigma)$, which is evidently a complete splitting of the circuit.
\end{proof}

\begin{fact}[\cite{FeighnHandel:recognition}, Lemma 4.21]\rc\
\label{FactNEGEdgeImage}
Each nonfixed \neg\ stratum $H_i$ is a single edge $E_i$ which, with its \neg\ orientation, has a splitting $f(E_i) = E_i \cdot u_i$ where $u_i$ is a nontrivial, closed, completely split circuit of height $< i$.  \qed
\end{fact}

\bigskip

\begin{fact}\label{FactPrincipalVertices}
For each periodic vertex $x \in G$, exactly one of the following holds:
\begin{enumerate}
\item $x$ is principal (and therefore fixed); 
\item $x$ is the unique periodic point in its Nielsen class, there are exactly two periodic directions based at $x$, and both of those directions are in the same \eg\ stratum.
\end{enumerate}
\end{fact}

\begin{proof} If $x$ is contained in a topological circle $C$ of periodic points then, applying (Periodic Edges), each edge of $C$ is fixed and (1) holds. If $x$ is contained in no such circle then exactly one of items (1) or (2) holds, by definition of principal.
\end{proof}

\noindent \textbf{Remark:} The above argument shows that circles as Definition~\ref{DefPrincipalVertices}~\pref{ItemBadCircle}, the definition of principal vertices, do not exist in \ct's.

\paragraph{Nielsen paths.}

\begin{fact}[(\recognition\ Lemma~4.13]\rc\
\label{FactPNPFixed}
Every periodic Nielsen path is a (fixed) Nielsen path.
\end{fact}

\begin{fact}\label{FactNielsenCircuit}
If $\sigma$ is an $f_\#$-periodic path or circuit then $\sigma$ is $f_\#$-fixed, $\sigma$ is completely split, and all terms in the complete splitting of $\sigma$ are fixed edges or indivisible Nielsen paths.
\end{fact}

\begin{proof} Suppose $\sigma$ is $f_\#$ periodic, say $f^N_\#(\sigma)=\sigma$. By Fact~\ref{FactEvComplSplit}, $\sigma = f^{kN}_\#(\sigma)$ is completely split for sufficiently large $k$. By Fact~\ref{FactComplSplitStable}, applied to $\sigma = f^N_\#(\sigma)$, all terms in the complete splitting of $\sigma$ are $f_\#$-periodic Nielsen paths. By Fact~\ref{FactPNPFixed}, all terms are Nielsen paths, and so $\sigma$ is $f_\#$-fixed.
\end{proof}

Our next few facts are concerned with indivisible Nielsen paths of \eg\ height.

\begin{fact}
\label{FactEGNPUniqueness}
For each \eg\ stratum~$H_r$, up to reversal there is at most one indivisible Nielsen path $\rho_r$ of height~$r$ (\recognition\rc\ Corollary~4.19). Moreover, if $\rho_r$ exists then $\rho_r$ may be written uniquely as a concatenation $\rho = \alpha\beta$ where $\alpha$ and $\beta$ are $r$-legal paths, each begins and ends with edges of $H_r$, and the turn $\{Df(\bar\alpha),Df(\beta)\}$ is degenerate (\cite{BestvinaHandel:tt} Lemma~5.11).
\qed\end{fact}

\begin{fact}
\label{FactNielsenBottommost} If $H_r$ is an \eg\ stratum, if there exists an indivisible Nielsen path $\rho_r$ of height~$r$, and if $\rho_r = a_0 b_1 a_1 b_2 \cdots a_{k-1} b_k$ is the decomposition into maximal subpaths $a_i$ in $H_r$ and $b_i$ in $G_{r-1}$, then:
\begin{enumerate}
\item \label{ItemNoZeroStrata} No zero stratum is enveloped by $H_r$.
\item Each $b_i$ is a Nielsen path.
\item \label{ItemBottommostEdges}
For each edge $E \subset H_r$ and each $k \ge 0$, the path $f^k_\#(E)$ splits into terms each of which is an edge of $H_r$ or one of the Nielsen paths $b_i$ in the decomposition of $\rho_r$. Furthermore, each term in the complete splitting of $f^k_\#(E)$ is an edge of $H_r$, a fixed edge, or an indivisible Nielsen path.
\end{enumerate}
\end{fact}

\begin{proof} Lemma~4.24 of \recognition\ implies \pref{ItemNoZeroStrata}, (2) and the first statement of  \pref{ItemBottommostEdges} for the case $k=1$; the case $k >1$ follows by induction.  The furthermore part of  \pref{ItemBottommostEdges} follows from Corollary~4.12 of \recognition.\rc
\end{proof}


\begin{fact}\label{FactEGNielsenCrossings}
If $H_r$ is an \eg\ stratum, and if $\rho_r$ is an indivisible Nielsen path of height $r$, then $\rho_r$ crosses each edge of $H_r$ at least once, the initial oriented edges of $\rho_r$ and $\bar\rho_r$ are distinct oriented edges of $H_r$, and:
\begin{enumerate}
\item\label{ItemEGNielsenNotClosed}
$\rho_r$ is not closed if and only if it crosses some edge of $H_r$ exactly once, and in this case:
\begin{enumerate}
\item \label{ItemEGParageometricOneInterior}
At least one endpoint of $\rho_r$ is not in $G_{r-1}$.
\item \label{ItemNoClosed} There does not exist a height~$r$ fixed conjugacy class. In particular, there does not exist a closed, height~$r$ Nielsen path.
\item \label{ItemNongeomFFS}
There exists a proper free factor system $\F$ such that for each line $\ell \in \B$, $\ell$ is carried by $\F$ if and only if the realization of $\ell$ in $G$ is a concatenation of edges of $G \setminus H_r$ and copies of $\rho_r$.
\end{enumerate}
\item\label{ItemEGNielsenClosed}
$\rho_r$ is closed if and only if it crosses each edge of $H_r$ exactly twice, and in this case:
\begin{enumerate}
\item \label{ItemEGNielsenPointInterior}
The endpoint of $\rho_r$ is not in $G_{r-1}$.
\item \label{ItemFewClosed} The only height~$r$ fixed conjugacy classes are those represented by $\rho_r$, its inverse, and their iterates.
\end{enumerate}
\end{enumerate}
\end{fact}

\begin{proof} The statement about initial oriented edges is found in \recognition\ Corollary~4.19 eg-(i),\rc\ which derives from \BookOne\ Theorem~5.1.5 \hbox{eg-(i)}. The statements regarding numbers of edge crossings of $H_r$ made by $\rho_r$ are found in \BH\ Theorem~5.15 but with a hypothesis saying ``$f$ is stable'', but all that is needed for the proof of Theorem~5.15 to go through are the properties of folds contained in (EG Nielsen Paths); see also \BookOne\ Lemma~5.2.5 for the same statements but with a hypothesis saying that ``$f$ is $\F$-Nielsen minimized''. 

Item~\pref{ItemNongeomFFS} is just Lemma~5.1.7 of \BookOne.

If $\rho_r$ crosses some edge of $H_r$ exactly once then the hypotheses of \BookOne\ Lemma~5.1.7 hold, the conclusion of which says that $\rho_r$ is not closed and that at most one of its endpoints is in a noncontractible component of $G_{r-1}$, but by Facts~\ref{FactNielsenBottommost} and~\ref{FactContrComp}, every component of $G_{r-1}$ is noncontractible. If $\rho_r$ crosses every edge of $H_r$ exactly twice, it follows by \BookOne\ Proposition~5.3.1 that $\rho_r$ is a closed path which begins and ends with distinct edges and that its endpoint is not contained in $G_{r-1}$.

It remains to prove items~\pref{ItemNoClosed} and~\pref{ItemFewClosed}. Consider a height~$r$ circuit~$c$ fixed by~$f_\#$. By Fact~\ref{FactEvComplSplit}, $c$ completely splits. Each term of the complete splitting must be a fixed edge of height~$<r$ or an indivisible Nielsen path of height~$\le r$, including at least one term of height~$r$ which, by Fact~\ref{FactEGNPUniqueness}, must be $\rho_r$ or its inverse. Up~to switching orientations of $c$ and of $\rho_r$, we may assume that $c_i=\rho_r$ is a term of $\sigma$ with terminal endpoint $x \not\in G_{r-1}$. Consider the term $c_{i+1}$ following $c_i$, with initial endpoint~$x$. Since no fixed edge in $G_r$ is incident to $x$, $c_{i+1}$ cannot be a fixed edge. If $\rho_r$ is not closed then the only candidate for $c_{i+1}$ is the inverse of $\rho_r$, which is not allowed to follow $\rho_r$ in a complete splitting. If $\rho_r$ is closed then the only allowed candidate for $c_{i+1}$ is another copy of~$\rho_r$, and by induction $c$ is an iterate of $\rho_r$.
\end{proof}

Our last fact about Nielsen paths is concerned with zero strata and with \neg\ strata that are \emph{superlinear}, meaning not fixed and not linear.

\begin{fact}
\label{FactNoSuperlinearNielsen}
If $H_r$ is a zero stratum or an \neg\ superlinear stratum then no Nielsen path crosses any edge of $H_r$.
\end{fact}

\begin{proof} Suppose some Nielsen path crosses an edge $E$ of $H_r$, and let $i$ be the minimal height of all such paths. By Fact~\ref{FactNielsenCircuit}, there is an indivisible Nielsen path $\rho_i$ of height $i$ that crosses~$E_r$. Note that $H_i$ is note a zero stratum for otherwise, by (Zero Strata), it would follows that $\rho_i$ is contained entirely in $H_i$, but $f^k(H_i) \intersect H_i = \emptyset$ for sufficiently large $k$, contradicting that $f^k_\#(\rho_i)=\rho_i$. 

\textbf{Case 1: $i=r$.} We have shown that $H_i=H_r$ is not a zero stratum, and if $H_r$ is \neg\ superlinear then by applying (\neg\ Nielsen Paths) we obtain a contradiction. 

\textbf{Case 2: $i > r$.} If $H_i = \{E_i\}$ is \neg\ then by (\neg\ Nielsen Paths) it follows that $H_i$ is linear and that $\rho_i = E_i w_i^s \overline E_i$ where $w_i$ is a closed Nielsen path of height $<i$, but then $w_i$ crosses $E_r$, contradicting minimality of height. If $H_i$ is \eg\ then by Fact~\ref{FactNielsenBottommost} we have $\rho_i = a_0 b_1 a_1 b_2 \cdots a_{k-1} b_k$ where the $a$'s are paths in $H_i$ and the $b$'s are Nielsen paths of height $<i$, but then one of the $b$'s must cross $E_r$, again contradicting minimality of height. We have already ruled out the possibility that $H_i$ is a zero stratum.
\end{proof}

\paragraph{Zero strata.}

\begin{fact} \label{FactContrComp}
For each filtration element $G_r$ the following are equivalent: 
\begin{enumerate}
\item \label{ItemHasContrComp}
$G_r$ has a contractible component; 
\item \label{ItemIsZeroStrat}
$H_r$ is a zero stratum;
\item \label{ItemIsContrComp}
$H_r$ is a contractible component of $G_r$.
\end{enumerate}
\end{fact}

\begin{proof}
The equivalence of \pref{ItemHasContrComp} and~\pref{ItemIsZeroStrat} is Lemma~4.15 of \recognition.\rc\ That \pref{ItemIsZeroStrat} implies \pref{ItemIsContrComp} is a consequence of (Zero Strata) which implies that $H_r$ is a component of $G_r$, and the fact that $f_\#$ is a bijection on circuits. That \pref{ItemIsContrComp} implies \pref{ItemHasContrComp} is obvious.
\end{proof}

\begin{fact}\label{FactEdgeToZeroConnector}
For any edge $E \subset G$, if $f_\#(E)$ is contained in a zero stratum~$H_t$ enveloped by the \eg\ stratum $H_s$ then $E$ is an edge in some other zero stratum $H_{t'} \ne H_t$ that is also enveloped by $H_s$.
\end{fact}

\begin{proof} Since the path $f_\#(E)$ is completely split and contained in the zero stratum $H_t$, its complete splitting must have just the single term $f_\#(E)$ which must be a taken connecting path of $H_s$. If the oriented edge $E$ is contained in an irreducible stratum $H_i$ then some term in the complete splitting of $f_\#(E)$ is an edge of $H_i$, contradiction. It follows that $E$ is contained in some zero stratum $H_{t'}$ enveloped by some \eg\ stratum~$H_{s'}$. 

By (Zero Strata) applied to $H_{t'}$ there exists an oriented edge $E_0 \subset H_{s'}$ having the same initial vertex as $E$. The paths $f_\#(E_0)$ and $f_\#(E)$ therefore have the same initial vertex, and by (Zero Strata) applied to $H_t$ the link of this vertex is contained in $H_{s} \union H_t$. By RTT-(i) the initial direction of $f_\#(E_0)$ is contained in $H_{s'}$, and so $s=s'$. 

If $H_t = H_{t'}$ then $f(H_t) \intersect H_t \ne \emptyset$, contradicting the definition of a zero stratum.
\end{proof}

\subsubsection{Facts about principal lifts, principal directions, and principal rays.} 
\label{SectionLiftFacts}
From here to the end of the section we fix a rotationless $\phi \in \Out(F_n)$ represented by a \ct\ $f \from G \to G$. We describe several natural structures on $f$ that are associated to principal automorphisms representing $\phi$: the definition of principal lifts is taken directly from \recognition; and the definitions of principal directions and principal rays are formulated based on results from \recognition.

Recall the natural bijection between the set of automorphisms $\Phi$ representing $\phi$, and the set of lifts $\ti f \from \wt G \to \wt G$ of $f \from G \to G$ via the universal covering map $\wt G \mapsto G$, where $\ti f$ corresponds to $\Phi$ if and only if the extension $\hat f$ to $\bdy\wt G \approx \bdy F_n$ equals the extension $\wh\Phi$ to $\bdy F_n$.

\begin{definition}[\recognition\ Definition~3.1] Given a lift $\ti f \from \wt G \to \wt G$ with corresponding automorphism $\Phi$ representing $\phi$, we say that $\ti f$ is a \emph{principal lift}\index{principal lift} if $\Phi$ is a principal automorphism.
\end{definition}

\begin{fact}[\recognition\ Corollary 3.17, Corollary 3.22, Corollary 3.27]\rc\
\label{FactPrincipalLift}
A lift $\ti f \from \wt G \to \wt G$ is principal if and only if there is a principal vertex $v \in G$ and a lift $\ti v \in \wt G$ such that $\ti f(\ti v) = \ti v$. \qed
\end{fact}

\begin{definition}[Principal directions] 
\label{DefPrincipalDirection}
Consider an oriented edge $E \subset G$ whose initial vertex is a principal vertex of $f$ and whose initial direction is fixed by $f$ but $E$ is not a fixed edge of $f$. If $E$ is nonlinear then we say that $E$ is a \emph{principal direction of $f$ in $G$}.\index{principal direction} Given a principal lift $\ti f \from \wt G \to \wt G$, and given an oriented edge $\wt E \subset \wt G$ with initial vertex $\ti v \in \Fix(\ti f)$, we say that $\wt E$ is a \emph{principal direction of $\ti f$ in $\wt G$} if its projection $E \subset G$ is a principal direction of $f$ in $G$. We also say that $E$ (or $\wt E$) is an \neg\ or \eg\ principal direction depending on the nature of the stratum of $G$ containing $E$.
\end{definition}

The concept of a ``principal ray'' is based on Lemmas~3.26 and~4.36 of \recognition\ which we compile together in the statement of Fact~\ref{FactSingularRay} to follow, after which we use it to formally define principal rays and then give its proof. The hypotheses on $\Phi$ and $\ti f$ in Fact~\ref{FactSingularRay} are satisfied when $\Phi$ is principal, equivalently $\ti f$ is a principal lift. While most of its numerous applications are in the principal case, there are a few applications in the nonprincipal case. 

%
%
%
%

\begin{fact}\label{FactSingularRay} 
Let $\Phi \in \Aut(F_n)$ represent $\phi$ with associated lift $\ti f \from \wt G \to \wt G$, suppose that $\Fix(\ti f) \ne \emptyset$, and suppose that there does not exist a nontrivial $\gamma \in F_n$ such that the set $\Fix(\wh\Phi) = \Fix(\hat f)$ is equal to $\Fix(\hat\gamma)$. For each $\ti v \in \Fix(\ti f)$ and each oriented edge $\wt E \subset \wt G$ with initial vertex $\ti v$ and with fixed initial direction, let $v,E \subset G$ be the projections of $\ti v, \wt E$. With this notation, the following hold:
\begin{enumerate}
\item \label{ItemFixedDirectionRay}
For each $\ti v$, $\wt E$ as above such that $E = H_r$ is a fixed edge, there exists a ray $\wt R \subset \wt G$ with initial edge $\wt E$ that converges to some $\ti\xi \in \Fix_N(\wh\Phi)$ and that projects to a ray in~$G_r$.
\item \label{ItemDirectionToAttractor}
For each $\ti v$, $\wt E$ as above such that $E$ is not a fixed edge, letting $f(E) = Eu$, letting $\ti u \subset \wt G$ be the lift of $u$ for which $\ti f(\wt E) = \wt E \ti u$, and letting $H_r$ be the stratum containing $E$, the following hold. There exists a properly nested sequence of oriented paths $\wt E \subset \ti f(\wt E) \subset\ti f^2_\#(\wt E) \subset \ldots$, each an initial segment of the previous, whose union is a ray $\wt R \subset \wt G$ that contains no fixed point other than $\ti v$ and that converges to some $\ti\xi \in \Fix_N(\wh\Phi)$; we say in this case that $\wt E$ iterates to $\ti\xi$. Furthermore:
\begin{enumerate}
\item \label{ItemComplSplitRay}
There exists a splitting $f(\wt E) = \wt E \cdot \ti u$ and an induced splitting
$$\wt R = \underbrace{\wt E \cdot \ti u \cdot \ti f_\#(\ti u) \cdot \ti f^2_\#(\ti u) \cdot \ldots \cdot \ti f^k_\#(\ti u)}_{\ti f^k_\#(\wt E)} \cdot \ti f^{k+1}_\#(\ti u) \cdot \ldots \qquad\qquad (*)
$$
\item\label{ItemEGPrincipalRay}
If $E$ is not an \neg-linear edge then $\ti\xi \in \Fix_+(\wh\Phi)$. Furthermore, if $H_r$ is \eg\ then the weak accumulation set of $\xi$ is the unique attracting lamination of height~$r$.
\item\label{ItemRayEndsAtAttr} 
If $E$ is an \neg-linear edge then $\ti\xi \in \bdy\Fix(\Phi)$.
\item\label{ItemRayHeight}
$\wt R$ projects to a ray in $G_r$, and if $H_r$ is an \neg-stratum then $\wt R \setminus \wt E$ projects to a ray in $G_s$ where $s$ is the height of the circuit~$u$.
\end{enumerate}
\item \label{ItemAttractorToDirection}
For each $\ti\xi \in \Fix_+(\wh\Phi)$ the following hold:
\begin{enumerate}
\item\label{ItemAttractorToDirectionExists}
There exists $\ti v \in \Fix(\ti f)$ and $\wt E$ such that $\wt E$ iterates to $\ti\xi$ as above.
\item\label{ItemAttractorToDirectionUnique}
For any ray $\wt R \subset \wt G$ with ideal endpoint $\ti\xi$, finite endpoint $\ti v$, and initial oriented edge $\wt E$, if $\Fix(\ti f) \intersect \wt R = \{\ti v\}$ then the initial direction of $\wt E$ is fixed by $\ti f$, $\wt E$ iterates to $\ti\xi$, and the projected edge $E$ is neither fixed nor linear.
\end{enumerate}
\end{enumerate}
\end{fact}

\begin{definition}[Principal rays and linear rays] 
\label{DefSingularRay} 
In the context of Fact~\ref{FactSingularRay}~\pref{ItemDirectionToAttractor} where $\wt E \subset \wt G$ is a non-fixed edge with fixed initial direction, we say that $\wt R$ is \emph{the ray in $\wt G$ generated by $\wt E$}, and that \emph{$\wt R$ represents $\ti\xi$}, and that \emph{$\wt E$ iterates to $\ti\xi$}. If $\wt E$ is a linear edge then we say that $\wt R$ is a \emph{linear ray}, and if $\wt E$ is a principal direction then $\wt R$ is a \emph{principal ray}. We use similar language downstairs in $G$ regarding the edge $E$ and the ray $R = E \cdot u \cdot f_\#(u) \cdot f^2_\#(u) \cdot \ldots$. 

Note that if $\Phi \in P(\phi)$ with associated principal lift $\ti f$ then by definition we have a bijection between the set of principal rays of $\ti f$ in $\wt G$ and the set of principal directions of $\ti f$ in $\wt G$, which associates to each principal ray its initial direction. Adding in Fact~\ref{FactSingularRay} items~\pref{ItemRayEndsAtAttr} and~\pref{ItemAttractorToDirection} we also have a surjection from the set of principal rays of $\ti f$ in $\wt G$ to the set $\Fix_+(\wt\Phi)$, which associates to each principal ray $\wt R$ its endpoint. Similarly, we have a bijection between principal rays of $f$ in $G$ and principal directions of $f$ in $G$.
\end{definition}

\smallskip\textbf{Remark.} In Section~\refRK{SectionAsymptotic} we shall focus on those principal rays which are generated by \neg\ principal directions, and we shall show that such rays are invariants of $\phi$ in that they can be characterized independently of relative train tracks. Such rays, in the appropriate context, are referred to as ``eigenrays'' of $\phi$.

\begin{proof}[Proof of Fact~\ref{FactSingularRay}] Item~\pref{ItemFixedDirectionRay} is proved in \recognition\ Lemma~3.26. The main clause of item~\pref{ItemDirectionToAttractor} is exactly  \recognition\ Lemma~4.36~(1), whose proof implicitly proves item~\pref{ItemComplSplitRay} as well. Also, item~\pref{ItemAttractorToDirectionExists} is the same as \recognition\ Lemma~4.36~(2), and the proof of the latter given in \recognition\ establishes item~\pref{ItemAttractorToDirectionUnique} as well. We note that while Lemma~4.36 assumes that $\ti f$ is principal, its proof holds under our present weaker hypotheses on~$\ti f$. Item~\pref{ItemRayHeight} is clear from the construction of~$\wt R$; see also \recognition\ Lemma~3.26~(2) and~(3).

The ``Furthermore'' clause of \pref{ItemEGPrincipalRay} follows from \recognition\ Lemma~3.26~(2). For the rest of the proof of \pref{ItemEGPrincipalRay} and \pref{ItemRayEndsAtAttr}, $E$ is linear if and only if $f_\#(u)=u$, and we must prove that this holds if and only if $\ti\xi \not\in \Fix_+(\wh\Phi)$. 

If $E \subset G$ is linear then $f^k_\#(u) = u$ for all $k \ge 1$. By construction $\ti f$ takes the initial vertex of $\ti u$ to its terminal vertex and so, letting $T_{\ti u}$ be the covering transformation which does the same to those two vertices, it follows that $\ti f$ commutes with $T_{\ti u}$, that $T_{\ti u}(\ti\xi)=\ti\xi$, and that $\ti\xi \in \bdy \Fix(\wh\Phi)$. It follows by Fact~\ref{LemmaFixPhiFacts} that $\ti\xi \not\in \Fix_+(\wh\Phi)$.

Assuming that $\ti\xi \not\in \Fix_+(\wh\Phi)$, and using the hypotheses on $\ti f$, it follows that $\ti\xi \in \bdy\Fix(\Phi)$ by applying Fact~\ref{LemmaFixPhiFacts} and Definition~\ref{DefPrinicipalAndRotationless} (which derive from \cite{GJLL:index}). In particular the subgroup $\Fix(\Phi)$ is nontrivial. The nonempty set $\Fix(\ti f)$ is $\Fix(\Phi)$ invariant and it accumulates on $\bdy\Fix(\Phi)$. Since the ray $\wt R$ ends at $\ti\xi \in \bdy\Fix(\Phi)$ it follows that for $x \in \wt R$ there is a uniform upper bound to $d(x,\Fix(\ti f))$ and so there is a uniform upper bound to $d(x,\ti f(x))$. Applying \pref{ItemComplSplitRay} it follows that there is a uniform upper bound to the length of $f^k_\#(\ti u)$, which implies that $E$ is linear.
\end{proof}

\subparagraph{Uniqueness of principal rays.} From item~\pref{ItemDirectionToAttractor} of Fact~\ref{FactSingularRay}~\pref{ItemDirectionToAttractor} we see that 
each principal direction iterates to a unique attracting fixed point. In the other direction, from item~\pref{ItemAttractorToDirection} we see that each attracting fixed point is iterated to by some principle direction. Although uniqueness need not hold, in the \eg\ case it almost holds, as we show next in Lemma~\ref{LemmaPrincipalRayUniqueness} which also covers the nonprincipal case. We will need this result in \PartThree\ in our study of weak attraction.

For the statement, consider an indivisible Nielsen path of \eg\ height having the form $\rho=\alpha_1 \bar\alpha_2$ as described in Fact~\ref{FactEGNPUniqueness}. That description was taken from \recognition\ Lemma~2.11 and we need one additional fact from that same source, namely that each of $\alpha_1,\alpha_2$ is an initial segment of the ray generated by its fixed initial direction. 

\begin{lemma}\label{LemmaPrincipalRayUniqueness}
Assume the same hypotheses as Fact~\ref{FactSingularRay}, and let $H_r$ be an \eg\ stratum.
\begin{enumerate}
\item\label{ItemNPExchangeSuff}
If $\wt E$ is a height~$r$ oriented edge with $\ti f$-fixed initial direction generating a ray $\wt R$, and if $\ti\rho = \alpha_1 \bar\alpha_2$ is a height~$r$ indivisible Nielsen path decomposed as in Fact~\ref{FactEGNPUniqueness}, and if $\wt E$ and $\ti\rho$ have the same initial direction, then $\wt R \intersect \ti\rho = \alpha_1$, and the ray $\wt R' =\alpha_2 \union (\wt R - \alpha_1)$ is generated by the initial direction of $\alpha_2$; these rays $\wt R,\wt R'$ clearly represent the same point of $\Fix_+(\wh\Phi)$.

We say in this case that $\wt R'$ is obtained from $\wt R$ by exchanging across the Nielsen path~$\ti\rho$. 
\item\label{ItemTwoPrincipalRays}
Given distinct \eg\ edges $\wt E_1 \ne \wt E_2$ of height~$r$ with fixed initial directions and generating rays $\wt R_1,\wt R_2$:
\begin{enumerate}
\item\label{ItemNPExchangeNecess}
If $\wt R_1,\wt R_2$ represent the same point of $\Fix_+(\wh\Phi)$ then $\wt R_2$ is obtained from $\wt R_1$ by exchanging across some height~$r$ indivisible Nielsen path for $\ti f$.
\item\label{ItemNPPrincipalRaysDisjoint}
If $\wt R_1,\wt R_2$ represent distinct points of $\Fix_+(\wh\Phi)$ then $\interior(\wt R_1) \intersect \interior(\wt R_2) = \emptyset$.
\end{enumerate}
\end{enumerate}
\end{lemma}

\begin{proof} To prove~\pref{ItemNPExchangeSuff}, as said just before the lemma we have $\alpha_1 \subset \wt R$. Since $\wt R$ is $r$-legal and the terminal directions of $\alpha_1,\alpha_2$ form an $r$-illegal turn it follows that $\wt R \intersect \rho = \alpha_1$. Since $\alpha_1,\alpha_2$ are $r$-legal and $\ti f$ is (the lift of) a relative train track map, it follows that there is a nested sequence of $r$-legal paths $\gamma_1 \subset \gamma_2 \subset\cdots$, each an initial segment of the next and with lengths going to $+\infinity$, such that for $i=1,2$ we have $\ti f^k_\#(\alpha_i)=\alpha_i\gamma_k$. It follows that $\wt R = \alpha_1 \bigcup \union_k \gamma_k$ and so $\wt R' = \alpha_2 \union (\wt R - \alpha_1)= \alpha_2 \bigcup \union_k \gamma_k$ is the ray generated by the initial direction of $\alpha_2$. This proves~\pref{ItemNPExchangeSuff}.

Consider $\wt E_i, \wt R_i$, as in \pref{ItemTwoPrincipalRays}, $i=1,2$. Let $\xi_i \in \Fix_+(\wh\Phi)$ be represented by $\wt R_i$. Let $x_i \in \Fix(\ti f)$ be the initial points of $\wt E_i$, and note that neither of $x_1,x_2$ lies in the interior of $\wt R_1$ or $\wt R_2$ (Fact~\ref{FactSingularRay}~\pref{ItemDirectionToAttractor}). If $x_1=x_2$ then since $\wt E_1 \ne \wt E_2$ it follows that $\wt R_1 \intersect \wt R_2 = \{x_1\}=\{x_2\}$ and $\xi_1 \ne \xi_2$, and there is nothing to prove. If $x_1 \ne x_2$ consider the path $\rho = [x_1,x_2]$, a nontrivial Nielsen path of $\ti f$. If $\interior(\wt R_1) \intersect \interior(\wt R_2) \ne \emptyset$ it follows that for $i=1,2$ the initial direction of $\alpha_i$ generates the ray $\wt R_i$. By exchanging $\wt R_1$ across $\alpha_1$ we obtain a ray $\wt R'_2$ with initial segment $\alpha_2$, and since both $\wt R'_2$ and $\wt R_2$ are generated by the initial direction of $\alpha_2$ it follows that $\wt R_2 = \wt R'_2$. This proves both~\pref{ItemNPExchangeNecess} and~\pref{ItemNPPrincipalRaysDisjoint}.
\end{proof}

\subparagraph{Weak accumulation of attracting fixed points.}
The following lemma will be used in \PartThree.

\begin{lemma}\label{LemmaFixPlusAccumulation}
For each $\Phi \in P(\phi)$ and $P \in \Fix_+(\Phi)$ there is a conjugacy class $[a]$ that is weakly attracted to every line in the weak accumulation set of $P$. 
\end{lemma}

\begin{proof} Let $\ti f \from \wt G \to \wt G$ be the principal lift corresponding to $\Phi$. Applying Fact~\ref{FactSingularRay}, let $\wt R$ be a principal ray with endpoint~$P$, let $\wt E$ be its initial oriented edge, a principal direction of~$\ti f$. Consider the splitting of $\wt R$ expressed in item~\pref{ItemComplSplitRay} of Fact~\ref{FactSingularRay}. Since $u$ is not fixed by $f_\#$ we have $\Length(f^k_\#(u)) \to +\infinity$ as $k \to +\infinity$. For sufficiently large $k$, say $k \ge k_0$, the paths $f^k_\#(u)$ are circuits whose initial directions are independent of $k$, as are their terminal directions. Let $[a]$ be the conjugacy class represented by the circuit $\sigma = f^{k_0}_\#(u) \cdot f^{k_0+1}_\#(u)$, and we have a circuit $f^i_\#(\sigma) = f^{k_0+i}_\#(u) \cdot f^{k_0+i+1}_\#(u)$. 

Given a line $\ell$ in $G$ which is in the accumulation set of $P$, and given a finite subpath $\alpha$ of~$\ell$, let $U_\alpha$ be the set of lines whose realization in $G$ contains $\alpha$ as a subpath. Choosing a lift $\ti\alpha \subset G'$, the ray $\wt R$ contains infinitely many translates of $\ti\alpha$. The length of these translates are fixed while the lengths of the paths $\ti f^j_\#(u)$ go to $+\infinity$ with $j$, so infinitely many of these paths intersect translates of $\ti\alpha$ in $\wt R$, and hence infinitely many of the paths $\ti f^j_\#(u) \cdot \ti f^{j+1}_\#(u)$ contain translates of $\ti\alpha$. It follows that infinitely many iterates $f^i_\#(\sigma)$ contain $\alpha$ as a subpath and so are in $U_\alpha$. Since the $U_\alpha$ form a weak neighborhood basis of $\ell$, we have proved that $f^i_\#(\sigma)$ weakly accumulates on $\ell$, i.e.\ $[a]$ is weakly attracted to $\ell$. 
\end{proof}

\subsubsection{Properties of \eg\ strata.} 

Recall Fact~\ref{FactEvComplSplit}, which derives from \recognition\rc\ Lemma~4.25, saying that for any \ct\ $f \from G \to G$ and any $\sigma$ which is either a path with endpoints at vertices or a circuit, there exists some $k \ge 1$ so that $f^k(\sigma)$ completely splits. We prove two lemmas with more quantitative control on the exponent $k$, achieved at the expense of discarding ``completeness'' of the splitting, using instead a somewhat coarser kind of splitting. The lemmas are stated for general relative train tracks---\cts\ are not needed---and for circuits whose height is equal to the height of some \eg\ stratum. Their proofs derive from \BookOne\ Section~4.2, although the first lemma can be derived from the \eg\ case of the proof of \recognition\rc\ Lemma~4.25 by adding a little quantitative information.

\begin{lemma}
\label{LemmaEGPathSplitting}
Let $f \from G \to G$ be a relative train track and $H_r \subset G$ an aperiodic \eg\ stratum. For each $L$ there exists an integer $k = k_L \ge 0$ such that for each height $r$ path or circuit $\sigma$ whose endpoints, if any, are vertices, if $\sigma$ has length $\le L$ then $f^k_\#(\sigma)$ splits into terms each of which is an edge or indivisible periodic Nielsen path of height~$r$ or a path in~$G_{r-1}$.
\end{lemma}

\begin{proof} Since there are only finitely many circuits and paths with endpoints at vertices that have length $\le L$, it suffices to consider just a single $\sigma$ and find a $K$ such that $f^K_\#(\sigma)$ satisfies the conclusions. 

Let $P_r$ be the set of paths $\rho$ of height $r$ such that for each $k \ge 0$ the path $f^k_\#(\rho)$ begins and ends with an edge in $H_r$ and has exactly one $r$-illegal turn, and such that the number of edges in $f^k_\#(\rho)$ is bounded independently of $k$. By \BookOne\ Lemma 4.2.5, $P_r$ is a finite $f_\#$ invariant set, from which it follows that there exists $K_1$ such that for each $\rho \in P_r$ the path $f^{K_1}_\#(\rho)$ is a periodic Nielsen path, which therefore splits into indivisible periodic Nielsen paths of height $r$ and paths in $G_{r-1}$. 

Since $f_\#$ applied to an $r$-legal path is $r$-legal, and since a subpath of an $r$-legal path is $r$-legal, it follows that the number of maximal $r$-legal paths in $f^i_\#(\sigma)$ is a nonincreasing function of~$i$. The number of illegal turns of $f^i_\#(\sigma)$ in $H_r$ is therefore nonincreasing, and so is constant for sufficiently large $i$, say $i \ge K_2$. Applying Lemma~4.2.6 of \BookOne\ it follows that $f^{K_2}_\#(\sigma)$ splits into subpaths each of which is $r$-legal or is one of the paths in $P_r$. Letting $K=K_1+K_2$ it follows that $f^K_\#(\sigma)$ splits into $r$-legal paths and indivisible periodic Nielsen paths of height $r$. Since an $r$-legal path splits into edges in $H_r$ and paths in $G_{r-1}$, this finishes the proof.
\end{proof}

We also need a form of Lemma~\ref{LemmaEGPathSplitting} whose conclusion has a stronger uniformity. This will be used in the analysis of limit trees in Section~\refRK{SectionLimitTrees}.

\begin{lemma}
\label{LemmaEGUnifPathSplitting}
Let $f \from G \to G$ be a relative train track and $H_r \subset G$ an aperiodic \eg\ stratum. For each $M$ there exists an integer $d \ge 0$ such that for each height $r$ path or circuit $\sigma$ with endpoints, if any, at vertices, if $\sigma$ contains at most $M$ edges in the subgraph $H_r$ then $f^d_\#(\sigma)$ splits into terms each of which is an edge or indivisible periodic Nielsen path of height $r$ or a path in~$G_{r-1}$.
\end{lemma}

\begin{proof} In this proof, all paths and circuits are in $G$ and have endpoints, if any, at vertices.

If a path or circuit $\alpha$ of height $r$ has a splitting satisfying the conclusions of the lemma, the terms being edges and indivisible Nielsen paths of height $r$ and paths in $G_{r-1}$, then $f_\#(\alpha)$ also has a splitting satisfying the conclusions. We are therefore free to increase the exponent on $f_\#$ as needed in this proof, which we shall do without mention.

\subparagraph{Case 1: $\sigma$ is a height $r$ path.} We prove the lemma in this case by induction on~$M$. If $\sigma$ is $r$-legal, in particular if $M=1$, we can take $d=0$. Assume by induction that the exponent $D=d_{M-1}$ works when $\sigma$ has $\le M-1$ edges in $H_r$. 

Let $B \ge 0$ be a bounded cancellation constant for $f^D$. Let $L$ be a constant such that if $\beta$ is a path of length $> L$ then the path $f_\#(\beta)$ has length $\ge 2B+1$; the constant $L$ exists because the lift of $f$ to the universal cover of $G$ is a quasi-isometry. Given a path $\alpha \beta \gamma$, if $\beta$ has length $> L$ then not all of $f^D_\#(\beta)$ is cancelled when $f^D_\#(\alpha) f^D_\#(\beta) f^D_\#(\gamma)$ is tightened to $f^D_\#(\alpha \beta \gamma)$: at most $B$ initial edges of $f^D_\#(\beta)$ cancel with $f^D_\#(\alpha)$, and at most $B$ terminal edges cancel with $f^D_\#(\gamma)$. 

Let $d_M$ be the maximum of $D$ and the constant $K=k_{M+ML-L}$ from Lemma~\ref{LemmaEGPathSplitting}. Let $\sigma$ be a height $r$ path with exactly $M$ edges in $H_r$. 

We first reduce to the subcase that $\sigma$ begins and ends with edges in $H_r$. In the case of a general path $\sigma$, we can write $\sigma = \alpha \tau \beta$ where $\alpha,\beta$ are the longest initial and terminal segments in $G_{r-1}$. Knowing that the path $f^{d_M}_\#(\tau)$ satisfies the conclusions, it has a splitting $f^{d_M}_\#(\tau) = \alpha' \cdot \tau' \cdot \beta'$ where $\alpha',\beta'$ are its longest initial and terminal segments in $G_{r-1}$ and where $\tau'$ satisfies the conclusions, and so $f^{d_M}_\#(\sigma) = [f^{d_M}_\#(\alpha) \alpha'] \cdot \tau' \cdot [\beta' f^{d_M}_\#(\beta)]$ has a splitting that satisfies the conclusions, completing the reduction.

Suppose now that $\sigma$ begins and ends with edges in $H_r$. If $\sigma$ has length $\le M+ML-L$ then $h^K_\#(\sigma)$ satisfies the conclusions. If~$\sigma$ has length $>M+ML-L$ then, since $\sigma$ begins and ends with edges in $H_r$, we can write $\sigma = \alpha \beta \gamma$ where $\beta$ in $G_{r-1}$ has length $> L$ and $\alpha,\gamma$ each have between~$1$ and $M-1$ edges in $H_r$. The induction hypothesis implies that $f^D_\#(\alpha)$ and $f^D_\#(\gamma)$ each satisfy the conclusions, so there are splittings
\begin{align*}
f^D_\#(\alpha) &= \alpha' \cdot \beta_1 \\
f^D_\#(\gamma) &= \beta_2 \cdot \gamma'
\end{align*}
where $\beta_1$, $\beta_2$ are each maximal subpaths in $G_{r-1}$, the paths $\alpha',\gamma'$ satisfy the conclusions, and the terminal edge of $\alpha'$ and the initial edge of $\gamma'$ are both in~$H_r$. Our choice of $L$ guarantees that $\beta' =  [\beta_1 f^D_\#(\beta) \beta_2]$ is nontrivial, and so 
\begin{align*}
f^D_\#(\sigma) & = [f^D_\#(\alpha) f^D_\#(\beta) f^D_\#(\gamma)] \\
  &= \alpha' \beta' \gamma'
\end{align*}
This is a splitting of $f^D_\#(\sigma)$, because the turns $\{\bar\alpha',\beta'\}$, $\{\bar\beta',\gamma'\}$ are legal by RTT-(i), and $\alpha',\gamma'$ have splittings that satisfy the conclusions. It follows that $f^D_\#(\sigma)$ has a splitting that satisfies the conclusions.

\subparagraph{Case 2: $\sigma$ is a height $r$ circuit.} Suppose that $\sigma$ contains exactly $M$ edges in~$H_r$. Using the constant $d_M$ from Case~1, let $B'$ be a bounded cancellation constant for $f^{d_M}$. Let $L'$ be a constant such that if $\beta$ is a path of length $>L'$ then the path $f^{d_M}_\#(\beta)$ has length $\ge 2B'+1$. Let $d'_M$ be the maximum of $d_M$ and the constant $K'=k_{M(L'+1)}$ from Lemma~\ref{LemmaEGPathSplitting}. If $\sigma$ has length $\le M(L'+1)$ then $f^{K'}_\#(\sigma)$ satisfies the conclusions. If $\sigma$ has length $>M(L'+1)$ then $\sigma$ has a maximal subpath $\beta$ in $G_{r-1}$ of length $> L'$, with a complementary subpath $\tau$ that begins and ends in $H_r$ and has exactly $M$ edges in $H_r$. Applying Case~1, the path $f^{d_M}_\#(\tau)$ satisfies the conclusions, so there is a splitting 
$$f^{d_M}_\#(\tau) = \beta_1 \cdot \tau' \cdot \beta_2
$$
where $\beta_1,\beta_2$ are maximal subpaths in $G_{r-1}$, $\tau'$ satisfies the conclusions, and $\tau'$ begins and ends with edges of $H_r$. The choice of $L'$ guarantees that $\beta' = [\beta_2 f^{d_M}_\#(\beta) \beta_1]$ is nontrivial, and so $f^{d_M}_\#(\sigma) = \beta' \tau'$. This is a splitting of the circuit $f^{d_M}_\#(\sigma)$, because $\{\bar\beta',\tau'\}$ and $\{\bar\tau',\beta'\}$ are legal by RTT-(i) and $\tau'$ has a splitting that satisfies the conclusions. It follows that $f^{d_M}_\#(\sigma)$ has a splitting that satisfies the conclusions.

\bigskip

In conclusion, we have proved the lemma with $d = \max\{d_M,d'_M\}$.
\end{proof}

\subsection{Properties of Attracting Laminations}
\label{SectionLams} 

In Section~\ref{SectionAttractingLams} we recalled the definition of attracting laminations, which is formulated invariantly without regard to relative train tracks. We also reviewed there several facts which are formulated invariantly (albeit their proofs often depend on relative train tracks). In this section we prove various facts and lemmas about attracting laminations which are explicitly formulated from the relative train track point of view.

\subsubsection{The relation between \eg\ strata and attracting laminations.} Basic to the theory of attracting laminations is the relation between \eg\ strata of relative train track maps and attracting laminations. We describe the basic facts about this relation, including information about free factor supports and about attracting neighborhoods. 

The first fact is compiled from results in \BookOne\ Section~3, from Definition~3.1.5 to Definition~3.1.14. A relative train track map is said to be \emph{\eg-aperiodic} if each of its \eg\ strata is aperiodic.

\begin{fact} 
\label{FactLamsAndStrata} 
For any $\phi \in \Out(F_n)$ and any relative train track representative $f \from G \to G$, $\phi$~acts as the identity on $\L(\phi)$ if and only if $f$ is \eg-aperiodic. If this is the case then there exists a bijection between attracting laminations $\Lambda \in \L(\phi)$ and \eg-strata $H_r \subset G$ such that $\Lambda$ corresponds to $H_r$ if and only if the following equivalent statements hold:
\begin{enumerate}
\item Each generic leaf of $\Lambda$ has height~$r$.
\item $\F_\supp(\Lambda) \sqsubset [\pi_1 G_r]$ and $\F_\supp(\Lambda) \not\sqsubset [\pi_1 G_{r-1}]$.  \qed
\end{enumerate}
\end{fact}

\emph{Notational convention.} The bijection arising from Fact~\ref{FactLamsAndStrata} is often denoted using subscripts as $\Lambda_r \leftrightarrow H_r$.

\medskip

The next two facts are concerned with relative train track maps $f \from G \to G$ that may not be \eg-aperiodic. Note that the straightened iterates $f^k_\# \from G \to G$ are relative train track map, and one of them is \eg-aperiodic. 

The proof of the following is an easy consequence of the definitions and the arguments of \BookOne\ Section~3.

\begin{fact}
Given $\phi \in \Out(F_n)$, a relative train track representative $f \from G \to G$, and an \eg-periodic iterate $f^k_\#$, the bijection $\Lambda_r \leftrightarrow H_r$ between $\L(\phi)=\L(\phi^k)$ and the set of \eg\ strata of $f^k_\#$ is natural in the following sense: if $\phi(\Lambda_r)=\Lambda_{r'}$ then $H_{r'} \subset f(H_r) \subset H_{r'} \union G_{r'-1}$. In particular, $\Lambda_r$ is fixed by $\phi$ if and only $H_r$ is an \eg-aperiodic stratum for $f$. \qed
\end{fact}

\subsubsection{Tiles and applications.} We review from 
Section~3 of \BookOne\ the relations between attracting laminations and tiles of \eg\ strata. In particular Fact~\ref{FactLeafAsLimit} states the relative train track characterization of attracting laminations. We give some other applications regarding weak attraction of paths and circuits, principal rays, generic leaves, and attracting neighborhood bases.

\smallskip

\textbf{Fixed notation:} For the remainder of the section we fix $\phi \in \Out(F_n)$ acting as the identity on $\L(\phi)$, and we fix $f \from G \to G$ a relative train track representative of $\phi$. The conclusions of Fact~\ref{FactLamsAndStrata} therefore hold, in particular $f$ is \eg-aperiodic. We also fix $\Lambda^+ \in \L(\phi)$ and its corresponding \eg-stratum $H_r \subset G$. 

\smallskip

Recall from Definition~3.1.7 of \BookOne\ that for any integer $k \ge 0$, a \emph{$k$-tile of height $r$} is a path of the form $f^k_\#(E)$, where $E$ is an edge of $H_r$; we may drop $k$ and/or $r$ if they are clear from the context. 


\begin{fact} 
\label{FactTiles}
With the notation fixed as above we have:
\begin{enumerate}
\item \label{ItemTileDecomp} 
For each $k$, each generic leaf of $\Lambda^+$ has a decomposition into subpaths each of which is either a $k$-tile of height $r$ or a path in $G_{r-1}$. 
\item \label{ItemUndergraphPieces} Let $\{\mu_i\}$ be the collection of maximal subpaths of $G_{r-1}$ that occur in $1$-tiles and suppose that $\tau_0$ is   a maximal subpath  of $G_{r-1}$ in a $k$-tile $\tau$. Then $\tau_0 = f^m_\#(\mu_i)$ for some $m \le k-1$ and some $\mu_i$; moreover, there are $m$-tiles immediately preceding and following $\tau_0$ in $\tau$.
\item \label{ItemTileExhausted} 
Every generic leaf of $\Lambda^+$ can be exhausted by tiles of height $r$, meaning that it can be written as an increasing union of subpaths each of which is a tile of height $r$. 
\item \label{ItemTilePF}
There exists $p$ such that each for each $k \ge i \ge 0$, each $k+p$-tile of height $r$ contains each $i$ tile of height $r$.
\item \label{ItemTileExpFac}
Letting $\lambda > 1$ be the Perron-Frobenius eigenvalue of the transition matrix of $H_r$, and letting $\ell_r(\tau)$ denote the number of $H_r$ edges of a path $\tau$, for each edge $E \subset H_r$ the ratio $\ell_r(f^{k+1}_\#(E)) \bigm/ \ell_r(f^k_\#(E))$ approaches $\lambda$ as $r \to +\infinity$.
\end{enumerate}
\end{fact}

\begin{proof}
Items~\pref{ItemTileDecomp} and~\pref{ItemTileExhausted} are contained in \BookOne, Lemma 3.1.10 (3) and~(4), respectively.
Item~\pref{ItemUndergraphPieces} follows from    Lemma 5.8 of \BH\  and an obvious induction argument.  To prove item~\pref{ItemTilePF}, choose $p$ so that the $p^{\text{th}}$ power of the transition matrix for $H_r$ is positive, and so for each $i \ge 0$ the $(p+i)^{\text{th}}$ power is also positive; item~\pref{ItemTilePF} for $k=0$ follows immediately, and induction establishes it for all higher values of~$k$. 

Evidently every subpath of $\ell$ is a subpath of a generic leaf of $\Lambda$ and so $\ell$ is itself a leaf of~$\Lambda$. Also, it follows by \pref{ItemTilePF} that $\ell$ is birecurrent. 

Item~\pref{ItemTileExpFac} follows from the Perron-Frobenius Theorem applied to the transition matrix~$M$, part of which says that for any non-negative nontrivial vector $V$ indexed by the edges of $H_r$ we have $\lim_{r \to +\infinity} \abs{M^{k+1} \cdot V} \bigm/ \abs{M^k  \cdot V} = \lambda$ where $\abs{\cdot}$ denotes sum of coordinates; apply this to the column vector $V$ whose coordinates are all $0$ except for a single $1$ corresponding to the edge~$E$.
\end{proof}

\subparagraph{Characterizing attracting laminations.} The following fact gives the relative train track characterization of the set of attracting laminations of~$\phi$. In the logical structure of \BookOne\ this fact is very closely intertwined with Fact~\ref{FactLamsAndStrata}, but we state it here as an explicit, separate observation in its own right. Recall from Section~\ref{SectionLineDefs} the definition of weak attraction for the action of $f_\#$ on $\wh\B(G)$.

\begin{fact}
\label{FactLeafAsLimit}
Continuing with the fixed notation, for any line $\ell \in \B$, $\ell$ is a leaf of $\Lambda^+$ if and only if some (any) edge of $H_r$ is weakly attracted to $\ell$ under the action of $f_\#$ on $\wh\B(G)$.
\end{fact}

\begin{proof} For the ``only if'' direction, suppose that $\ell$ is a leaf of $\Lambda^+$, and consider first the case that $\ell$ is a generic leaf. By Fact~\ref{FactTiles}~\pref{ItemTileExhausted} any subpath of $\ell$ is a subpath of some $k$-tile, which by Fact~\ref{FactTiles}~\pref{ItemTilePF} is a subpath of some tile $f^{k+p}_\#(E)$ for some edge $E$ of~$H_r$. It follows that $E$ is weakly attracted to $\ell$ by iteration of $f_\#$. From this it also follows that $E$ is weakly attracted to every weak limit of a generic leaf of $\Lambda^+$, which includes every leaf of~$\Lambda^+$.

For the ``if'' direction, consider any edge $E$ of $H_r$, and suppose $E$ is weakly attracted to a line $\ell$. Choose a generic leaf $\beta$ of~$\Lambda^+$. Each subpath of $\ell$ is contained in $f^k_\#(E)$ for some~$k$, which by Fact~\ref{FactTiles}~\pref{ItemTilePF} is contained in any $r$-tile of the form $f^{k+p}_\#(E')$. By Fact~\ref{FactTiles}~\pref{ItemTileDecomp}, some such tile $f^{k+p}_\#(E')$ is a subpath of $\beta$, proving that $\ell$ is in the weak closure of $\beta$ which is just~$\Lambda^+$.
\end{proof}

\subparagraph{Weak attraction of paths and circuits.} Next we characterize those paths and circuits in $G$ which are weakly attracted to~$\Lambda_+$.


\begin{fact} 
\label{FactAttractingLeaves}
Continuing with the fixed notation, we have:
\begin{enumerate}
\item\label{ItemSplitAnEdge} Given a path or circuit $\sigma \subset G$ the following are equivalent: 
\begin{enumerate}
\item\label{ItemSigmaWeaklyAttracted}
$\sigma$ is weakly attracted to $\Lambda^+$ under the action of $f$ of $\wh\B(G)$.
\item\label{ItemIteratedEdgeInHr}
There exists $k \ge 0$ and a splitting of $f^k_\#(\sigma)$, one term of which is an edge in $H_r$.
\item\label{ItemEventualEdgeInHr}
For all sufficiently large $k$, $f^k_\#(\sigma)$ is completely split and one term of the complete splitting is an edge of $H_r$.
\end{enumerate}
\item\label{ItemAllCircuitsRepelled} No conjugacy class in $F_n$ is weakly attracted to $\Lambda^+$ by iteration of $\phi^\inv$. More generally, letting $U$ be an attracting neighborhood of $\Lambda^+$, for any conjugacy class $[a]$ there are only finitely many values of $k \ge 0$ for which $\phi^{-k}[a] \in U$.
\end{enumerate}
\end{fact}

\begin{proof} The equivalence of \pref{ItemSigmaWeaklyAttracted} and~\pref{ItemIteratedEdgeInHr} is in Corollary 4.2.4 of \BookOne, and the equivalence of~\pref{ItemIteratedEdgeInHr} and~\pref{ItemEventualEdgeInHr} follows from Fact~\ref{FactComplSplitStable} and the evident fact that for each edge $E$ of $H_r$ and each $i \ge 1$ some term of the complete splitting of $f^i_\#(E)$ is an edge of $H_r$.  To prove item \pref{ItemAllCircuitsRepelled}, let $c$ be the circuit in $G$ representing $[a]$. Since $\phi$ fixes $\Lambda^+$, we have $\phi(U) \subset U$. By Lemma 3.1.16 of \BookOne\ a generic leaf $\ell$ is not a circuit, and combined with the fact that $\{\phi^{k}(U)\}$ is a weak neighborhood basis of $\ell$, it follows that there exists $I$ such that $[a] \not\in \phi^I(U)$, and so $\phi^{-i}[a] \not\in U$ for $i \ge I$.
\end{proof}

\subparagraph{Principal rays.} 
The next fact contains applications of tiles to the understanding of rays in leaves.

\begin{corollary}\label{CorTilesExhaustRay}
Continuing with the notation fixed above, for any leaf $\ell$ of $\Lambda^+$, and for any subray $R \subset \ell$ whose initial vertex is principal and whose initial edge is in $H_r$, $R$ is a principal ray if and only if $R$ is the union of a nested increasing sequence of tiles in $\ell$. 
\end{corollary}

\begin{proof}
The ``only if'' direction follows from the definition of a principal ray. Suppose $R$ is the union of oriented tiles $\tau_1 \subset \tau_2 \subset \tau_3 \subset \cdots \subset \ell$ where $\tau_i$ is an $m_i$ tile, then $m_i \to +\infinity$ as $i \to +\infinity$. Eventually $\tau_i$ contains the initial point of $R$, and so after throwing away finitely many terms, each $\tau_i$ is an initial segment of $R$. There is an $m_0$ so that the initial direction of each $m_0$-tile is fixed. After passing to a subsequence, there is an $m_0$-tile $\tau_0$ such that $\tau_i = f^{m_i-m_0}_\#(\tau_0)$ for each $i \ge 1$, and so $R$ is the principal ray associated to the initial direction of $\tau_0$. 
\end{proof}

\subparagraph{Generic leaves.} The following fact characterizes generic leaves in terms of their asymptotic behavior.

\begin{fact}\label{FactTwoEndsGeneric}
Continuing with the notation fixed above, for any leaf $\ell$ of $\Lambda^+$, the following are equivalent:
\begin{enumerate}
\item \label{ItemGenericLeaf}
$\ell$ is a generic leaf of $\Lambda^+$.
\item \label{ItemBirecurrentLeaf}
$\ell$ is birecurrent and has height $r$.
\item \label{ItemDoubleHeightLeaf}
Both ends of $\ell$ have height $r$.
\end{enumerate}
\end{fact}

\begin{proof} The implications \pref{ItemGenericLeaf} $\implies$ \pref{ItemBirecurrentLeaf} $\implies$ \pref{ItemDoubleHeightLeaf} are obvious, and \pref{ItemBirecurrentLeaf} $\implies$ \pref{ItemGenericLeaf} follows from \BookOne\ Lemma~3.1.15. 

We prove \pref{ItemDoubleHeightLeaf} $\implies$ \pref{ItemBirecurrentLeaf}. Consider a generic leaf $\beta$ of $\Lambda^+$. For each $k \ge 0$, applying Fact~\ref{FactTiles}~\pref{ItemTileDecomp} to $\beta$ it follows that there exists $N_k > 0$ so that any subpath of $\beta$ that contains at least $N_k$ edges of $H_s$ contains a $k$-tile. Since every subpath of $\ell$ occurs as a subpath of $\beta$, every subpath of $\ell$ that contains at least $N_k$ edges of $H_s$ contains a $k$-tile. Since both ends of $\ell$ have height $s$ it follows that every initial or terminal ray of $\ell$ contains a $k$-tile for all~$k$. By Fact~\ref{FactTiles}~\pref{ItemTileExhausted}, each subpath of $\beta$, and hence each subpath $\sigma$ of $\ell$, is contained in a subpath which is a tile, and hence $\sigma$ occurs as a subpath of every initial or terminal ray  of $\ell$. This proves that $\ell$ is birecurrent.
\end{proof}

\subparagraph{Construction of an attracting neighborhood basis.}
Recall the remark following Definition~\ref{DefAttractingLaminations} which, for each attracting lamination of $\phi$ that is fixed by $\phi$, demonstrates existence of an attracting neighborhood for the action of~$\phi$. The following fact explains, under an additional constraint regarding existence of a certain fixed line, how to find a simple description of attracting neighborhoods expressed in terms of relative train track maps. 

\begin{fact}\label{FactAttrNhdBasis}
For any $\phi \in \Out(F_n)$, any $\Lambda \in \L(\phi)$ which is fixed by $\phi$, and any relative train track representative $f \from G \to G$ with \eg-aperiodic stratum $H_r$ corresponding to $\Lambda$, if there exists a generic leaf $\ell \in \Lambda$ which is fixed by $\phi$ with fixed orientation, then the realization of $\ell$ in $G$ has a nested sequence of finite subpaths of the form $\gamma_0 \subset \gamma_1 \subset \gamma_2 \subset $ whose union is~$\ell$, such that each $\gamma_i$ begins and ends with an edge of $H_r$ and such that the sequence of sets $W(\gamma_0) \supset W(\gamma_1) \supset W(\gamma_2) \supset \cdots$ is a nested attracting neighborhood basis for $\Lambda$. 
\end{fact}

\newcommand\pf{\text{PF}}
\renewcommand\O{\mathcal{O}}

\begin{proof} Let $\lambda>1$ be the Perron-Frobenius eigenvalue of the transition matrix of $H_r$ and let $\ell_r(\tau)$ denote the number of $H_r$ edges of a path $\tau$ in $G$. Applying Fact~\ref{FactAttractingLeaves}~\pref{ItemTileExpFac} we may pick $k$ so that under the action of $f$ the number of $H_r$ edges of a $k$ tile is multiplied by at least $\lambda'=(\lambda+1)/2>1$, that is to say,
$$\ell_r(f^{k+1}_\#(E)) \bigm/ \ell_r(f^k_\#(E)) > \lambda' \qquad\qquad (*)
$$

We may decompose $\ell$ uniquely as an alternating concatenation of $0$-tiles (edges of $H_r$) and subpaths each either maximal in $G_{r-1}$ or trivial, yielding the \emph{$0$-tile decomposition of $\ell$} which we denote 
$$\ell = \ldots E_{i-1} \, \mu_i \, E_i \, \mu_{i+1} \, E_{i+1} \ldots
$$
Pushing forward by iteration of $f_\#$ and using that $f_\#(\ell)=\ell$, by induction we obtain for each integer $i \ge 0$ the $i$-tile decomposition of $\ell$, a bi-infinite concatenation of $i$-tiles and subpaths each either maximal in $G_{r-1}$ or trivial. In particular, denoting $\tau_i = f^k_\#(E_i)$ and $\nu_i = f^k_\#(\mu_i)$ we obtain the $k$-tile and the $(k+1)$-tile decompositions
\begin{align*}
\ell &= \ldots \tau_{i-1} \, \nu_i \, \tau_i \, \nu_{i+1} \, \tau_{i+1} \ldots \\
     &=  \ldots f_\#(\tau_{i-1}) \, f_\#(\nu_i) \, f_\#(\tau_i) \, f_\#(\nu_{i+1}) \, f_\#(\tau_{i+1}) \ldots
\intertext{Note that we have decompositions}
\tau_i &= E_{a(i)} \, \mu_{a(i)+1} \, \ldots \, \mu_{a(i+1)-1} E_{a(i+1)-1} \\
f_\#(\tau_i) &= E_{b(i)} \, \mu_{b(i)+1} \, \ldots \, \mu_{b(i+1)-1} E_{b(i+1)-1}
\end{align*}
for strictly increasing sequences $a(i)$ and $b(i)$. Note also that the first edge of $f_\#(E_{a(i)})$ is $E_{b(i)}$ and the last edge of $f_\#(E_{a(i+1)-1})$ is $E_{b(i+1)-1}$.

Applying $(*)$ we have $$\frac{b(i+1) - b(i)}{a(i+1)-a(i)} > \lambda' > 1
$$
from which it follows that the integer sequence $b(i)-a(i)$ is strictly increasing for $i \in \Z$ and so is negative for $i$ near $-\infinity$ and positive near $+\infinity$. We may therefore choose $I \ge 1$ so that for all $i \ge I$ we have $b_{-i} < a_{-i}$ and $a_i < b_i$.
It follows that the sequence of paths $\gamma_i = E_{a(-i)} \, \mu_{a(-i)+1} \, \ldots \mu_{a(i)-1} E_{a(i)-1}$,  $i \ge I$ is nested, its union is all of $\ell$, each begins and ends with an edge of $H_r$, and $f_\#(\gamma_i)$ contains $\gamma_{i+1}$ (as a subpath). It follows in turn that $W(\gamma_i)$ is a neighborhood basis of $\ell$, and by induction that $f^j_\#(\gamma_i)$ contains $\gamma_{i+j}$ for all $j \ge 0$ and so $W(\gamma_i)$ is an attracting neighborhood.
\end{proof}

\subsubsection{\eg\ principal rays and $\Fix_N$} 
\label{SectionEGPrincipalRays}
For any outer automorphism $\phi$, any expanding lamination, any generic leaf $\ell$ which $\phi$ fixes preserving orientation, and any $\Phi \in \Aut(F_n)$ representing $\phi$ and fixing a lift $\ti\ell$ of $\ell$, both endpoints of $\ti\ell$ are contained in $\Fix_+(\wh\Phi)$. This fact, reminiscent of Nielsen theory, is stated in \recognition\ Remark~3.6. We need a version of this fact which starts not with a generic leaf but only with an \eg\ principal ray $R$. We will need this fact in the proof of Lemma~\refRK{LemmaEigenrayDefs}.

For the statement we need a definition. Consider a \ct\ $f \from G \to G$ and a linear \neg\ edge $E$ with $f(E) = Ew^k$ as in (Linear Edges) in the Definition of a \ct\ (\recognition\ Definition~4.7, or Definition~\ref{DefCT}). The ray generated by $E$ has the form $Ew^{\infinity} = Ewww\cdots$, called the linear ray generated by $E$. The ray $Ew^{-\infinity} = E w^\inv w^\inv w^\inv \cdots$ will be called the \emph{exceptional ray} of $E$. Note that the open ray obtained from an exceptional ray by removing its initial point is a nested increasing union of pretrivial paths.

\begin{lemma} \label{LemmaEGPrincipalRays}
Consider a rotationless $\phi \in \Out(F_n)$, a \ct\ representative $f \from G \to G$, and an \eg-stratum $H_r$ with associated $\Lambda \in \L(\phi)$. For any principal direction $E \subset H_r$ generating a principal ray $R$ in $G$ with initial vertex $v$, there exists a leaf of $\Lambda$ of the form $\ell = \overline R' \gamma R$ such that $f_\#(\ell)=\ell$ preserving orientation, $\gamma$ is a (possibly trivial) Nielsen path in $G_{r-1}$, and one of the following holds:
\begin{enumerate}
\item\label{ItemPrincipleRayGeneric}
$R'$ is the ray generated by some nonfixed edge $E' \subset G_r$ with fixed initial direction.
\item\label{ItemExceptionalRay}
$R'$ is the exceptional ray of some linear edge $E' \subset G_{r-1}$.
\item\label{ItemFixedRay}
$\gamma$ is trivial and $R'$ splits into Nielsen paths and fixed edges contained in $G_{r-1}$.
\end{enumerate}
Furthermore, for any $\Phi \in P(\phi)$ with associated principle lift $\ti f \from \wt G \to \wt G$ fixing a lift $\ti v$ of $v$, letting $\wt R$ be the lift of $R$ with initial vertex $\ti v$, letting $\xi \in \bdy F_n$ be the limit point of $\wt R$, letting $\ti\ell$ be the lift of $\ell$ containing $\wt R$, and letting $\xi'$ be the opposite end of $\ti\ell$, we have $\ti f_\#(\ti\ell)=\ti\ell$ preserving orientation, and furthermore:
\begin{enumeratecontinue}
\item\label{ItemOppEndInFixN} $\xi \in \Fix_+(\wh\Phi)$; $\xi' \in \Fix_N(\wh\Phi)$; and $\xi' \in \Fix_+(\wh\Phi)$ if and only if $E'$ is a principal direction.
\end{enumeratecontinue}
\end{lemma}

\begin{proof} Once~\pref{ItemPrincipleRayGeneric}--\pref{ItemFixedRay} are proved, the rest is proved by writing $\ti\ell = \wt {\overline R}' \cdot \ti \gamma \cdot \wt R$ and applying Fact~\ref{FactSingularRay}.

We turn now to the construction of the leaf $\ell$. Define a \emph{split leaf segment of $\Lambda$ in $G$} to be a completely split path $\delta$ in $G_r$ such that some generic leaf $\ell$ of $\Lambda$ splits as $\ell =  \bar\rho_- \cdot \delta \cdot \rho_+$.  We use without comment that the $f_\#$ image of a split leaf segment is a split leaf segment, and that every concatenation of consecutive terms of the complete splitting of a split leaf segment, which we shall call ``a split leaf subsegment'', is a split leaf segment. Since each generic leaf of $\Lambda$ crosses every edge of $H_r$ infinitely often, there exists a split leaf segment of the form $\overline E'' \cdot \alpha \cdot E$ such that $E''$ is an edge of $H_r$ and $\alpha$ is a (possibly trivial) completely split path in $G_{r-1}$. 

Suppose first that $\alpha$ is trivial or a Nielsen path. After replacing $\overline E'' \cdot \alpha \cdot E$ by its image under some iterate $\ti f^k_\#$ so that $\ti f^k_\#(E'')$ has fixed initial direction, and then passing to a split leaf subsegment, we may assume that $E''$ has fixed initial direction and so is a principal direction. Letting $E'=E''$ and letting $R'$ be the ray generated by $E'$,  the line $\ell = \overline R' \alpha R$ satisfies~\pref{ItemPrincipleRayGeneric} with $E' \subset H_r$.

Suppose now that $\alpha$ is nontrivial and not a Nielsen path. Discarding $E''$ we narrow our focus to split leaf segments of the form $\alpha \cdot E$ where $\alpha$ is a completely split path in $G_{r-1}$ and is not a Nielsen path. There is a unique splitting $\alpha = \beta_2 \cdot \beta_1 \cdot \gamma$ such that $\gamma$ is the longest Nielsen path which is a terminal segment of the complete splitting of $\alpha$, and $\beta_1$ is the term of the complete splitting preceding $\gamma$; note that $\beta_2$ may be trivial but $\beta_1$ is nontrivial. Amongst all such paths $\alpha \cdot E = \beta_2 \cdot \beta_1 \cdot \gamma \cdot E$ under consideration, let $s<r$ be the minimum height of $\beta_1$. Discarding $\beta_2$ we further narrow our focus to split leaf segments of the form $\beta_1 \cdot \gamma \cdot E$ where $\gamma$ is a Nielsen path in $G_{r-1}$ and $\beta_1$ is a height~$s$ complete splitting term which is neither a fixed edge nor an indivisible Nielsen path. Note also that $\beta_1$ is not contained in a zero stratum, because by applying (Zero Strata) of the definition of a \ct\ (\recognition\ Definition~4.7, or Definition~\ref{DefCT}) it follows that no vertex in a zero stratum is fixed by $f$, but the terminal point of $\beta_1$, which equals the initial point of $\gamma$, is fixed. We now go through the remaining cases of $\beta_1$. 

If $\beta_1 = \overline E'$ where $E'$ is an \eg\ or \neg\ edge with fixed initial direction then, taking $R'$ to be the ray generated by $E'$, we are done with $\ell = \overline R' \cdot \gamma \cdot R$ satisfying~\pref{ItemPrincipleRayGeneric}. If $\beta_1 = E'$ where $E'$ is an \eg\ edge with nonfixed terminal direction then, after replacing $\beta_1 \cdot \gamma \cdot E$ by its image under some iterate $\ti f^k_\#$ so that $\ti f^k_\#(E')$ has fixed terminal direction, and then passing to a split leaf subsegment, we reduce to the previous case. If $\beta_1 = E_s$ is a nonfixed, nonlinear \neg\ edge with fixed initial direction then $f(\beta_1) = E_s \cdot u$ where $u$ is of height~$<s$ and is not a Nielsen path, and so after one replacing $\beta_1 \cdot \gamma \cdot E$ with its image under $\ti f_\#$ and passing to a split leaf subsegment, we obtain an example with smaller~$s$, a contradiction. If $\beta_1 = E_s$ is a nonfixed, linear \neg\ edge with fixed initial direction and $f(E_s) = E_s \cdot u$ then $f^k_\#(\beta_1 \cdot \gamma \cdot E)$ contains a split leaf subsegment of the form $u^k \cdot \gamma \cdot f^k_\#(E)$ and passing to a weak limit one obtains $\ell = \overline R' R$ satisfying \pref{ItemFixedRay} with $R' = \bar\gamma u^{-\infinity}$. In the remaining case wheree $\beta_1 = E_i w^p \overline E_j$ is an exceptional path, with $f(E_i) = E_i w^{d_i}$, $f(E_j) = E_j w^{d_j}$, then $f^k_\#(\beta_1 \cdot \gamma \cdot E)$ contains a split leaf subsegment of the form $E_i w^{p + k(d_i - d_j)} \overline E_j \cdot \gamma \cdot f^k_\#(E)$, and passing to a weak limit one obtains $\ell = \overline R' \cdot \gamma \cdot R$ where $R'$ is one of the two rays $E_j w^{-\infinity}$ or $E_j w^{+\infinity}$ depending on whether $d_i-d_j$ is positive or negative; this is a linear ray if $d_j > 0$ and an exceptional ray otherwise.
\end{proof}

\subsubsection{Pushing forward attracting laminations.} One commonly used result about attracting laminations says that for each $\Phi \in \Aut(F_n)$ and each $\Phi$-invariant free factor of $A$, the inclusion $A \inject F_n$ induces an injection from the set of attracting laminations of the outer automorphism class of $\Phi \restrict A \in \Aut(A)$ to the set of attracting laminations for the outer automorphism class of $\Phi \in \Aut(F_n)$ itself. This is proved by an easy relative train track argument, using that every outer automorphism preserving a free factor is represented by a relative train track map having a filtration element representing that free factor. We need a generalization of this result which applies without the assumption that $A$ is a free factor; in the conclusion we must sacrifice injectivity of the induced map.

Given a finite rank subgroup $A \subgroup F_n$, the embedding $\bdy A \inject \bdy F_n$ induces an embedding $\wt\B(A) \inject \wt\B(F_n)$ which in turn induces a continuous map $\beta_A \from \B(A) \to \B(F_n)$, defined by saying that the $\beta_A$-image of an $A$-orbit in $\wt\B(A)$ is the unique $F_n$-orbit in $\wt\B(F_n)$ containing the given $A$-orbit.

The following lemma will be used in the proof of Proposition~\ref{PropGeomLams}.


\begin{lemma}\label{LemmaPushLam}
For each $\phi \in \Out(F_n)$, each representative $\Phi \in \Aut(F_n)$, and each finite rank $\Phi$-invariant subgroup $A \subgroup F_n$, letting $\phi \restrict A \in \Out(A)$ denote the outer automorphism class of $\Phi \restrict A$, for each dual lamination pair $\Lambda^\pm \in \L^\pm(\phi \restrict A)$ the following hold:
\begin{enumerate}
\item \label{ItemPushLamIsLam}
$\beta_A(\Lambda^+) \in \L(\phi)$ and $\beta_A(\Lambda^-) \in \L(\phi^\inv)$
\item \label{ItemPushLamSupp}
Letting $B \subgroup A$ represent the free factor support of $\Lambda^\pm$ in $A$, the free factor supports of $\beta_A(\Lambda^+)$ and of $\beta_A(\Lambda^-)$ in $F_n$ are each equal to the free factor support of $B$ in~$F_n$.
\item \label{ItemPushLamDual}
$\beta_A(\Lambda^+)$, $\beta_A(\Lambda^-)$ are a dual lamination pair in $\L^\pm(\phi)$.
\end{enumerate}
\end{lemma}

Note that unlike Fact~\ref{FactMalnormalRestriction} the notation $\phi \restrict A$ does not presuppose well-definedness independent of the choice of $\Phi$. However, it is at least well-defined up to post-composing $\Phi$ by an inner automorphism determined by an element of~$A$.

\begin{proof} Since duality of laminations is defined by having the same free factor support, clearly \pref{ItemPushLamSupp}$\implies$\pref{ItemPushLamDual}. 

Also \pref{ItemPushLamIsLam}$\implies$\pref{ItemPushLamSupp}, which we see as follows. We add a superscript ``$A$'' to certain notations, e.g.\ $[\cdot]^A$ denotes conjugacy classes in $A$ and $\F^\A_\supp(\cdot)$ denotes free factor support in~$A$; with no superscript these notations retain their ordinary meaning in~$F_n$. By Fact~\ref{FactSubgroupSupport}, $\F_\supp(B)$ is the free factor support of the set of lines carried by the one component subgroup system $[B]$, and so $\beta_A(\Lambda^+)$ is carried by $\F_\supp(B)$, proving that $\F_\supp(\beta_A(\Lambda^+)) \sqsubset \F_\supp(B)$. For the opposite direction, suppose that $\C = \{[C_1],\ldots,[C_K]\}$ is a free factor system in $F_n$ that supports $\beta_A(\Lambda^+)$. By the Kurosh Subgroup Theorem the collection of conjugacy classes $\{[A \intersect C_k^g]^A \suchthat g \in F_n, 1 \le k \le K\}$ forms a free factor system in $A$, and by supposition it carries $\Lambda^+$ in $A$. It follows that $\{[B]^A\} = \F_\supp^A(\Lambda^\pm) \sqsubset \C$ which implies that $B \subset C_k^g$ for some $k=1,\ldots,K$ and some $g \in F_n$, which implies that $\F_\supp(B) \sqsubset \C$. Since this is true for all $\C$ supporting $\beta_A(\Lambda^+)$ we have $\F_\supp(B) \sqsubset \F_\supp(\beta_A(\Lambda^+))$, and so $\F_\supp(\beta_A(\Lambda^+)) = \F_\supp(B)$. The same is true for $\Lambda^-$ by replacing $\phi$ with $\phi^\inv$, which proves~\pref{ItemPushLamSupp}.

We prove~\pref{ItemPushLamIsLam} just for $\Lambda^+$; the same follows for $\Lambda^-$ by replacing $\phi$ with $\phi^\inv$. For the proof we drop the ``$+$'' superscript and simply write $\Lambda$ for $\Lambda^+$. We adapt the proof of \BookOne\ Lemma~3.1.9. We may freely pass to any positive power of $\phi$, and so in particular we may assume $\phi \restrict A$ fixes $\Lambda$ and so $\phi$ fixes $\beta_A(\Lambda)$. 

Pick a relative train track map $h \from H \to H$ representing $\phi \restrict A$ with filtration elements $H_i$ and strata $H_i \setminus H_{i-1}$, having in particular an aperiodic \eg\ stratum $H_r \setminus H_{r-1}$ corresponding to $\Lambda$. After passing to a positive power, there is an edge $E \subset H_r \setminus H_{r-1}$ and a point $x \in \interior(E)$ such that $h(x)=x$. Up in the universal cover choose lifts $\wt E \subset H$ of $E$, $\ti x \in \interior(E)$ of $x$, and $\ti h \from \wt H \to \wt H$ of $h$ so that $\ti h(x)=x$. After post-composing $\Phi$ by an inner automorphism determined by an element of~$A$, we may assume that $\ti h$ is the lift of $h$ corresponding to $\Phi \restrict A$. Letting $\ti\ell = \union_{i=1}^\infinity \ti h^k_\#(E) \in \wt\B(H)$, its downstairs image $\ell \in \B(H)$ is a generic leaf of the realization of $\Lambda$ in $H$, as shown in the proof of \BookOne\ Lemma~3.1.9. By definition the leaf $\ell \in \Lambda$ is birecurrent in $\B(A)$ and dense in $\Lambda$, and these evidently imply that the leaf $\beta_A(\ell) \in \beta_A(\Lambda)$ is birecurrent in $\B(F_n)$ and dense in $\beta_A(\Lambda)$. It remains to prove existence of an attracting neighborhood of $\beta_A(\ell)$ in $\B(F_n)$ under the action of~$\phi$. 

The leaf $\ti\ell$ crosses edges of $\wt H_r \setminus \wt H_{r-1}$ infinitely often on both directions; enumerate those crossings in order as $\wt E_k$, $k \in \Z$, with $\wt E = \wt E_0$. For each integer $k \ge 0$ let $\ti\ell_k$ be the subsegment of $\ti \ell$ starting with $\wt E_{-k}$ and ending with $\wt E_k$, and so $\union_k \ti\ell_k = \ti\ell$. Let $\ell_k \in \wh\B(H)$ be the downstairs image of $\ti\ell_k$, so $\ell$ is a weak limit of $\ell_k$ as $k \to +\infinity$. After passing to a positive power we may assume that the $h_\#$-image of each edge of $H_r \setminus H_{r-1}$ crosses at least two edges, and we obtain the inclusion $\ti h_\#(\ti \ell_k) \supset \ti\ell_{2k}$ (c.f.\ \BookOne\ Lemma~2.5.1). From this inclusion, together with Lemma~\ref{LemmaDoubleSharpFacts}~\pref{ItemOneToTwoSharps}, it follows that
\begin{description}
\item[$(*)$] For each $d \ge 0$, if $k \ge \BCC(h) + d$ then $\ti h_{\#\#}(\ti\ell_k) \supset \ti\ell_{k+d}$. 
\end{description}
because both components of $\ti\ell_{2k} - \ti\ell_{k+d}$ have length at least $k-d \ge \BCC(h)$.

Property~$(*)$ with $d=1$ is a key step in the proof of \BookOne\ Lemma~3.1.9 for showing that $\ell$ has an attracting neighborhood under iteration of $\phi \restrict A$. The strategy of our proof that $\beta_A(\ell)$ has an attracting neighborhood under iteration of $\phi$ is to use bounded cancellation arguments to transfer this key step over to the setting of a topological representative of $\phi$. 

Pick an $F_n$-marked graph $G$ and a $\pi_1$-injective map $\mu \from H \to G$ which, with appropriate choices of bases points and paths between them, induces the injection $A \inject F_n$. Let $\ti\mu \from \wt H \to \wt G$ be an $A$-equivariant lift of~$\mu$. Pick a homotopy equivalence $g \from G \to G$ representing~$\phi$, and so $g \composed \mu$ is homotopic to $\mu \composed h$. Let $\ti g \from \wt G \to \wt G$ be the lift corresponding to the automorphism~$\Phi$. It follows that the homotopy between the maps $g \composed \mu$ and $\mu \composed h \from H \to G$ lifts to an $A$-equivariant homotopy between the maps $\ti g \composed \ti \mu$ and $\ti\mu \composed \ti h \from \wt H \to \wt G$. By compactness of~$H$, the track of each point under this homotopy has diameter in $\wt G$ bounded by a uniform constant~$T$.

Working with the realizations of $\Lambda$ in $H$ and of $\beta_A(\Lambda)$ in $G$, the realizations in $G$ of the leaves of $\beta_A(\Lambda)$ are precisely the $\mu_\#$-images of the realizations in $H$ of the leaves of $\Lambda$. In particular, $\ell' = \mu_\#(\ell)$ is the realization in $G$ of $\beta_A(\ell) \in \beta_A(\Lambda)$. We must prove that $\ell'$ has an attracting neighborhood in $\B(G)$ under the action of~$g_\#$.

In $\wt G$ let $\ti\ell'_k = \ti\mu_{\#\#}(\ti\ell_k)$ and in $G$ let $\ell'_k = \mu_{\#\#}(\ell_k)$. Applying Lemma~\ref{LemmaDoubleSharpFacts}~\pref{ItemDblSharpContain} we have a nested sequence $\ti\ell'_1 \subset \ti\ell'_2 \subset \cdots$ whose union is contained in $\ti\ell' = \ti\mu_\#(\ti\ell)$. Letting $V'_k = V(\ell'_k;G)$ we obtain a nested sequence of weakly open sets $V'_1 \supset V'_2 \supset \cdots$ of $\B(G)$. We shall prove:
\begin{itemize}
\item[] (a) The sequence $V'_1 \supset V'_2 \supset \cdots$ is a weak neighborhood basis for $\ell'$.
\item[] (b) For all sufficiently large $k$, say $k \ge K$, we have $g_\#(V'_k) \subset V'_{k+1}$.
\end{itemize}
from which it immediately follows that $V'_K$ is a weakly attracting neighborhood of $\ell'$ under iteration of $g_\#$.

We claim that:
\begin{description}
\item[$(**)$] For each $M \ge 0$ there exists $e_M$ such that if $e \ge e_M$ then for each $k$ both components of $\ti\ell'_{k+e} - \ti\ell'_k$ have length $> M$. 
\end{description}
To prove the claim, from the nesting of these arcs we obtain a bijection correspondence between the endpoints of $\ti\ell'_{k+e}$ and the endpoints of $\ti\ell'_k$, so it is sufficient to prove that the distance between corresponding endpoints is $>M$. We have $\ti\ell'_{k+e} - \ti\ell'_k = \ti\mu_{\#\#}(\ti\ell_{k+e}) - \ti\mu_{\#\#}(\ti\ell_k)$. For each arc $\alpha \subset \wt H$, by Lemma~\ref{LemmaDoubleSharpFacts}~\pref{ItemOneToTwoSharps} the arc $\ti\mu_{\#\#}(\alpha) \subset \wt G$ is obtained from the arc $\ti\mu_\#(\alpha)$ by trimming away at most $\BCC(\mu)$ from both ends of $\ti\mu_\#(\alpha)$, and so we get a bijective correspondence between the endpoints of $\ti\mu_\#(\alpha)$ and those of $\ti\mu_{\#\#}(\alpha)$ so that corresponding endpoints have distance $\le \BCC(\mu)$. It therefore suffices to show that corresponding endpoints of the arcs $\ti\mu_\#(\ti\ell_{k+e})$, $\ti\mu_\#(\ti\ell_k)$ have distance $>M-2\BCC(\mu)$. Using that $\ti\mu \from \wt H \to \wt G$ is an $(a,c)$ quasi-isometric embedding for some $a \ge 1$, $c \ge 0$, it is sufficient to show that corresponding endpoints of $\ti\ell_{k+e},\ti\ell_k$ have distance $> a(M - 2 \BCC(\mu) + c)$, which is guaranteed by choosing $e > e_M = a(M-2\BCC(\mu)+c)$. 

From the claim it follows that $\union_k \ti\ell'_k = \ell'$, and so $V'_1 \supset V'_2 \supset \cdots$ is a neighborhood basis of $\ell'$ in $\B(G)$, verifying~(a). We shall also use the claim, with an appropriate choice of $M$, to verify~(b). 

We have:
\begin{align*}
\ti\ell'_{k+d}   &= \ti\mu_{\#\#}(\ti\ell_{k+d}) \subset \ti\mu_{\#\#}(h_{\#\#}(\ti\ell_k)) \\
\intertext{which follows by applying $(*)$ and Lemma~\ref{LemmaDoubleSharpFacts}, as long as $k \ge K_d= \BCC(h) + d$. We make an appropriate choice of $d$ below. Applying Lemma~\ref{LemmaDoubleSharpFacts} again we have:}
\ti\ell'_{k+d}    &\subset (\ti\mu \composed \ti h)_{\#\#}(\ti\ell_k) \subset (\ti\mu \composed \ti h)_\#(\ti\ell_k) \\
\intertext{Let $M = \max\{T,\BCC(g \composed \mu)\}$, choose $e \ge e_M$ and let $d=2e+1$, and so}
\ti\ell'_{k+2e+1} &\subset (\ti\mu \composed \ti h)_\#(\ti\ell_k) \\
\intertext{Applying $(**)$ twice, using $M \ge T$ the first time and $M \ge \BCC(g \composed \mu)$ the second, together with a few more applications of Lemma~\ref{LemmaDoubleSharpFacts}, we have:}
\ti \ell'_{k+e+1} &\subset (\ti g \composed \ti\mu)_\#(\ti\ell_k) \\
\ti \ell'_{k+1} &\subset (\ti g \composed \ti \mu)_{\#\#}(\ti \ell_k) \subset \ti g_{\#\#} (\ti \mu_{\#\#}(\ti\ell_k)) = \ti g_{\#\#}(\ti \ell'_k)
\end{align*}
and so $g_\#(V'_k) \subset V'_{k+1}$.
\end{proof}

\bigskip

\section{Geometric \eg\ strata and geometric laminations}
\label{SectionGeometric}

A geometric outer automorphism $\phi \in F_n$, as defined in Section~4 of \cite{BestvinaHandel:tt}, is one modeled by a pseudo-Anosov homeomorphism $f \from S \to S$ of a surface with nonempty boundary. The broader concept of a geometric \eg\ stratum $H_r = G_r \setminus G_{r-1}$ of a relative train track map $f \from G \to G$ was introduced in Definition~5.1.4 of \BookOne: roughly speaking $H_r$ is geometric if the restriction $f \from G_r \to G_r$ is modeled by a 2-dimensional dynamical system obtained by gluing together the restriction $f \from G_{r-1} \to G_{r-1}$ with a pseudo-Anosov surface homeomorphism $h \from S \to S$, by attaching all but one component of $\bdy S$ to $G_{r-1}$. 

In Section~\ref{SectionGeometricModelsAndStrata} we review and reformulate the definition of geometric strata from Definition~5.1.4 of \BookOne, couching it in terms of existence of the 2-dimensional dynamical system alluded to above, which is formally described in Definitions~\ref{DefWeakGeomModel} and~\ref{DefGeometricStratum} where it is dubbed a \emph{geometric model} for $f$ and $H_r$. 

In Section~\ref{SectionGeomModelComplement} we study a natural graph of groups decomposition of $F_n$ associated to a geometric stratum called the \emph{peripheral splitting}, which records group theoretic information underlying the topological manner in which the geometric model is glued together from the surface $S$ and the graph $G$. 

In Sections~\ref{SectionLamsGeomStratum} and~\ref{SectionLaminationGeometricity} we study the dual lamination pair of a geometric stratum. Proposition~\ref{PropGeomLams}, which is based on the geometric case of the proof of Proposition 6.0.8 of \BookOne, shows that the laminations of a geometric stratum $H_r$ can be identified with the stable and unstable laminations of the pseudo-Anosov homeomorphism $h \from S \to S$. Lemma~\ref{containmentSymmetry} shows that the duality relation between attraction laminations of an outer automorphism and those of its inverse preserves inclusion of laminations. Proposition~\ref{PropGeomEquiv} is a new result which proves that ``geometricity'' is actually an invariant of an attracting lamination, indeed of a dual lamination pair, not just of a stratum. 

In Sections~\ref{SectionPreservingSurface} and~\ref{SectionFreeBoundaryInvariant} we establish some complementarity properties for geometric strata that are analogous to properties of subsurfaces and their complements. The results of these sections will be used in \PartTwo\ to establish certain invariance properties of elements of $\IA_n(\Z/3)$.

\subsection{Defining and characterizing geometric strata}  
\label{SectionGeometricModelsAndStrata} 

In this section we review the definition of geometric strata, and other results from \BookOne\ concerning their properties. In particular, given a \ct\ $f \from G \to G$ and an \eg\ stratum $H_r$, geometricity of $H_r$ is defined in Definition~5.1.4 of \BookOne\ in terms of an object that we term a ``weak geometric model'', formalized in Definition~\ref{DefWeakGeomModel}. In Definition~\ref{DefGeomModel} we reformulate geometricity again in terms of a (stronger) ``geometric model''. The difference between the two models is the context in which they are constructed: a ``weak'' geometric model is a homotopy model for the restriction of $f$ to $G_r$; a ``strong'' geometric model is a homotopy model for all of $f$ which is a certain extension of a weak geometric model.

Fact~\ref{FactGeometricCharacterization} asserts the equivalence of various formulations of geometricity of an \eg\ stratum. The proof of Fact~\ref{FactGeometricCharacterization}, which follows several results of \BookOne, takes up subsections~\ref{SectionGeometricFFS} and~\ref{SectionGeometricityProof}.

\subsubsection{Defining weak geometric models and geometric strata.} 
\label{SectionGeometricDefs}
The definitions of various 1-dimensional topological representatives of elements of $\Out(F_n)$ --- train track maps and relative train track maps \BH, improved relative train track maps \BookOne, and \cts\ \recognition --- can each be broken into two parts: static data consisting of a filtered marked graph satisfying certain conditions; and dynamic data consisting of a homotopy equivalence of that graph, again satisfying various conditions.

Given a \ct\ $f \from G \to G$ and an \eg\ stratum $H_r \subset G$, Definition~\ref{DefWeakGeomModel} of a weak geometric model and~\ref{DefGeomModel} of a geometric model, are similarly broken up into static data and dynamic data. The static data describes how to glue up a 2-complex from a surface and a graph, and how to mark that 2-complex by a certain homotopy equivalence to~$G$. The dynamic data describes a dynamical system on the 2-complex obtained by gluing together a homeomorphism of the surface and a homotopy equivalence of the graph, with conditions that describe relations amongst the parts of the dynamical system and the given \ct\ $f$, these relations all expressed as commutative or homotopy commutative diagrams.

\medskip\noindent
\textbf{Notation fixed for the remainder of Section~\ref{SectionGeometricModelsAndStrata}.} We fix the following:

\textbf{$\bullet$} A rotationless $\phi \in \Out(F_n)$;  

\textbf{$\bullet$} A \ct\ $f \from G \to G$ representing $\phi \in \Out(F_n)$;  

\textbf{$\bullet$} An \eg\ stratum $H_r \subset G$.


\begin{definition}
\label{DefWeakGeomModel}
\textbf{Weak geometric model for the \eg\ stratum $H_r$ of $f \from G_r \to G_r$.} \\
The static data of the weak geometric model is as follows:
\begin{enumerate}
\item  \label{ItemWeakGMSurface}
We are given a compact, connected surface $S$ of negative Euler characteristic having nonempty boundary with components $\bdy S = \bdy_0 S \union\cdots\union \bdy_m S$ ($m \ge 0$). We refer to $\bdy_0 S$ as the \emph{upper boundary} of $S$ and to $\bdy_1 S, \ldots, \bdy_m S$ as the \emph{lower boundaries}. 
\item  \label{ItemWeakGMAttach}
For each lower boundary $\bdy_i S$, $i=1,\ldots,m$, we are given a homotopically nontrivial closed edge path $\alpha_i \from \bdy_i S \to G_{r-1}$. The map $\alpha_i$ need not be a local embedding. 
\item \label{ItemWeakGMY}
We let $Y$ be the 2-complex defined as the quotient of the disjoint union $S \disjunion G_{r-1}$ obtained by gluing each lower boundary circle $\bdy_i S$ to $G_{r-1}$ using $\alpha_i$ as a gluing map. We let $j \from S \disjunion G_{r-1} \to Y$ denote the quotient map. 
\end{enumerate}
Several subsets of the domain of the quotient map $j \from S \disjunion G_{r-1} \to Y$ are topologically embedded in $Y$, and we will by convention identify these subsets with their images in $Y$. These include: $G_{r-1}$ and $\bdy_0 S$, each identified with a subcomplex of $Y$; the noncompact surface $\interior(S)$ is identified with an open subset of $Y$; and the noncompact surface-with-boundary $\interior(S) \union \bdy_0 S$ is identified with an open subset of $Y$. 
\begin{enumeratecontinue}
\item \label{ItemEmbAndRetract}
We are given an embedding $G_r \inject Y$ which extends the embedding $G_{r-1} \inject Y$, and we identify $G_r$ with its image in $Y$. These satisfy the following properties:
\begin{enumerate}
\item \label{ItemInteriorBasePoint}
$G_r \cap \partial_0 S$ is a single point denoted $p_r$, and there is a closed indivisible Nielsen path $\rho_r$ of height~$r$ in $G_r$ and based at $p_r$, such that the loop $\bdy_0 S$ based at $p_r$ and the path $\rho_r$ are homotopic rel base point in $Y$.
\item \label{ItemOpenDiscInS}
$Y - (G_r \union \bdy_0 S)$ is homeomorphic to the open 2-disc.
\item \label{ItemDefRetract}
There is a deformation retraction $d \from Y \to G_r$ such that the restriction $d \restrict \bdy_0 S$ is a parameterization of $\rho_r$.
\item \label{ItemGrEmbedsInY} The composition $G_{r-1} \xrightarrow{j} Y \xrightarrow{d} G_r$ is the inclusion map. 
\item\label{ItemRelInteriorHr}
The interior of $H_r$ in $G_r$ equals $H_r - G_{r-1} = H_r \intersect (Y - G_{r-1}) = H_r \intersect (\interior(S) \union \bdy_0 S)$.
\end{enumerate}
\end{enumeratecontinue}
Subitems~\pref{ItemOpenDiscInS}--\pref{ItemRelInteriorHr} follow quickly from the main item~\pref{ItemEmbAndRetract} and subitem~\pref{ItemInteriorBasePoint}; we include them here for clarity at the risk of redundancy. Item~\pref{ItemOpenDiscInS} is true because if not then by an application of Van Kampen's theorem the graph $G_r \union \bdy_0 S$ would be $\pi_1$-injective in $Y$, contradicting that $\bdy_0 S$ and $\rho_r$ are homotopic rel $p_r$ in~$Y$. The existence of a deformation retraction $d \from Y \to G_r$ which is locally injective on $\bdy_0 S$ is immediate, and since $d \restrict \bdy_0 S$ and $\rho_r$ are two immersed loops homotopic rel $p$ in $G_r$ it follows that they are equal up to reparameterization. Items \pref{ItemGrEmbedsInY}, \pref{ItemRelInteriorHr} follow immediately.

The dynamic data of the weak geometric model is as follows:
\begin{enumeratecontinue}
\item \label{ItemWGMDynSyst}
We are given a homotopy equivalence $h \from Y \to Y$, and we are given a homeomorphism $\Psi \from (S,\bdy_0 S) \to (S,\bdy_0 S)$ with pseudo-Anosov mapping class $\psi \in \MCG(S)$, subject to the following compatibility conditions: 
\begin{enumerate}
\item\label{ItemCTHomComm}
The maps $(f \restrict G_r) \composed d$ and $d \composed h \from Y \to G_r$ are homotopic.
\item\label{ItemPsAnHomComm}
The maps $j \composed \Psi$ and $h \composed j \from S \to Y$ are homotopic. 
\end{enumerate}
\end{enumeratecontinue}
To summarize, a weak geometric model of the \ct\ $f \from G_r \to G_r$ for the \eg\ stratum $H_r$ is a tuple of data $(Y,d,S,(\bdy_i S)_{i=0}^m,j,(\alpha_i)_{i=1}^m,h,\Psi)$ as described and satisfying the conditions above. Our usual terminology will suppress all of the data except $Y$, saying that \emph{$Y$ is a weak geometric model of $f \from G_r \to G_r$ for $H_r$} (or ``with respect to $H_r$'', or other such phrases). When the \ct\ $f$ is understood we simply say \emph{$Y$ is a weak geometric model for $H_r$}. This completes Definition~\ref{DefWeakGeomModel}. 
\end{definition}

\begin{definition}
\label{DefGeometricStratum}
\textbf{Geometricity of the \eg\ stratum $H_r$.}
We say that the \eg\ stratum $H_r$ of the \ct\ $f \from G \to G$ is \emph{geometric} if a weak geometric model of $f \from G_r \to G_r$ for $H_r$ exists. 
\end{definition}

\subparagraph{Remark: Comparing definitions of geometricity: Definition~\ref{DefGeometricStratum} versus Definition~5.1.4 of \BookOne.} These two definitions of geometricity of $H_r$ are logically equivalent. The definitions are similarly structured, defining geometricity in terms of the existence of a certain homotopy model for $f \from G_r \to G_r$ which incorporates a pseudo-Anosov surface homeomorphism, but the models used in the two definitions have some differences of expression and mathematical detail. Nonetheless existence of these two models is equivalent, because each is equivalent to existence of a closed, indivisible Nielsen path $\rho_r$ of height~$r$. For Definition~\ref{DefWeakGeomModel} this is part of Fact~\ref{FactGeometricCharacterization} stated below. For \BookOne\ Definition~5.1.4 this equivalence is proved by quoting results of \BookOne\ as follows: combining Definition~5.1.4 with Theorem~5.1.5~(eg-iii) one obtains existence of $\rho_r$; and in the other direction, if a closed, indivisible Nielsen path $\rho_r$ of height~$r$ exists then by Lemma~5.1.7 (or see Fact~\ref{FactEGNielsenCrossings}~\pref{ItemEGNielsenClosed} above) the path $\rho_r$ crosses each edge of $H_r$ exactly twice, and then by Proposition~5.3.1, $H_r$ satisfies Definition~5.1.4. 

The reader may also be interested in a more direct comparison between the two types of geometric models. Existence of a geometric model as in Definition~\ref{DefWeakGeomModel} fairly directly implies existence of the model built into the \BookOne\ definition. First, in Definition~\ref{DefWeakGeomModel} we do not express annulus neighborhoods of $\bdy_i S$ as they are expressed in \BookOne, although the structure of those neighborhoods and the requirements thereon are easily recovered. Also, our item~\pref{ItemWeakGMAttach} does not require the attaching maps $\alpha_i$ to be local embeddings, as they are required to be in \BookOne; that requirement is easily restored by homotoping the attaching maps $\alpha_i$, however doing so may destroy the possibility of the embedding $G_r \inject Y$. Fortunately, while our item~\pref{ItemEmbAndRetract} does require this embedding and the deformation retraction $d \from Y \to G$, all that is required in place of $d$ in the \BookOne\ definition is a homotopy equivalence $Y \mapsto G$. Also, our \pref{ItemPsAnHomComm} requires commutativity only up to homotopy, not commutativity on the nose as required in \BookOne, but such stronger commutativity is also easily recovered. From this discussion it is easy to fill in the details and prove that geometricity of $H_r$ as in Definition~\ref{DefGeometricStratum} implies geometricity of $H_r$ as in \BookOne.  

In the other direction, most items of Definition~\ref{DefGeometricStratum} are recovered from Definition~5.1.4 of \BookOne\ in a similarly easy fashion, with the exception of item~\pref{ItemEmbAndRetract}, where we require the embedding $G_r \inject Y$ and deformation retraction $d \from Y \to G_r$. The reader who knows the proof of \BookOne\ Proposition~5.3.1, or who follows closely the proof of Fact~\ref{FactGeometricCharacterization} which is modeled on \BookOne\ Proposition~5.3.1, will see that if one weakens \pref{ItemEmbAndRetract} appropriately by replacing the deformation retraction $d \from Y \to G_r$ with a homotopy equivalence as in \BookOne, then the full strength of \pref{ItemEmbAndRetract} can nevertheless be recovered (see the ``Remark'' shortly into the proof of Fact~\ref{FactGeometricCharacterization}).

In fact the gist our proof of Fact~\ref{FactGeometricCharacterization}, immediately below, will be to go through the proof of \BookOne\ Proposition~5.3.1, tweaking it and using various of its parts, in order to directly verify item~\pref{ItemEmbAndRetract} and other items of Definition~\ref{DefWeakGeomModel}.

\begin{fact}[Characterization of geometric strata]
\label{FactGeometricCharacterization}
The following are equivalent:
\begin{enumerate}
\item\label{ItemEGStratumIsGeometric}
$H_r$ is a geometric stratum.
\item\label{ItemEGNPIsClosed}
There exists a closed, height~$r$ indivisible Nielsen path $\rho_r$.
\item\label{ItemEGNPCrossesEachTwice}
There exists a height~$r$ indivisible Nielsen path $\rho_r$ which crosses each edge of $H_r$ exactly twice.
\end{enumerate}
Furthermore, if these hold then each component of $G_{r-1}$ is noncontractible.
\end{fact}

The proof is found in Section~\ref{SectionGeometricityProof}, with preliminary material in Section~\ref{SectionGeometricFFS}.

\subsubsection{Defining geometric models.}
\label{SectionDefiningGeometricModels}
We continue with the notation fixed in Section~\ref{SectionGeometricDefs}.

A ``weak'' geometric model for $H_r$ extends by a unique construction to a (``strong'') geometric model for $H_r$. The former consists of a certain 2-complex $Y$ and homotopy equivalence $h \from Y \to Y$ which is a homotopy model of the restricted map $f \from G_r \to G_r$, the latter is an extension of $h \from Y \to Y$ to a homotopy model $h \from X \to X$ for the entire map $f \from G \to G$.


\begin{definition} 
\label{DefGeomModel}
\textbf{Geometric model for the \eg\ stratum $H_r$ of $f \from G \to G$.} \index{geometric model} 
\\
Given a weak geometric model $Y$ of $f \from G_r \to G_r$ relative to $H_r$, and incorporating all the notation of Definition~\ref{DefWeakGeomModel}, a geometric model of $f \from G \to G$ relative to $H_r$ is defined as follows:
\begin{enumerate}
\item\label{ItemFullGeomModel}
Let $X$ be the 2-complex obtained as the quotient of the disjoint union $G\, \disjunion\, Y$ obtained by identifying the two copies of $G_r$, one embedded in $G$ and the other embedded in~$Y$, via the identity map on~$G_r$. We identify $G,Y$ with their image subcomplexes in~$X$ under the quotient map $G \, \disjunion \, Y \to X$.
\item\label{ItemDefRetrRestrict}
Let $d \from X \to G$ be the deformation retraction obtained by extension of the deformation retraction $d \from Y \to G_r$, requiring that the restriction of $d$ to $G \setminus G_r = X \setminus Y$ be the identity.
\item\label{ItemXYDefRetrExt}
Let $h \from X \to X$ be the homotopy equivalence obtained by extension of the homotopy equivalence $h \from Y \to Y$, requiring that $h \restrict (G \setminus G_r) = f \restrict (G \setminus G_r)$. Note that the maps $d \composed h$ and $f \composed d \from X \to G$ are homotopic. 
\end{enumerate}
This completes the defining items~\pref{ItemFullGeomModel}--\pref{ItemXYDefRetrExt} of the geometric model. Again, the terminology in full is that ``$X$ (with accompanying data) is a geometric model of $f \from G \to G$ for $H_r$'', and in brief we say that ``$X$ is a geometric model for $H_r$''. We will also say that $Y \subset X$ is the weak geometric submodel.

We continue with an additional item that follows immediately from the above defining items, plus some definitions/terminologies/notations which are used in connection with geometric models.

\begin{enumeratecontinue}
\item\label{ItemAttPts}
The frontier in $X$ of the subset $\interior(S) \subset Y \subset X$ is the disjoint union of $j(\bdy S)$ with the finite set 
$$\interior(S) \intersect (\union_{s>r} H_s) = \bigl(H_r \intersect (\union_{s>r} H_s)\bigr) - \{p_r\}
$$
Each element of the latter set is a vertex of $H_r$ called an \emph{attaching point} of $X$.
\end{enumeratecontinue}
\noindent
The composition $S \xrightarrow{j} Y \xrightarrow{d} G_r \subset G$ is equal to the composition $S \xrightarrow{j} Y \subset X \xrightarrow{d} G$, and by abusing notation we shall write this map as $d \restrict S \from S \to G$. Subject to choices of base points and paths between them, the map $d \restrict S$ induces a homomorphism $d_* \from \pi_1 S \to \pi_1 G \approx F_n$, and an application of Van Kampen's theorem implies that this homomorphism is injective (or see Lemma~\ref{LemmaLImmersed}). The image subgroup $d_*(\pi_1 S) \subgroup F_n$ has conjugacy class denoted $[\pi_1 S]$ or just~$[S]$, and this conjugacy class is well-defined independent of choices of base points and paths; we refer to $[\pi_1 S]$ as the \emph{surface subgroup system} of the geometric model.

Also, for $0 \le i \le m$ the closed curve $d \restrict \bdy_i S = \gamma_i$ with its two orientations determines an unordered pair of inverse conjugacy classes in $F_n$ denoted $[\bdy_i S]^\pm = \{[\bdy_i S],[\bdy_i S]^\inv\}$; the assignment of one as the positive and the other as the negative orientation is usually not relevant, except for situations where an orientation of $S$ and the induced boundary orientation on $\bdy S$ is under consideration (see also \recognition\ Definition~4.1, ``unoriented conjugacy classes''). Note that $[\bdy_i S]$ need not be a root-free conjugacy class, and if $i \ne j$ then the conjugacy classes $[\bdy_i S]$, $[\bdy_j S]$ might have a common power or might even be equal. However, $[\bdy_0 S]$ is always root-free and never has a common power with $[\bdy_i S]$, $i=1,\ldots,m$. We refer to the set $[\bdy S]^\pm = \union_{i=0}^m [\bdy_i S]^\pm$ as the \emph{peripheral conjugacy classes} of the geometric model. We also say that $[\bdy_0 S]^\pm$ are the \emph{top} peripheral conjugacy classes, and that $\union_{i=1}^m[\bdy_1 S]^\pm$ are the \emph{bottom} peripheral conjugacy classes. When orientation is irrelevant we shall abuse notation and work with $[\bdy S] = \{[\bdy_0 S],\ldots,[\bdy_m S]\}$ instead of with $[\bdy S]^\pm$.

This completes Definition~\ref{DefGeomModel} with all its associated terminology/notation.
\end{definition}

\subsubsection{Invariant free factor systems associated to a geometric model.}
\label{SectionGeometricFFS}
In this section we continue with the notation fixed in Section~\ref{SectionGeometricDefs}.

Before starting the proof of Fact~\ref{FactGeometricCharacterization} we state and prove a lemma describing some properties of various $\phi$-invariant free factor systems associated to a geometric model, in particular the free factor supports of the peripheral conjugacy classes $\bdy S$ and of the subgroup $\pi_1 S$ itself. 

For the reader who wishes to get quickly to the proof of Fact~\ref{FactGeometricCharacterization} in Section~\ref{SectionGeometricityProof}, only item~\pref{ItemTopIsNotLower} of Lemma~\ref{LemmaScaffoldFFS} is needed for that proof. The other items of Lemma~\ref{LemmaScaffoldFFS} will be used in various later contexts employing geometric models.


\begin{lemma} 
\label{LemmaScaffoldFFS} If $X$ is a geometric model of $f \from G \to G$ and $H_r$ then, adopting all the notation of \ref{DefGeomModel}, we have:
\begin{enumerate}
\item\label{ItemBdyAndSurfSupps}
$\F_\supp[\bdy S] = \F_\supp[\pi_1 S]$
\item\label{ItemTopIsNotLower}
$\F_\supp[\bdy_0 S] \not\sqsubset [\pi_1 G_{r-1}]$
\item\label{ItemInvFFSFromGeometric}
Each of the following free factor systems is $\phi$-invariant: 
$$\F_\supp[\pi_1 S] = \F_\supp[\bdy S], \quad \F_\supp[\bdy_0 S], \quad\text{and}\quad \F_\supp\{[\bdy_1 S],\ldots,[\bdy_m S]\}
$$
\item\label{ItemRelFFS}
$\F_\supp\bigl([\pi_1 G_{r-1}],[\pi_1 S]\bigr) = [\pi_1 G_r]$
\item\label{ItemFFSLowerBdys}
Given $t < u < r$ such that $G_t,G_u$ are core filtration elements, if $\F_\supp\{[\bdy_1 S],\ldots,[\bdy_m S]\} \sqsubset [\pi_1 G_t]$ then $[\pi_1 G_u] \not\sqsubset \F_\supp\{[\pi_1 G_t],[\pi_1 S]\}$.
\end{enumerate}
\end{lemma}

\begin{proof} Let $G^0_r$ be the component of $G_r$ containing $H_r$, let $Y^0$ be the component of $Y$ containing the connected set $j(S)$, and note that the deformation retraction $d \from Y \to G_r$ restricts to a deformation retraction $d' \from Y^0 \to G^0_r$. Noting that $F'=\pi_1(Y^0)=\pi_1(G^0_r)$ is a free factor of $F_n$, clearly we have
$$\F_\supp\{[\bdy_1 S],\ldots,[\bdy_m S]\} \sqsubset \F_\supp[\bdy S] \sqsubset \F_\supp(\pi_1 S) \sqsubset [F']
$$
For the proofs of \pref{ItemBdyAndSurfSupps} and~\pref{ItemTopIsNotLower}, by restricting to $F'$ we may assume that $F'=F_n$ and so $Y=Y^0=X$ and $G=G_r = G^0_r$. 

Item~\pref{ItemBdyAndSurfSupps} will follow once we prove the reverse inclusion $\F_\supp(\pi_1 S)  \sqsubset \F_\supp[\bdy S]$, which we do using Stallings' method from \cite{Stallings:transversality}. That inclusion is obvious if $\F_\supp[\bdy S] = \{[F_n]\}$, so suppose that $\F_\supp[\bdy S]$ is a proper free factor system of~$F_n$. Let $L$ be an $F_n$-marked graph having a proper subgraph $K$ with noncontractible components such that $[\pi_1 K]=\F_\supp[\bdy S]$. Consider the $\pi_1$-injective map $f \from S \to L$ obtained by composing $S \xrightarrow{j} Y = X \xrightarrow{d} G \mapsto L$ where the final map is a homotopy equivalence preserving marking. We may homotope this map so that $f(\bdy S) \subset K$. By general position we may perturb $f$ to be transverse to the set $M$ consisting of the midpoints of the edges of $L \setminus K$. By $\pi_1$-injectivity of $f$, each component of $f^\inv(M)$ is a circle bounding a disc in the interior of $S$. Choosing an innermost such disc~$D$, and using that $G$ is an Eilenberg-Maclane space, by a further homotopy supported on a neighborhood of $D$ we may alter the map so as to remove $\bdy D$ from the inverse image of $M$. By induction, we may assume $f(S) \intersect M = \emptyset$. By a further homotopy rel $\bdy S$ we may push $S$ entirely into~$K$. Since $S$ is connected, we have pushed it into a single component of~$K$, from which it follows that $\F_\supp[\pi_1 S] \sqsubset [\pi_1 K]=\F_\supp[\bdy S]$.

Item~\pref{ItemTopIsNotLower} is a consequence of Fact~\ref{FactEGNielsenCrossings} and Definition~\ref{DefGeomModel}~\pref{ItemInteriorBasePoint} which together imply that $[\bdy_0 S]$ is represented by a height~$r$ circuit in~$G=G_r$. Here is another proof that avoids Fact~\ref{FactEGNielsenCrossings}, again based on Stallings method. Arguing by contradiction, if $\F_\supp(\bdy_0 S) \sqsubset [G_{r-1}]$ then the deformation retraction $d \from X \to G$ may be homotoped relative to $G_{r-1}$ so that it takes $\bdy_0 S$ to $G_{r-1}$. Again we may assume that $d$ is transverse to the set $M$ of midpoints of edges of $G \setminus G_{r-1}$. Each component of $f^\inv(M)$ is an embedded circle in $\interior(S)$ which is homotopically trivial in $X$. By $\pi_1$-injectivity of $j \from S \to X$, each such circle is homotopically trivial in $S$ and so bounds a disc in $\interior(S)$. Using this fact and proceeding just as above, we may alter $d$ by homotopy rel $G_{r-1}$ so that $d(S) \subset G_{r-1}$, contradicting that $d \from S \to G$ is a homotopy equivalence.

To prove~\pref{ItemInvFFSFromGeometric}, $\phi$-invariance of $[\pi_1 S]$ and of $[\bdy S]$ is built into Definition~\ref{DefGeomModel}, and so their free factor supports are $\phi$-invariant as well. From \pref{ItemTopIsNotLower} and Definition~\ref{DefGeomModel} it follows that $[\bdy_i S]$ is supported by $[\pi_1 G_{r-1}]$ if and only if $i \ne 0$, and together with $\phi$-invariance of $[\pi_1 G_{r-1}]$ it follows that $\phi$ preserves $[\bdy_0 S]$ and the set $\{[\bdy_1 S],\ldots,[\bdy_m S]\}$, and hence their free factor supports are preserved as well.

To prove~\pref{ItemRelFFS}, clearly $[\pi_1 G_{r-1}] \sqsubset [\pi_1 G_r]$, and from the definition of geometric models it follows that $[\pi_1 S] \sqsubset [\pi_1 G_r]$, and so $\F_\supp\bigl([\pi_1 G_{r-1}],[\pi_1 S]\bigr) \sqsubset [\pi_1 G_r]$. It follows from~\pref{ItemTopIsNotLower} that $\F_\supp\bigl([\pi_1 G_{r-1}],[\pi_1 S]\bigr)$ is not carried by $[\pi_1 G_{r-1}]$, and from \pref{ItemInvFFSFromGeometric} that it is $\phi$-invariant. But by (Filtration) in the definition of a \ct\ (Definition~\ref{DefSplittings}) there are no $\phi$-invariant free factor systems properly between $[\pi_1 G_{r-1}]$ and $[\pi_1 G_r]$, and so the equation~\pref{ItemRelFFS} holds.

To prove~\pref{ItemFFSLowerBdys}, by homotoping the attaching maps of the lower boundaries, namely the maps $\alpha_i \from \bdy_i S \to G_{r-1}$ for $i=1,\ldots,m$, we obtain a marked 2-complex $X'$ which deformation retracts to a filtered marked graph $G'$ such that $[\pi_1 G_j] = [\pi_1 G'_j]$ for $j=t,u,r-1,r$, and such that $G'_t \union S$ deformation retracts to a core subgraph of $G'_t \union H'_r$. It follows that $\F_\supp\{[\pi_1 G_t],[\pi_1 S]\} \sqsubset [\pi_1(G'_t \union H'_r)]$. Since the core subgraph $G'_u$ is not a subgraph of $G'_t \union H'_r$ it follows that $[\pi_1 G'_u] \not\sqsubset [\pi_1(G'_t \union H'_r)]$ and so $[\pi_1 G_u]=[\pi_1 G'_u] \not\sqsubset \F_\supp\{[\pi_1 G_t],[\pi_1 S]\}$.

\end{proof}

\subsubsection{Characterizing geometric strata: Proof of Fact \ref{FactGeometricCharacterization}}
\label{SectionGeometricityProof}

The ``furthermore'' clause follows from Fact~\ref{FactNielsenBottommost}.

The implication \pref{ItemEGStratumIsGeometric}$\implies$\pref{ItemEGNPIsClosed} for ``improved relative train track representatives'' is contained in \BookOne, Theorem~5.1.5, item eg-(iii), and the proof of this particular item is found on page~590 of \BookOne. 

In our present \ct\ context, we prove \pref{ItemEGStratumIsGeometric}$\implies$\pref{ItemEGNPIsClosed} as follows. Assuming $H_r$ to be geometric, consider a weak geometric model $Y$ as notated in Definition~\ref{DefGeometricStratum}, with quotient map $j \from G_{r-1} \disjunion S \to Y$. Since $\bdy_0 S$ is preserved by the pseudo-Anosov homeomorphism $\Psi \from S \to S$, it follows that the circuit $c = d_\#(\bdy_0 S)$ is invariant by~$f_\#$. By Fact~\ref{FactNielsenCircuit}, the circuit $c$ splits completely into fixed edges and indivisible Nielsen paths. By Lemma~\ref{LemmaScaffoldFFS}~\pref{ItemTopIsNotLower} the circuit $c$ has height~$r$, and so one of the terms of the complete splitting of $c$ must have height~$r$, but that term cannot be a fixed edge, so by Fact~\ref{FactEGNPUniqueness} that term is the unique (up to reversal) indivisible Nielsen path $\rho_r$ of height~$r$. The path $\rho_r$ must be closed for otherwise, by Fact~\ref{FactEGNielsenCrossings}~\pref{ItemNoClosed}, one of the endpoints of $\rho_r$ is not contained in $G_{r-1}$, but then the term of the complete splitting of $c$ incident to that endpoint is an edge of $H_r$ and so is not a fixed edge nor an indivisible Nielsen path, a contradiction.

\smallskip
\textbf{Remark.} Observe that this proof of \pref{ItemEGStratumIsGeometric}$\implies$\pref{ItemEGNPIsClosed} does not make use of the fact that $d \from Y \to G_r$ is a deformation retraction, only that it is a homotopy equivalence which restricts to the inclusion $G_{r-1} \subset G_r$. As such, this proof works just as stated using instead the definition of geometricity given in \BookOne\ Definition~5.1.4, and so that definition also implies item~\pref{ItemEGNPIsClosed}. Combining this with the proofs to come of \pref{ItemEGNPIsClosed}$\implies$\pref{ItemEGNPCrossesEachTwice}$\implies$\pref{ItemEGStratumIsGeometric}, we obtain a proof that geometricity as defined in \BookOne\ Definition~5.1.4 implies geometricity as defined here in Definition~\ref{DefGeometricStratum}.

\medskip

The implication \pref{ItemEGNPIsClosed}$\implies$\pref{ItemEGNPCrossesEachTwice} is the ``only if'' direction of Fact~\ref{FactEGNielsenCrossings}~\pref{ItemEGNielsenClosed}.

We now prove \pref{ItemEGNPCrossesEachTwice}$\implies$\pref{ItemEGStratumIsGeometric}. So, assume that there exists an indivisible Nielsen path $\rho_r$ of height~$r$ crossing each edge of $H_r$ exactly twice. By Fact~\ref{FactNielsenBottommost}~\pref{ItemNoZeroStrata} and Fact~\ref{FactContrComp}, each component of $G_{r-1}$ is noncontractible, a fact which we assume henceforth without comment. Let $G_r^0$ be the component of $G_r$ containing $H_r$, and hence also containing $\rho_r$. It suffices to construct a weak geometric model $Y^0$ of the restricted \ct\ $f \restrict G_r^0$ for its stratum $H_r$, because that automatically extends to a weak geometric model $Y$ of $f \restrict G_r$ for $H_r$ by extending the homotopy equivalence of $Y^0$ to a homotopy equivalence of $Y$ which, on $Y-Y^0 = G-G_r^0$, is equal to $f$. We may therefore assume that $G=G_r^0=G_r$.

In order to construct $Y$ we shall follow the method laid out in the proof of Proposition 5.3.1 of \BookOne. All of the hypotheses of that proposition are satisfied except that we replace the hypothesis ``$f$ is $\F$-Nielsen minimized'' by the statement (EG Nielsen Paths) of Definition~\ref{DefCT}. Lemma~4.18 of \recognition\rc\ says that with this replacement, the statement and proof of Proposition 5.3.1 of \BookOne\ remain valid. As said earlier, that proof mostly produces a weak geometric model, but among other things we must carefully note the steps of that proof in which item~\pref{ItemEmbAndRetract} of Definition~\ref{DefWeakGeomModel} is established. For that purpose we go through some details of the proof of \BookOne~Proposition~5.3.1, carefully citing various steps.

\medskip

Following \BookOne, page~571 in the proof of Proposition~5.3.1, let $Y=M(\rho_r)$ be the mapping cylinder of $\rho_r \from [0,1] \to G_r$, which recall is the target of the quotient map 
$$q \from Q \disjunion G_r \to M(\rho_r), \qquad Q = [0,1] \cross [0,1]
$$
that identifies $(s,0) \in Q$ to $\rho_r(s) \in G_r$ for each $s \in [0,1]$. The embedding $G_r \to Y= M(\rho_r)$ is simply the restriction of $q$. Denote the ``upper boundary'' of $Q$ to be $\bdy_u Q = \bigl(\{0\} \cross [0,1]\bigr) \union \bigl([0,1] \cross \{1\}\bigr) \union \bigl(\{1\} \cross [0,1]\bigr)$. Choose a deformation retraction $d_Q \from Q \to [0,1] \cross \{0\}$ that restricts to a homeomorphism $\bdy_u Q \to [0,1] \cross \{0\}$. Clearly $d_Q$ induces a deformation retraction $d \from M(\rho_r) =Y \to G_r$, and on $Q$ we have the commutative relation $d \composed q = q \composed d_Q$ (by construction, $d \composed q \restrict \bdy_u Q$ is a parameterization of the Nielsen path $\rho_r$). Note that already we have an embedding $G_r \mapsto Y$ and deformation retraction $d \from Y \to G_r$ satisfying item~\pref{ItemGrEmbedsInY}.

Next we define the surface $S$ and its boundary circles $\bdy_0 S,\ldots,\bdy_m S$ required for item~\pref{ItemWeakGMSurface} of Definition~\ref{DefWeakGeomModel}. Subdivide $[0,1] \approx [0,1] \cross 0 \subset Q$ into subintervals each mapped by $\rho_r$ to an edge of $G_r$, so each edge of $H_r$ occurs exactly twice. Define a quotient map $q_S \from Q \to S$ by identifying in pairs, consistent with $\rho_r$, those subintervals of $[0,1] \cross 0$ which are mapped to edges of $H_r$, and so $S$ is a compact surface with nonempty boundary. Clearly $q_S(\bdy_u Q)$ is a connected subset of $\bdy S$ and so is contained in a component which we define to be the upper boundary $\bdy_0 S$; the remaining components, if any, are enumerated arbitrarily as $\bdy_1 S,\ldots,\bdy_m S$.  

Next we define the quotient map $j \from S \,\disjunion\, G_{r-1} \to Y$ referred to in item~\pref{ItemWeakGMY} of Definition~\ref{DefWeakGeomModel}, although we do not yet specify exactly the identifications made by this quotient map that are needed in order to completely verify~\pref{ItemWeakGMY}. The quotient map $q$ clearly factors into two quotient maps as follows:
$$Q \disjunion G_r \xrightarrow{q_S} S \disjunion G_r \xrightarrow{j_r} M(\rho_r)=Y
$$
Consider the 1-complex $\kappa = q_S([0,1] \cross 0) \subset S$. The deformation retraction $d_Q \from Q \to [0,1] \cross 0]$ clearly induces a deformation retraction $S \to \kappa$ whose homotopy inverse is the inclusion $\kappa \inject S$. The edges of $\kappa$ whose interiors are contained in the interior of $S$ are labelled one-to-one by edges of $H_r$. The remaining edges of $\kappa$ form a subcomplex $\kappa_0$ labelled (not necessarily one-to-one) by edges of $G_{r-1}$. Clearly we have
$$\kappa_0 = \bdy S \setminus q_S(\bdy_u Q)
$$
Under the embedding of $G_r$ in $M(\rho_r)$, the restriction $j_r \restrict \kappa$ may be identified with the labelling map $\gamma \from \kappa \to G_r$, the image of which contains $H_r$. By restricting $j_r$ we there obtain a surjective map $j \from S \disjunion G_{r-1} \to Y$, and clearly this is a quotient map. 

We verify further items of Definition~\ref{DefWeakGeomModel} with citations from the proof of Proposition~5.3.1 on p.\ 578 of \BookOne, several of which are applications of \BookOne\ Lemma~5.3.9, using the fact that the relative train track map $f \restrict G_r$ is reduced; in our context this fact is part of property (Filtration) in the definition of a \ct. 

Next we verify \pref{ItemInteriorBasePoint} of Definition~\ref{DefWeakGeomModel}. The proof of Proposition~5.3.1 uses Lemma~5.3.9 to show that the quotient map $q_S$ identifies the endpoints of the arc $\bdy_u Q$, and so we obtain $\bdy_0 S$ as the quotient of $\bdy_u Q$ under that identification.  We have already seen that the quotient map $q \from Q \disjunion G_r \to Y$ restricts on $\bdy_u Q$ to a parameterization of $\rho_r$, and so it follows that the quotient map $j \from S \disjunion G_{r-1} \to Y$ restricts on $\bdy_0 S$ to a parameterization of $\rho_r$. It is also clear from the construction that $G_r \intersect \bdy_0 S$ is the base point $p_r$ of $\rho_r$, and we have verified Definition~\ref{DefWeakGeomModel}~\pref{ItemInteriorBasePoint}. 

Next, for $i=1,\ldots,m$ we define $\alpha_i$ and prove item~\pref{ItemWeakGMAttach} of Definition~\ref{DefWeakGeomModel}. At this stage we know that $\kappa_0 = \bdy S \setminus \bdy_0 S$ is the union of the lower boundary components $\bdy_1 S, \ldots, \bdy_m S$, and on each of these we define $\alpha_i = \gamma \restrict \bdy_i S \from \bdy_i S \to G_{r-1}$. The proof of Proposition~5.3.1 uses Lemma~5.3.9 to show that $\alpha_i$ is homotopically nontrivial. 

Next we verify item~\pref{ItemWeakGMY} of Definition~\ref{DefWeakGeomModel}. We have already verified that the quotient map $j \from S \disjunion G_{r-1} \to Y$ does indeed identify each $x \in \bdy_i S$ to $\alpha_i(x)$, for each $i=1,\ldots,m$. We must verify that $j$ identifies no other point pairs. By construction, the only point pairs of $S \disjunion G_{r-1}$ that are identified by $j$ are point pairs in $G_{r-1} \union \kappa$. Also, $j$ maps $\kappa_0$ to $G_{r-1}$ and maps the union of the interiors of the edges of $\kappa - \kappa_0$ injectively to the union of the interiors of the edges of $H_r$, which are disjoint from $G_{r-1}$. Letting $V$ be the set of vertices of $\kappa$ not in $\kappa_0$, we therefore need only consider point pairs in $G_{r-1} \union \kappa_0  \union V$. Also, the only identifications of point pairs of $G_{r-1} \union \kappa_0$ made by $j$ are those made by the union of the maps $\alpha_i \from \bdy_i S \to G_{r-1}$, which are defined just by restricting $j$ to the components of $\kappa_0 = \bdy_1 S \union\cdots\union \bdy_m S$. It therefore only remains to consider two points of $V$ or a point of $V$ and a point of $G_{r-1}$. We again cite the proof of Proposition~5.3.1, which uses Lemma~5.3.9 once again to show that for each $v \in V$ we have $j_r(\Link(\kappa, v)) = \Link(G_r, j_r(v))$, which implies that $j(v) = j_r(v) \not\in G_{r-1}$, and that if $v \ne w \in V$ then $j(v)=j_r(v) \ne j_r(w)=j(w)$. 

We turn to item~\pref{ItemWGMDynSyst} of Definition~\ref{DefWeakGeomModel}. Defining the homotopy equivalence $h \from Y \to Y$ to be the composition $Y \xrightarrow{d} G \xrightarrow{f} G \inject Y$, this clearly satisfies the dynamic Definition~\ref{DefWeakGeomModel}~\pref{ItemCTHomComm}. It remains to construct a homeomorphism $\Psi \from (S,\bdy_0 S) \to (S,\bdy_0 S)$ with pseudo-Anosov mapping class and to verity item~\pref{ItemPsAnHomComm} of Definition~\ref{DefWeakGeomModel}. These tasks are accomplished in the final portions of the proof of Proposition~5.3.1 on p.\ 579 of~\BookOne\ to which we make one last citation. In outline, the brunt of that proof is contained in Corollary~5.3.8 of~\BookOne\ which says that the homotopy equivalence inclusion $\kappa \inject S$, when composed with the map $S \xrightarrow{j} Y \xrightarrow{d} G \xrightarrow{f} G$ may be lifted to a homotopy equivalence $f_\kappa \from \kappa \to \kappa$ that permutes the free homotopy classes of the components of $\kappa_0$. One also checks that $f_\kappa$ fixes the free homotopy class of the circuit of $\kappa$ that is the image of $\bdy_0 S$ under the deformation retraction $S \to \kappa$. Conjugating $f_\kappa$ by the homotopy equivalence between $S$ and $\kappa$ one obtains a homotopy equivalence of $S$ preserving the free homotopy classes of $\bdy S$, in particular preserving the class of $\bdy_0 S$. By the Dehn-Nielsen-Baer theorem \cite{FarbMargalit:primer} this map is homotopic the desired homeomorphism $\Psi \from (S,\bdy_0 S) \to (S,\bdy_0 S)$. A further check that this map has no periodic conjugacy classes other than the boundary components implies that its mapping class is pseudo-Anosov.

This completes the proof of Fact~\ref{FactGeometricCharacterization}.

\subparagraph{Remark.} Unlike in the proof of Proposition~5.3.1, we do not in the end replace $\alpha_i$ with an immersion $\alpha'_i$ in the same free homotopy class, because this would destroy the mapping cylinder properties that are used to obtain the embedding $G_r \inject Y$ and the deformation retraction $d \from Y \to G_r$. In this regard, on the bottom of page~578 of \BookOne\ the proof of Proposition~5.3.1 incorrectly asserts that the mapping cylinder is homeomorphic to the space obtained by gluing $S$ to $Y$ using the immersions $\alpha'_i$ as gluing maps, but this assertion is inconsequential for anything else in \BookOne.

\subsection{Complementary subgraph and peripheral splitting}
\label{SectionGeomModelComplement}
For this section we fix a rotationless $\phi \in \Out(F_n)$, a representative \ct\ $f \from G \to G$, a geometric \eg-stratum~$H_r$, and a geometric model $X$ for $H_r$ with associated notation from Definition~\ref{DefGeomModel}. 

\begin{definition}[The complementary subgraph]
\label{DefComplSubgraph}
Define the \emph{complementary subgraph} of $X$ to be
$$L = \closure(X - \interior(S)) = (G \setminus H_r) \union \bdy_0 S
$$
This graph $L$ is naturally identified with the quotient of the disjoint union of $G \setminus H_r$ and an arc $E_\rho$ modulo the following identifications: if $p_r \in G \setminus H_r$---equivalently $p_r$ is not an interior point of $G_r$---then both endpoints of $E_\rho$ are identified to $p_r$; otherwise, the endpoints of $E_\rho$ are identified just with each other, forming a circle component of $L$. In either case, $E_\rho$ forms a loop in $L$ based at $p_r$ that is identified with~$\bdy_0 S$. Item~\pref{ItemInteriorBasePoint} of Definition~\ref{DefGeomModel} implies that the composition $E_\rho \inject (G \setminus H_r) \disjunion E_\rho \mapsto L \inject X \xrightarrow{d} G$ is a parameterization of the Nielsen path $\rho$. Definition~\ref{DefGeomModel}~\pref{ItemAttPts} implies that $L \intersect \interior(S)$ is just the set of attaching points of $X$.

Listing the noncontractible components of $L$ as $\{L_l\}$, associated to $L$ is the subgroup system in $\pi_1(X)\approx F_n$ defined by $[\pi_1 L] = \{[d_* \pi_1(L_l)]\}$. 
\end{definition}

Lemma~\ref{LemmaLImmersed} addresses malnormality, or failure thereof, for the subgroup systems of $F_n$ associated to the complementary subgraph $L$ and the surface $S$. The proof will be an application of Bass-Serre theory, depending on the construction the ``peripheral splitting'' of $F_n$ that is associated to $X$. This is a certain graph of groups presentation of $F_n$ obtained by splitting $X$ at the components of $\bdy S$ and the attaching points. In Section~\ref{SectionFreeBoundaryInvariant} we shall also use Bass-Serre theory to characterize the ``free boundary circles'' of $X$.

Lemma~\ref{LemmaLImmersed} will be used in several later arguments in this section, including the proofs of Propositions~\ref{PropGeomLams} and~\ref{PropVertToFree}.



\begin{lemma}
\label{LemmaLImmersed}
The maps $L \xrightarrow{d} X$ and $S \xrightarrow{j} X$ are $\pi_1$-injective (for the latter see also the discussion at the end of Definition~\ref{DefGeomModel}). Furthermore:
\begin{enumerate}
\item\label{ItemComplementMalnormal}
The subgroup system $[\pi_1 L]$ in $\pi_1(X) \approx F_n$ is malnormal and has one component for each noncontractible component $L_l$, that is, if $l \ne l'$ then $[d_*\pi_1 L_l] \ne [d_*\pi_1 L_{l'}]$. As a consequence:
\begin{enumerate}
\item\label{ItemComplementCircuits}
Distinct circuits in $L$ map via $d$ to distinct circuits in $G$. 
\item\label{ItemKInL}
For any subgraph $K \subgroup L$ with noncontractible components $\{K_j\}$, each map $K_j \inject G \inject X$ is $\pi_1$-injective, the subgroup system $[\pi_1 K] = \{[d_* \pi_1 K_j]\}$ in $\pi_1(X) \approx F_n$ is malnormal, and if $j \ne j'$ then $[d_* \pi_1 K_j] \ne [d_* \pi_1 K_{j'}]$.
\end{enumerate} 
\item\label{ItemSeparationOfSAndL} 
The subgroup $d_* \pi_1 S \subgroup F_n$ is its own normalizer (although it need not be malnormal). For any nontrivial $c \in \pi_1 S \inject F_n$ with conjugacy class $[c]$ in $F_n$, if $[c]$ is carried by $[\pi_1 L]$, or if there exists $c' \in \pi_1 S$ such that $c,c'$ are conjugate in $F_n$ but not in $\pi_1 S$, then $c$ is peripheral in $\pi_1 S$ meaning that $c$ is represented by a loop in $\bdy S$.
\end{enumerate}
\end{lemma}

For an example where $\pi_1 S$ is not malnormal, let $S$ have three boundary components, and attach the two lower boundary components to the same circuit in $G_{r-1}$.

To define the peripheral splitting of $F_n$ we shall use the method of \cite{ScottWall} by first constructing a graph of spaces, whose underlying 2-complex $\wh X$ is obtained from $X$ by blowing up the attaching points in a certain manner, as follows. 

\begin{definition}[Blowing up attaching points]
\label{DefBlowupAttaching}
Let the attaching points of $X$ be enumerated as $\{z_k\}_{k=1}^q$. Let $v_k \ge 1$ be the valence of $z_k$ in the graph~$L$. Choose a regular neighborhood $U$ of the finite set $\{z_k\}$, with component $U_k$ containing~$z_k$. We have $U_k = (U_k \intersect \interior(S)) \union (U_k \intersect L)$ where the subset $U_k \intersect \interior(S)$ is an open disc in $\interior(S)$ containing $z_k$, the subset $U_k \intersect L$ is a union of $v_k$ half-open arcs with common endpoint $z_k$, and the intersection of those two subsets is the point~$z_k$. Let $\wh X$ be the 2-complex obtained from $X$ by pulling each attaching point $z_k$ apart, replacing $U_k$ with the disjoint union of the two sets $U_k \intersect \interior(S)$ and $U_k \intersect L$ together with an arc $\eta_k$ having one endpoint $z^S_k$ identified with the copy of $z_k$ in $U_k \intersect \interior(S)$ and having opposite endpoint $z^L_k$ identified with the copy of $z_k$ in $U_k \intersect L$. We note that the quotient map $\wh X \to X$ obtained by collapsing each $\eta_k$ to the point $z_k$ is a homotopy equivalence, and so there is an induced isomorphism $\pi_1(\wh X) \to \pi_1(X) \approx F_n$ (well defined, as usual, up to inner automorphism).
\end{definition}

Recall that a \emph{graph of spaces} structure on $\wh X$ is defined by exhibiting $\wh X$ as a certain quotient of a collection of path connected \emph{edge spaces} and \emph{vertex spaces}. Each edge space is exhibited as the product of a compact arc and space called the \emph{core} of the edge space; we shall refer to the endpoints of the arc crossed with the core as the two \emph{endpoint cores} of that edge space, and the interior points of the arc crossed with the core as the \emph{interior cores}. The entire graph of spaces is obtained from the disjoint union of the vertex spaces and the edge spaces by attaching each boundary core to some vertex space via some $\pi_1$-injective map. Each edge space minus its two boundary cores forms an \emph{open edge space} which is foliated by interior cores and which maps homeomorphically via the quotient to an open subset of the graph of spaces. 

\begin{definition}[The peripheral graph of spaces]
\label{DefPeriphGraphOfSpaces}
Define the ``peripheral graph of spaces'' structure on $\wh X$ is defined as follows. The components of $\bdy S$ are enumerated as $\bdy_i S$, $i=0,\ldots,m$. Let $N$ be a closed regular neighborhood of $\bdy S$ in the surface $S$, chosen so that $N$ is disjoint from the closure of $U$. The component of $N$ containing $\bdy_i S$ is an annulus $N_i$ exhibited as the product of a compact arc and a circle. Let $\wh S = \interior(S) - \union_{i=0}^m \interior(N_i)$, a compact surface with boundary which is a deformation retract of~$\interior(S)$ and of~$S$ itself. The vertex spaces of $\wh X$ are $\wh S$ and the components of $L$. Each annulus $N_i$, $i=0,\ldots,m$ is an edge space with core being a circle: one endpoint core of $N_i$ is attached to $\wh S$ by the inclusion map $\bdy N_i \intersect \wh S \inject \wh S$, which is $\pi_1$-injective because $\wh S$ is not a disc; the other endpoint core of $N_i$ is attached to some component of $L$ by the restricted map $\bdy N_i \intersect \bdy S \xrightarrow{j} L$, and $\pi_1$-injectivity follows from Definition~\ref{DefWeakGeomModel}. Also, each closed arc $\eta_k$ is an edge space with core being a point, and with endpoint cores attached to $z^S_k \in \wh S$ and $z^L_k \in L$. 
\end{definition}

\begin{definition}[The peripheral splitting of $F_n$ associated to $X$] 
\label{DefPeripheralSplitting}
This is the graph of groups presentation $\Gamma(X)$ of $F_n \approx \pi_1(\wh X)$ that is associated to the graph of spaces structure on $\wh X$, as explained in \cite{ScottWall}. In detail, the underlying graph of $\Gamma(X)$ is the quotient of $\wh X$ obtained by collapsing each vertex space to a vertex of $\Gamma(X)$ and collapsing each edge space to an edge of $\Gamma$ by collapsing each of its interior cores to a point. The graph $\Gamma(X)$ has one ``$S$-vertex'' obtained by collapsing $\wh S$, and a finite set of ``$L$-vertices'' obtained by collapsing the components of $L$. Also $\Gamma(X)$ has one edge for each $N_i$ and one for each $z_k$. The graph $\Gamma(X)$ is bipartite: each edge has one endpoint on the $\wh S$ vertex and one endpoint amongst the $L$ vertices. By choosing a base point in each vertex space and in the core of each edge space, and then taking fundamental groups, we associate a vertex group to each vertex of $\Gamma(X)$ and an edge group to each edge; in particular, the group of the $S$-vertex is $\pi_1 S$, the group of each $L$-vertex is $\pi_1$ of the associated component of $L$, the group of each $N_i$ edge space is infinite cyclic, and the group of each $\eta_k$ edge space is trivial. By appropriate choice of paths connecting up base points, we associate to each endpoint of each edge a monomorphism from the edge group of that edge to the the vertex group of that endpoint. The fundamental group $\pi_1(\Gamma(X))$ is defined as in \cite{Serre:trees} or in \cite{ScottWall}, as is the isomorphism $\pi_1(\Gamma(X)) \approx \pi_1(\wh X) \approx F_n$, well defined up to inner automorphism.

Let $T(X)$ be the Bass-Serre tree of $\Gamma(X)$, on which $\pi_1(\Gamma(X)) \approx F_n$ acts with quotient graph of groups $\Gamma(X) \approx T(X) / F_n$. A vertex of $T(X)$ is referred to as an $L$-vertex or $S$-vertex according to which its image under the orbit map $T(X) \mapsto \Gamma(X)$ is an $L$-vertex or $S$-vertex.
\end{definition}

\begin{remark} \label{RemarkNotMinimalReFree}
The action of $F_n$ on $T(X)$ need not be minimal. In fact if an $L$-vertex of $T(X)$ projects to an $L$-vertex of $\Gamma(X)$ associated to a ``free boundary circle'' in $Y$, meaning the image in $Y$ of a component $\bdy_i S$ which embeds in $Y$ and does not locally separate $Y$, then that $L$-vertex of $T(X)$ has valence~$1$. This happens in particular when the stratum $H_r$ is the top stratum and $\bdy_i S = \bdy_0 S$ is the top boundary circle.
\end{remark}

\begin{proof}[Proof of Lemma \ref{LemmaLImmersed}] $\pi_1$-injectivity of the maps $L \to G \inject X$ and $S \inject X$ follows directly from the general fact that vertex groups of any graph of groups inject into its fundamental group. 

To prove \pref{ItemComplementMalnormal}, consider two $L$-vertices $V \ne W$ of $T(X)$. We must prove that $\Stab(V) \intersect \Stab(W)$ is trivial. Since $\Gamma(X)$ and hence $T(X)$ is bipartite, the arc $\overline{VW}$ contains some $S$-vertex $U$. Letting $E \ne E'$ be the edges of $\overline{VW}$ incident to $U$ it follows that $\Stab(V) \intersect \Stab(W) \subgroup \Stab(E) \intersect \Stab(E')$, and so it suffices to prove that $\Stab(E) \intersect \Stab(E')$ is trivial. Suppose it is nontrivial. The surface $\wh S$ has a hyperbolic structure which arises as follows: for some properly discontinuous, cocompact action $\pi_1(\wh S) \act \hyp^2$ whose limit set is a $\pi_1(\wh S)$-invariant Cantor set $C \subset \bdy\hyp^2$, the group $\pi_1(\wh S)$ acts on the convex hull $\Hull(C) \subset \hyp^2$, and $\wh S$ is identified isometrically with the quotient $\Hull(C) / \pi_1(\wh S)$. There exist geodesic lines $L \ne L'$ which are components of $\bdy\Hull(C)$ such that under the action of $\pi_1(\wh S)$ on $\Hull(C)$ we have $\Stab(E)=\Stab(L)$ and $\Stab(E')=\Stab(L')$. But $\Stab(L) \intersect \Stab(L')$ is clearly trivial, because the lines $L,L'$ have disjoint endpoint pairs in $C$. Item~\pref{ItemComplementCircuits} is an immediate consequence. Item~\pref{ItemKInL} follows by using the fact that for each component $L_l$ of $L$, the components of $K$ in $L_l$ define a malnormal subgroup system in $\pi_1(L_l)$.

To prove \pref{ItemSeparationOfSAndL}, fix an $S$-vertex $V$ of $T(X)$. The identification of $T(X) / F_n$ with $\Gamma(X)$ induces an isomorphism between $\Stab(V)$ and the image of the injection $\pi_1 S \approx \pi_1(\wh S) \inject \pi_1(X)\approx F_n$. Consider  another vertex $W \ne V$ of $T(X)$ such that $\Stab(V) \intersect \Stab(W)$ is nontrivial. Letting $E$ be the first edge of the segment from $V$ to $W$ we have $\Stab(V) \intersect \Stab(W) \subset \Stab(E)$. As shown above, either $\Stab(E)$ is trivial or $\Stab(E)=\Stab(L)$ for some boundary line $L$ of the universal cover of $\wh S$, and so any element of $\Stab(V) \intersect \Stab(W)$ is peripheral in $\pi_1 S$. If $W$ is also an $S$-vertex then $\Stab(V) \ne \Stab(W)$ which implies that $\pi_1 S$ is its own normalizer in~$F_n$. Letting $c \in \pi_1 S \approx \Stab(V)$ be as in \pref{ItemSeparationOfSAndL}, the desired conclusion follows once the correct $W$ is chosen: if $[c]$ is carried by $[\pi_1 L]$ then take $W$ to be the unique $L$-vertex such that $c \in \Stab(W)$; and if $c$ is conjugate in $F_n$ to $c' \in \pi_1 S$ but $c,c'$ are not conjugate in $\pi_1 S$ then take $W$ to be the $S$-vertex $W=\theta \cdot V$ where $c' = \theta \, c \, \theta^\inv$.
\end{proof}

\subsection{The laminations of a geometric stratum} 
\label{SectionLamsGeomStratum}

The main result of this section is Proposition~\ref{PropGeomLams} which says that the lamination pair of a geometric stratum can be identified with the stable/unstable geodesic lamination pair of the associated pseudo-Anosov mapping class. This result can be bound in \BookOne, embedded in the proof of the geometric case of Proposition~6.0.8. We give a different proof, based on little bit of Nielsen--Thurston theory for surface mapping classes, and we get a little bit more out of the proof by identifying the free factor supports of these laminations with the free factor support of the surface~$S$ in the geometric model.

\subsubsection{Review of Nielsen--Thurston theory.} 
\label{SectionNTReview}
Given a finite type surface $S$, Nielsen described the action on $\bdy\pi_1 S$ of automorphisms of $\pi_1 S$ representing elements of the mapping class group of a surface~$S$. Thurston's theory of geodesic laminations on $S$ and his classification of mapping classes was related to Nielsen's theory in \cite{Miller:Nielsen} and \cite{HandelThurston:nielsen}; another main reference is \cite{CassonBleiler}. 

We review basic facts about geodesic laminations and pseudo-Anosov mapping classes, and in Proposition~\ref{PropNielsenThurstonTheory} we give a distilled version of the Nielsen-Thurston theory which is adequate for our applications. Our concern here is \emph{solely} with non-closed surfaces, and we will state results only in that case.  We shall take the liberty to state definitions and results in a manner that illuminates the connection with geometric strata and laminations in free groups; the reader may easily compare these statements with the standard forms found in the references.

\smallskip
\textbf{Hyperbolic structures.} Let $S$ be a compact surface with nonempty boundary~$\bdy S$. We assume that $\chi(S) \le -1$, and so $\pi_1 S$ is a free group of rank $\ge 2$ with Gromov boundary denoted $\bdy\pi_1 S$. Let $\wt S \union \bdy\pi_1 S$ denote the Gromov compactification of the universal cover $\wt S$, which is homeomorphic to a closed disc whose boundary is equal to $\bdy\wt S \union \bdy\pi_1 S$ in which $\bdy\pi_1 S$ sits as a Cantor subset. The deck transformation action $\pi_1 S \act \wt S$ extends uniquely to a homeomorphic action on $\wt S \union\bdy\pi_1 S$.

Let $\hyp^2$ denote the hyperbolic plane and let $\hyp^2 \union \bdy\hyp^2$ denote its compactification with the circle at infinity, identified with the Gromov compactification.

By a \emph{hyperbolic structure} on $S$ we shall mean a hyperbolic metric on $S$ with totally geodesic boundary. Fix a hyperbolic structure $\mu$ and let $\ti\mu$ be the lifted hyperbolic structure on the universal cover $\wt S$. Associated to $\mu$ is the \emph{developing map} $D^\mu \from \wt S \inject \hyp^2$, an isometric embedding, and the \emph{holonomy} $\rho_\mu \from \pi_1 S \to \Isom(\hyp^2)$, a properly discontinuous action that extends the deck transformation action of $\pi_1 S$ on~$\wt S$; these are uniquely determined by~$\mu$ up to post composition by some isometry of~$\hyp^2$. The \emph{limit set} $\Lim_\mu \subset S^1_\infinity=\bdy\hyp^2$ is the accumulation set of any orbit. The embedded image $\wt S \subset \hyp^2$ is identified with the \emph{convex hull} of the limit set $\Hull(\Lim_\mu)$, which is smallest closed subset of $\hyp^2$ containing each geodesic $(\xi,\eta) \subset \hyp^2$ such that $\xi\ne\eta \in \Lim_\mu$; we also denote $[\xi,\eta] = (\xi,\eta) \union \{\xi,\eta\} \subset \hyp^2 \union S^1_\infinity$. There exists a unique equivariant continuous homeomorphism $D^\mu_\infinity \from \bdy\pi_1 S \to \Lim_\mu$, and this homeomorphism pieces together with $D^\mu$ to produce an equivariant homeomorphism defined on the Gromov compactification 
$$D^\mu \disjunion D^\mu_\infinity \from \wt S \disjunion \bdy\pi_1 S \to \Hull(\Lim_\mu) \disjunion \Lim_\mu
$$
When $\mu$ has been specified we often use this homeomorphism to identify the domain and range.

Each line $\ti\ell \in \wt\B(\pi_1 S)$ is, formally, a 2-point subset $\{\xi,\eta\} \subset \bdy\pi_1 S \subset \bdy\hyp^2$, and its \emph{realization} in $\hyp^2$ is the bi-infinite geodesic $\ti\ell_\mu = (\xi,\eta) \subset  \Hull(\Lim_\mu) \subset \hyp^2$. For each $\ell \in \B(\pi_1 S)$, its realization $\ell_\mu$ is the bi-infinite geodesic in $S$ obtained by choosing any lift $\ti\ell \in \wt\B(\pi_1 S)$ and projecting $\ti\ell_\mu$ to $S$; this is well-defined independent of the choice.

\smallskip\textbf{Geodesic laminations.} A \emph{geodesic lamination} on $S$ is a closed subset $\Lambda \subset \B(\pi_1 S)$ such that for some (any) hyperbolic structure $\mu$ on $S$ the following hold: for each $\ell \in \Lambda$ the corresponding geodesic $\ell_\mu$ on $S$ is either a bi-infinite simple geodesic or it wraps infinite-to-one around a simple closed geodesic on $S$; if $\ell \ne \ell' \in \Lambda$ then $\ell_\mu \intersect \ell'_\mu = \emptyset$; and $\Lambda_\mu = \union_{\ell \in \Lambda} \ell_\mu$ is a compact subset of $\interior(S)$ which we refer to as the \emph{realization} of $\Lambda$ in the hyperbolic structure $\mu$. In the literature, the set $\Lambda_\mu \subset S$ itself is referred to as a geodesic lamination on $S$ with respect to $\mu$, but for present purposes we take the liberty of focussing on $\Lambda \subset \B(\pi_1 S)$ as the primary object of focus. 

A basic principle of geodesic laminations, which we assume throughout, is that their topological properties are independent of the choice of $\mu$, because for any other choice $\mu'$ there is a homeomorphism $S \mapsto S$ isotopic to the identity taking $\Lambda_\mu$ to $\Lambda_{\mu'}$. With our present point of view focussed on free group laminations, this principle is also reflected in the abstract lamination $\Lambda$ being independent of $\mu$. 

Consider a geodesic lamination $\Lambda$ on $S$ with realization $\Lambda_\mu$ and let $\wt \Lambda_\mu \subset \wt S$ be the total lift of $\Lambda_\mu$. An \emph{(upstairs, open) principal region} $\wt R \subset \wt S$ of $\wt\Lambda_\mu$ is a component of $\wt S - \wt \Lambda_\mu$, and a \emph{(downstairs, open) principal region} $R \subset S$ of $\Lambda_\mu$ is the image under the universal covering map $\wt S \mapsto S$ of some upstairs principal region $\wt R$. The map $\wt R \to R$ is a universal covering map whose deck transformation group is the subgroup $\Stab(\wt R) \subgroup \pi_1 S$. The ``closed'' version $\overline R$ of the principal region $R$ may be identified either with the closure of $\wt R$ in $\wt S$ modulo the action of $\Stab(\wt R)$ or with the completion of the geodesic metric on $R$ itself; the map $\overline R \to S$ is locally injective, and the image of $\overline R - R$ is a union of finitely many leaves of $\Lambda_\mu$ called the \emph{boundary leaves} of $R$. The sets of principal regions and of boundary leaves are each finite and nonempty. 

We say that $\Lambda$ \emph{fills} $S$ if for each principal region $R$ of $\Lambda_\mu$, either $R$ is to homeomorphic an open disc, or $R$ is homeomorphic a half-open annulus such that that $\bdy R$ is a component of $\bdy S$. Each of these two cases has a more precise description, expressed in terms of a corresponding upstairs principal region $\wt R$ and its stabilizer $\Stab(\wt R) \subset \pi_1 S$, as follows. $R$~is homeomorphic to an open disc if and only if $\Stab(\wt R)$ is trivial, in which case $\wt R$ is the interior (relative to $\wt S$) of the convex hull of a finite subset of $\bdy\pi_1 S$ of cardinality $k \ge 3$, and the map $\wt R \to R$ is a homeomorphism; in this case we say that $R$ is a \emph{(open) ideal polygon}, or more precisely an ideal $k$-gon. $R$~is homeomorphic to a half-open annulus such that $\bdy R$ is a component of $\bdy S$ if and only if $\bdy\wt R$ is a component of $\bdy \wt S$ and the subgroup $\Stab(\wt R) = \Stab(\bdy\wt R)$ is infinite cyclic, in which case $\wt R$ is the interior (relative to $\wt S$) of the convex hull of the union of the 2-point subset $\bdy \Stab(\wt R) \subset \bdy\pi_1 S$ and a $\Stab(\wt R)$-invariant subset of $\bdy\pi_1 S - \bdy\Stab(\wt R)$ having a finite number $k \ge 1$ of $\Stab(\wt R)$-orbits; in this case we say that $R$ is an \emph{(open) crown} (see \cite{CassonBleiler} Figure~4.2), or more precisely a $k$-pointed crown. In all cases the map of the closed principal region $\overline R \mapsto S$ is injective. Note that if $\Lambda$ fills $S$ then $\Lambda$ is not a single periodic leaf, and $\Lambda$ is minimal, meaning that every leaf is dense. 

\smallskip\textbf{Pseudo-Anosov mapping classes.} The \emph{mapping class group} of $S$ is 
$$\MCG(S) = \Homeo(S) / \Homeo_0(S)
$$
where $\Homeo_0(S)$ is the normal subgroup of the homeomorphism group $\Homeo(S)$ consisting of those homeomorphisms that are isotopic to the identity. By the Dehn-Nielsen-Baer theorem \cite{FarbMargalit:primer}, the kernel of the natural map $\Homeo(S) \to \Out(\pi_1 S)$ is $\Homeo_0(S)$ and its image is the subgroup $\Out(\pi_1 S, \bdy S) \subgroup \Out(\pi_1 S)$ that preserves the set of conjugacy classes of cyclic subgroups associated to the components of $\bdy S$, and so we obtain a natural isomorphism $\MCG(S) \to \Out(\pi_1 S,\bdy S)$.

A mapping class $\phi \in \MCG(S)$ is \emph{pseudo-Anosov} if there exists a pair of geodesic laminations $\Lambda^s,\Lambda^u \subset \B(\pi_1 S)$ called the \emph{stable and unstable lamination}, such that their realizations $\Lambda^s_\mu,\Lambda^u_\mu \subset S$ are filling, and such that there is a number $\lambda>1$ and positive $\pi_1 S$-equivariant Borel measures on the lifts $\wt\Lambda^s,\wt\Lambda^u \subset \wt\B(\pi_1 S)$ which are pushed forward by any lift $\ti\phi$ to their multiples by factors of $\lambda^\inv,\lambda$ respectively. A leaf $\ell$ of $\Lambda^s$ or $\Lambda^u$ is \emph{fixed} if $\phi(\ell)=\ell$ and is \emph{periodic} if $\phi^k(\ell)=\ell$ for some $k \ne 0$. 

By Lemma~6.1 of \cite{CassonBleiler}, there exists a homeomorphism $\Phi \from S \to S$ representing $\phi$ that preserves both $\Lambda^s_\mu$ and $\Lambda^u_\mu$. If in addition $\Phi$ exists preserving each individual principal region of $\Lambda^s_\mu$ and $\Lambda^u_\mu$, its boundary leaves, and their orientations then we say that $\Phi$ is \emph{rotationless}; note that a rotationless power of $\Phi$ exists because there are only finitely many principal regions and boundary leaves. We say that $\phi$ is rotationless if it has a rotationless representative. If $\Phi$ is rotationless and if $\wt\Phi \from \wt S \to \wt S$ is a lift of $\Phi$, we say that $\wt\Phi$ is an \emph{$s$-principal lift} if there exists a principal region $\wt R$ of $\wt\Lambda^s_\mu$ such that $\wt\Phi$ preserves $\wt R$ and preserves some (every) boundary leaf of $\wt R$. A $u$-principal lift is similarly defined. We remark that these notions of ``rotationless'' and ``principal lift'' are in agreement with the corresponding notions for outer automorphisms of $\pi_1 S$ as given in \recognition\ (or see Definition~\ref{DefPrinicipalAndRotationless}), but we do not make use of this correspondence.

The following proposition is our summary of the Nielsen-Thurston theory. Given a homeomorphism of a circle $h \from C \to C$, we say that $h$ has \emph{alternating source-sink dynamics} if for some $k \ge 1$ the fixed point set $\Fix(h)$ consists of $k$ attractors (sinks) and $k$ repellers (sources) alternating around $C$, and each component of $C-\Fix(h)$ is preserved by $h$. Also, if $J \subset C$ is compact subarc invariant under $h$ and with fixed endpoints, we say that $h$ has \emph{alternating source sink dynamics relative to $J$} if the restricted fixed point set $\Fix(h \restrict C-J)$ consists of a countable set of alternating attractors and repellers, accumulating on each endpoint of $J$.

\begin{proposition}\label{PropNielsenThurstonTheory}
If $\phi \in \MCG(S)$ is a rotationless pseudo-Anosov mapping class, if $\Phi \from S \to S$ is a rotationless representative, and if $\wt\Phi \from \wt S \to \wt S$ is a lift of $\Phi$, then  $\wt\Phi$ is $s$-principal if and only if $\wt\Phi$ is $u$-principal, in which case we say that $\wt\Phi$ is principal. Furthermore:
\begin{enumerate}
\item\label{ItemNTRegionsCorrespond}
If $\wt\Phi$ is principal then $\wt\Phi$ preserves a unique principal region $\wt R^s$ of $\wt\Lambda^s$ and a unique principal region $\wt R^u$ of~$\wt \Lambda^u$. Furthemore, letting $R^s,R^u$ be the corresponding principal regions downstairs of $\Lambda^s,\Lambda^u$, the following hold:
\begin{enumerate}
\item $R^s$ is a $k$-pointed ideal polygon if and only if $R^u$ is a $k$-pointed polygon, in which case $\wt\Phi$ acts with alternating source-sink dynamics on the circle $\bdy\wt S \union \bdy\pi_1 S$, with $k$ attractors being the ideal points of $\wt R^u$, and $k$ repellers being the ideal points of $\wt R^s$.
\item $R^s$ is a $k$-pointed crown if and only if $R^u$ is a $k$-pointed crown, in which case $\bdy R^s \intersect \bdy R^u$ is the same component $c$ of $\bdy S$, $\bdy \wt R^s \intersect \bdy \wt R^u$ is the same component $\ti c$ of $\bdy \wt S$, and $h$ acts with alternating source sink dynamics on the circle $\bdy\wt S \union \bdy \pi_1 S$ relative to the compact subarc $\ti c \union \bdy \ti c$, with attractors being the ideal points of $\wt R^u$ and repellers being the ideal points of $\wt R^s$.
\end{enumerate}
\item The property of $\wt\Phi$, $\wt R^s$, and $\wt R^u$ defined by the first sentence of \pref{ItemNTRegionsCorrespond} gives a $\pi_1 S$-equivariant bijection between three sets: principle lifts of $\Phi$; principle regions of $\wt\Lambda^s$; and principle regions of $\wt\Lambda^u$. This descends to a bijection between principle regions of $\Lambda^s$ and of $\Lambda^u$.
\end{enumerate}
\end{proposition}

\paragraph{Remarks on the proof.} The existing texts on Nielsen-Thurston theory prove the analogue of Proposition~\ref{PropNielsenThurstonTheory} for a closed surface; see for example \cite{CassonBleiler} Theorem 5.5. Nielsen's original works \cite{Nielsen:InvestigationsI_II_IIIEnglish} analyze the bounded case as well, and this case can be related to Thurston's theory of geodesic laminations as outlined in \cite{Miller:Nielsen} Section~9, yielding the above proposition.

\begin{definition}[Proper geodesics and proper equivalence]
\label{DefProperGeodesic}
A \emph{proper geodesic} in $\wt S$ is the convex hull in $\wt S$ of a pair of points in $\bdy \wt S \union \pi_1 S$, excluding a pair of points in the closure of same component of $\bdy\wt S$. Every proper geodesic is one of the following: a \emph{proper geodesic line} in the interior of $\wt S$; a \emph{proper geodesic ray}, having one ideal endpoint and intersecting $\bdy\wt S$ transversely at its finite endpoint; or a \emph{proper geodesic arc} intersecting $\bdy \wt S$ transversely at two finite endpoints. The path downstairs in $S$ obtained by projection of a proper geodesic upstairs is also called a \emph{proper geodesic} in $S$ (proper geodesics downstairs in $S$ may be regarded as parameterized by arc length, with equivalence being defined by reparameterization). Two proper geodesics $\ti\ell,\ti\ell'$ in $\wt S$ are \emph{properly equivalent} if they have the same ideal endpoints and their finite endpoints lie in the same components of $\bdy\wt S$. Two proper geodesic paths downstairs in $S$ are properly equivalent if properly equivalent lifts to $\wt S$ can be chosen. We use the notation $[\ell]$ to denote the proper equivalence class of a proper geodesic $\ell$ in $S$.
\end{definition}

The mapping class group $\MCG(S)$ acts on the set of proper equivalence classes of proper geodesics. Given $\theta \in \MCG(S)$ and a proper geodesic $\ell$, the image $\theta[\ell]$ is (well) defined as follows: choosing a homeomorphism $\Theta \from S \to S$ representing $\theta$, and choosing a lift $\wt\Theta \from \wt S \to \wt S$ of $\Theta$ and a lift $\ti\ell$ of $\ell$, the proper equivalence class $\theta[\ell]$ is represented by the projection to $S$ of the unique proper geodesic in $\wt S$ having the same ideal endpoints and finite endpoints as the quasigeodesic $\wt\Theta(\ti\ell)$.

The following consequence of Nielsen/Thurston theory will be applied in the proof of Proposition~\ref{PropGeomLams}, and in Lemma~\refWA{geometricFullHeightCase} which is a piece of the weak attraction theory of \PartThree.

\begin{proposition}\label{PropNTWA} Let $\phi \in \MCG(S)$ be pseudo-Anosov with stable and unstable laminations $\Lambda^s,\Lambda^u \subset \B(\pi_1 S)$ and with realizations $\Lambda^s_\mu,\Lambda^u_\mu \subset \interior(S)$.  For each proper geodesic $\ell$ in $S$ the following are equivalent: 
\begin{enumerate}
\item\label{NTTrans}
$\ell$ crosses $\Lambda^s_\mu$ transversely. 
\item\label{NTNotLeafNorRegion}
$\ell$ is neither a leaf of $\Lambda^s_\mu$ nor is $\ell$ contained in a principal region of $\Lambda^s_\mu$.
\item\label{NTNotWA}
$\ell$ is weakly attracted to $\Lambda^u_\mu$ under iteration of $\phi$ in the following sense: for each $\epsilon>0$ and $M>0$, there exists $K$ such that for some (any) sequence of proper geodesics $\ell_i$ representing $\theta^i[\ell]$, and for any $i \ge K$, there exists a subpath of $\ell_i$ and a leaf segment of $\Lambda^u_\mu$, each of length $\ge M$, and each lifting to geodesics in $\wt S$ having Hausdorff distance $\le \epsilon$ from each other.
\end{enumerate}
Furthermore, every proper geodesic arc is weakly attracted to $\Lambda^u$ by iteration of $\phi$.
\end{proposition}

\begin{proof} Choose a homeomorphism $\Theta \from S \to S$ representing $\theta$ such that $\Theta(\Lambda^s_\mu)=\Lambda^s_\mu$ and $\Theta(\Lambda^u_\mu)=\Lambda^u_\mu$; such a choice exists by \cite{CassonBleiler} Lemma~6.1.

Note that \pref{NTNotLeafNorRegion} is an invariant of proper equivalence. Note also that \pref{NTNotLeafNorRegion} is preserved by the action of $\theta$: for any proper geodesic $\ell$, choosing a lift $\wt \Theta \from S \to S$ and a lift $\ti\ell \subset \wt S$, if $\ti\ell$ is a leaf of $\wt\Lambda^s_\mu$ then so is $\wt\Theta(\ti\ell)$ and it is the unique representative of its proper equivalence class; and if $\ti\ell$ is contained in a principle region $R$ of $\wt\Lambda^s_\mu$ then $\wt\Theta(\ti\ell)$ is contained in $\wt\Theta(R)$, the proper geodesic with the same endpoints as $\wt\Theta(\ti\ell)$ is also contained in $\wt\Theta(R)$, and any properly equivalent proper geodesic is also contained in $\wt\Theta(R)$.

Equivalence of \pref{NTTrans} and~\pref{NTNotLeafNorRegion} is obvious, as is equivalence of the two version of \pref{NTNotWA} one using the quantifier ``some'' and the other using ``any''. To prove \pref{NTNotWA}$\implies$\pref{NTNotLeafNorRegion}, if \pref{NTNotLeafNorRegion} fails then it follows by induction that for all $i$ the proper geodesic path $\ell_i$ is either a leaf of $\Lambda^s_\mu$ or is contained in a principle region of $\Lambda^s_\mu$, from which it follows that \pref{NTNotWA} fails.

It remains to prove \pref{NTNotLeafNorRegion}$\implies$\pref{NTNotWA}. Property~\pref{NTNotLeafNorRegion} is invariant under taking positive powers of $\theta$, and we may therefore pass to a power such that $\Theta$ preserves each principal region and each boundary leaf of $\Lambda^u_\mu$ and of $\Lambda^s_\mu$. Choose any principal region $R$ of $\Lambda^s_\mu$ and any boundary leaf $\gamma$ of~$R$, lift $R$ to a principal $\wt R \subset \wt S$, lift $\gamma$ to a boundary leaf $\ti\gamma$ of $\wt R$ and let $\wt\Theta \from \wt S \to \wt S$ be a lift of $\Theta$ that preserves $\wt R$ and $\ti\gamma$.
Since $\gamma$ is dense in $\Lambda^s_\mu$ it follows that $\ell$ crosses $\gamma$ transversely, and so there is a lift $\ti\ell$ of $\ell$ that crosses $\ti\gamma$ transversely. Orient $\ti\ell$ so that it points out of $\wt R$ where it crosses $\ti\gamma$. Let $x_-,x_+ \in \bdy\wt S \union \bdy\pi_1 S$ be the initial and terminal points of~$\ti\ell$. Define $X_+$ to be the following compact subset of $\bdy\wt S \union\bdy\pi_1 S$: $X_+ = x_+$ if $x_+ \in \bdy\pi_1 S$; otherwise $X_+$ is the closure in $\bdy\wt S \union \bdy\pi_1 S$ of the component of $\bdy\wt S$ containing~$x_+$. Define $X_-$ similarly using $x_-$. 
Let $\eta,\eta'$ be the endpoints of $\ti\gamma$, let $(\eta,\eta')$ be the open subinterval of the circle $\bdy\wt S \union \bdy\pi_1 S$ containing $x_+$, and note that since $\ti\ell$ crosses leaves of $\wt\Lambda^s_\mu$ close to $\ti\gamma$ and just outside of $\wt R$ it follows that $X_+ \subset (\eta,\eta')$. 

Applying Proposition~\ref{PropNielsenThurstonTheory} it follows that $(\eta,\eta')$ contains a unique attractor $\xi_+$ and that the action of $\wt\Theta$ on $(\eta,\eta')$ converges uniformly on compact sets to the point $\xi_+$, in particular the sequence of compact sets $\wt\Theta^i(X_+)$ converge uniformly to~$\xi_+$. Let $\ti m$ be a boundary leaf of $\Lambda^u_\mu$ with ideal endpoint~$\xi_+$. Let $\ti\ell_i$ be the proper geodesic with the same endpoints as $\wt\Theta(\ti\ell)$, one in $\wt\Theta^i(X_+)$ and the other in $\wt\Theta^i(X_-)$. Since $\wt\Theta^i(X_-) \subset (\bdy\wt S \union \bdy\pi_1 S)$ is contained in the complement of $(\eta,\eta')$ for all~$i$, and since the sequence of compact sets $\Theta^i(X_+)$ converges uniformly to $\xi_+ \in (\eta,\eta')$, it follows that for any $\epsilon,M>0$ if $i$ is sufficiently large then there are subpaths of $\ti\ell_i$ and of $\ti\mu$ of length $\ge M$ and at Hausdorff distance $<\epsilon$ from each other. It follows that $\ell$ is weakly attracted to $\Lambda^u_\mu$ by iteration of~$\theta$.
\end{proof}

\subsubsection{Comparing free group laminations and surface laminations.} 
\label{SectionFreeVersusSurfaceLams}
Consider a rotationless $\phi \in \Out(F_n)$ represented by a \ct\ $f \from G \to G$ with geometric \eg\ stratum $H_r$. Consider also a geometric model $X$ with weak geometric submodel $Y$ of $f \from G \to G$ relative to $H_r$, with all associated notations from Definitions~\ref{DefWeakGeomModel} and~\ref{DefGeomModel}, in particular the surface $S$ and pseudo-Anosov mapping class $\psi \in \MCG(S)$. 

By $\pi_1$-injectivity of the map $S \xrightarrow{j} Y \inject X$ combined with the fact that finite rank subgroups of $F_n$ are quasiconvex, we obtain a $\pi_1 S$-equivariant embedding $\bdy(\pi_1 S) \inject \bdy F_n$. As explained in Section~\ref{SectionSubgroupLinesAndEnds}, that embedding induces a $\pi_1 S$-equivariant embedding $\wt \B(\pi_1 S) \subset \wt \B$ which induces in turn a continuous function $\B(\pi_1 S) \to \B$. 

\begin{proposition}\label{PropGeomLams} 
With the notation above, letting $\Lambda^+ \in \L(\phi)$ be the lamination associated to $H_r$ and $\Lambda^- \in \L(\phi^\inv)$ its dual lamination, and letting $\Lambda^s,\Lambda^u \subset \B(S)$ be the stable and unstable laminations of the pseudo-Anosov mapping class $\psi \in \MCG(S)$ associated to the geometric model, we have:
\begin{enumerate}
\item \label{ItemStableIsMinus}
The map $\B(\pi_1 S) \to \B$ takes $\Lambda^u$, $\Lambda^s$ homeomorphically to $\Lambda^+$, $\Lambda^-$ respectively.
\item \label{ItemLambdasMinimal}
Every leaf of $\Lambda^+$ is dense in $\Lambda^+$, and similarly for $\Lambda^-$.
\item \label{ItemAllGeneric}
All leaves of $\Lambda^+$ and $\Lambda^-$ are generic. 
\item \label{ItemLamSurfSupport}
$\F_\supp(\Lambda^+) = \F_\supp(\Lambda^-) = \F_\supp[\pi_1 S]$.
\end{enumerate}
\end{proposition}

\begin{proof} Note immediately that \pref{ItemStableIsMinus}$\implies$\pref{ItemLambdasMinimal}, and that \pref{ItemLambdasMinimal}$\implies$\pref{ItemAllGeneric} using that each leaf of a minimal lamination is birecurrent. It remains to prove \pref{ItemStableIsMinus} and~\pref{ItemLamSurfSupport}.

\smallskip

We first prove \pref{ItemStableIsMinus} and~\pref{ItemLamSurfSupport} in the special case that $X=Y=S$ and so $F_n = \pi_1 S$. To prove~\pref{ItemStableIsMinus}, up to replacing $\phi$ by $\phi^\inv$ it suffices to prove that $\Lambda^u = \Lambda^+$ in $\B=\B(\pi_1 S)$, and since $\L(\phi) = \{\Lambda^+\}$ it suffices to prove that $\Lambda^u \in \L(\phi)$, which we do using the Nielsen-Thurston theory. Fix a hyperbolic structure $\mu$ on~$S$. We may pass to a power of $\phi$ which is rotationless in $\MCG(S)$. Choose a rotationless homeomorphism $\Phi \from S \to S$ preserving $\Lambda^s_\mu,\Lambda^u_\mu$. Choose a principal lift $\wt\Phi \from \wt S \to \wt S$ preserving a principal region $\wt R^u$ of $\wt\Lambda^\mu$. Let $\ti\ell_\mu$ be a boundary leaf of $\wt R_u$, let $\ti\ell \in \wt\B(\pi_1 S)$ be the corresponding abstract line, and let $\ell \in \B(\pi_1 S)$ be its image downstairs. Applying Proposition~\ref{PropNielsenThurstonTheory}, the endpoints $\xi,\eta$ of $\ti\ell_\mu$ are attractors for the action of $\wt\Phi$ on $\bdy\pi_1 S$. Choosing attracting neighborhoods $U_\xi,U_\eta$ of $\xi,\eta$, respectively, one obtains an attracting neighborhood of $\ell$ under the action of $\phi$ on $\B(\pi_1 S)$, namely, the image in $\B(\pi_1 S)$ of all lines in $\wt\B(\pi_1 S)$ having one endpoint in $U_\xi$ and the other endpoint in $\U_\eta$. Since $\Lambda^u$ is minimal, the line $\ell$ is birecurrent and dense in $\Lambda^u$. Since $\Lambda^u$ is filling, the corresponding geodesic $\ell_\mu$ in $S$ is not closed, and so $\ell$ is not supported by a rank~$1$ free factor of $\pi_1 S$. It follows that $\Lambda^u$ satisfies the definition of an attracting lamination of $\phi$ with generic leaf $\ell$ (Definition~\ref{DefAttractingLaminations}), proving~\pref{ItemStableIsMinus} in the special case.

Next, continuing in the special case where $F_n=\pi_1 S$, we show that $\F_\supp(\Lambda^u) = \{[\pi_1 S]\}$; applying this to $\psi^\inv$ it follows that $\F_\supp(\Lambda^s) = \{[\pi_1 S]\}$. The free factor system $\F_\supp(\Lambda^u) = \F_\supp(\Lambda^+)$ has one component. Applying Lemma~\ref{LemmaScaffoldFFS}~\pref{ItemBdyAndSurfSupps} which says that the free factor support of the boundary components of $S$ is the whole of $\{[\pi_1 S]\}$, to prove~\pref{ItemLamSurfSupport} it suffices to prove that for any connected free factor system $\{[A]\}$ of $\pi_1 S$ that supports $\Lambda^u$, and for any component $b$ of $\bdy S$, $[A]$ supports $b$. There is a $k$-pointed principal region $R$ for $\Lambda^u_\mu$ containing $b$, for some $k \ge 1$. Pick a lift $\wt R \subset \wt S$ so that $\ti b = \bdy\wt R$ is a lift of~$b$. Consider the infinite cyclic group $Z_b = \Stab(\ti b) = \Stab(\wt R)$. From the description of crown principal regions and their corresponding covers it follows that the boundary of the closure of $\wt R$ in $\wt S$ is the union of $\ti b$ and a bi-infinite sequence of oriented boundary leaves $\ell_i$ of $\wt\Lambda^u_\mu$, $i \in \Z$, so that the terminal endpoint of $\ell_i$ equals the initial endpoint of $\ell_{i+1}$ in $\Lim_\mu$. Choose a generator $z$ for $Z_b$ such that $z(\ell_i) = \ell_{i+k}$. Choose a free factor $A$ representing the conjugacy class $[A]$ so that $\bdy\ti\ell_0 \subset \bdy A$. By combining malnormality of $A$ with Fact~\ref{FactBoundaries} and an induction argument, it follows that $\bdy\ti\ell_i \subset \bdy A$ for all $i \in \Z$. The collection of points $\union_{i \in \Z} \bdy\ell_i$ is invariant under $Z_b$ they accumulate on the two points of $\bdy\ti b$, and it follows that $\bdy\ti b \subset \bdy A$, implying that $b$ is carried by $[A]$ as desired.

\smallskip

We now turn to the general case of \pref{ItemStableIsMinus} and~\pref{ItemLamSurfSupport}. We use the notation established in setting up Proposition~\ref{PropGeomLams}. By appropriate choices of base points and paths amongst them, the composition $S \xrightarrow{j} Y \inject X$ induces an injection 
$$\pi_1 S \inject \pi_1(Y) \approx \pi_1(G_r) \subset \pi_1(G) \approx F_n \qquad (*)
$$
We identify $\pi_1 S$ with its image in $F_n$ under the injection~$(*)$, and we let $\beta\from \B(\pi_1 S) \to \B(F_n)$ be the induced continuous map. From the dynamic conditions in \GeomModelsDef\ it follows that we may choose a homeomorphism $\Psi \from S \to S$ representing the mapping class $\psi \in \MCG(S)$, and we may choose an automorphism $\Phi \in \Aut(F_n)$ representing $\phi$ which leaves $\pi_1 S$ invariant, so that the automorphism $\Phi \restrict \pi_1 S$ is a representative of the pseudo-Anosov mapping class $\psi$ under the injection $\MCG(S) \approx \Out(\pi_1 S, \bdy S) \inject \Out(\pi_1 S)$. Applying the special case combined with Lemma~\ref{LemmaPushLam}, it follows that $\beta(\Lambda^u)$, $\beta(\Lambda^s)$ are a dual lamination pair in $\L^\pm(\phi)$ whose free factor support in $F_n$ equals $\F_\supp[\pi_1 S]$. By Lemma~\ref{LemmaScaffoldFFS}~\pref{ItemTopIsNotLower} we have $\F_\supp[\pi_1 S] \not\sqsubset [G_{r-1}]$, but from $(*)$ we have $\F_\supp[\pi_1 S] \sqsubset [G_r]$. It follows that $\F_\supp(\beta(\Lambda^u)) \not\subset [G_{r-1}]$ and $\F_\supp(\beta(\Lambda^u)) \sqsubset [G_r]$. But the only element of $\L(\phi)$ with this property is $\Lambda^+$, and so $\beta(\Lambda^u) = \Lambda^+$. By duality it follows that $\beta(\Lambda^s) = \Lambda^-$. This proves~\pref{ItemLamSurfSupport} and also proves most of~\pref{ItemStableIsMinus}. 

What is left is to prove that the maps  $\Lambda^u \to \Lambda^+$ and $\Lambda^s \to \Lambda^-$ obtained by restriction $\beta \from \B[\pi_1 S] \to \B$ are homeomorphisms. We prove this for $\Lambda^u$, the same following for $\Lambda^s$ by replacing $\phi$ with~$\phi^\inv$. The idea of the proof is that although $\pi_1 S \subgroup F_n$ need not be a free factor, nor even malnormal, by Bass-Serre theory it just misses being malnormal, in that the intersections of $\pi_1 S$ with any of its conjugates by elements of $F_n - \pi_1 S$ are peripheral in $\pi_1 S$. Since no leaf of $\Lambda^u$ is peripheral, the map $\beta$ restricts to a bijection  between $\Lambda^u$ and $\Lambda^+$. In the Hausdorff setting this would be enough to verify that this bijection is a homeomorphism, but we must work a little harder in the non-Hausdorff setting.

We may assume henceforth that $G=G_r$ and that $X=Y$ and so we may identify $\pi_1(G_r) \approx \pi_1(Y) \approx F_n$; this assumption is justified by restricting $f \from G \to G$ to the component $G^0_r$ of $G_r$ that contains $H_r$, and by restricting the geometric model $Y$ for $H_r$ to the component $Y^0_r$ of $Y$ that contains $H_r$.  

It follows by Lemma~\ref{LemmaLImmersed}~\pref{ItemSeparationOfSAndL} that if $\gamma \in F_n-\pi_1 S$ then $\pi_1 S \intersect \pi_1 S^\gamma$ is either the trivial subgroup or a peripheral subgroup in $\pi_1 S$ meaning an infinite cyclic subgroup conjugate into the fundamental group of some component of $\bdy S$. Applying this together with Fact~\ref{FactBoundaries} it follows that for any $\gamma \in F_n - \pi_1 S$ the intersection $\bdy\pi_1 S \intersect \gamma \cdot \bdy\pi_1 S$ is either empty or is the pair of points fixed by the stabilizer of some component of $\bdy\wt S$. This implies in turn that $\wt\B(\pi_1 S) \intersect \gamma \cdot \wt\B(\pi_1 S)$ is either empty or a single line corresponding to a component of $\bdy\wt S$. Let $\Gamma$ be a set of left coset representatives of $\pi_1 S$ in $F_n$, and let $\gamma_0 \in \pi_1 S \intersect \Gamma$ be the representative of the coset $\pi_1 S$. Applying Fact~\ref{FactLinesClosed} it follows that the subset $\union \{\gamma \cdot \wt\B(\pi_1 S) \suchthat \gamma \in \Gamma - \gamma_0\}$ is closed in $\wt\B$. Also, the finite subset $\B^\bdy(\pi_1 S) \subset \B(\pi_1 S)$ represented by components of $\bdy S$ is closed in $\B(\pi_1 S)$, and so its total lift $\wt\B^\bdy(\pi_1 S)$ consisting of all lines corresponding to components of $\bdy\wt S$ is closed in $\wt\B(\pi_1 S)$ and so is also closed in~$\wt\B$. The set
$$\B - \biggl(\bigl( \bigcup_{\gamma \in \Gamma-\gamma_0} \gamma  \cdot \wt\B(\pi_1 S) \bigr) \union \wt\B^\bdy(\pi_1 S)  \biggr)
$$
is therefore open in $\B$. The intersection of this set with $\wt\B[\pi_1 S] = \union\{\gamma \cdot \wt\B(\pi_1 S) \suchthat \gamma \in \Gamma\}$ is equal to $\wt\B^\int(\pi_1 S) = \wt\B(\pi_1 S) - \wt\B^\bdy(\pi_1 S)$ and so the latter is open in $\wt\B[\pi_1 S]$. Furthermore, $\wt\B^\int(\pi_1 S)$ is disjoint from each of its translates $\gamma \cdot \wt\B^\int(\pi_1 S)$ for $\gamma \in \Gamma - \gamma_0$. The collection of sets $\{\gamma \cdot \wt\B^\int(\pi_1 S) \suchthat \gamma \in \Gamma\}$ is therefore pairwise disjoint, and each set in this collection is an open subset of the union of the collection.

Consider now the lamination $\Lambda^u \subset \B(\pi_1 S)$ and its image $\Lambda^+$ under the continuous map $\beta \from \B(\pi_1 S) \to \B$. By definition of unstable laminations we have $\Lambda^u \subset \B^\int(\pi_1 S)$. Let $\wt\Lambda^u \subset \wt\B(\pi_1 S)$ be its total lift with respect to $\pi_1 S$ and so $\wt\Lambda^u \subset \wt\B^\int(\pi_1 S)$. The total lift of $\wt\Lambda^u$ with respect to $\pi_1 S$ is equal to the set $\union_{\gamma \in \Gamma} \left(\gamma \cdot \wt\Lambda^u\right)$, and from what is proved in the previous paragraph this is a disjoint union whose terms are each open in the union. It follows that the inclusion of $\wt\Lambda^u$ into the union $\union_{\gamma \in \Gamma} \left(\gamma \cdot \wt\Lambda^u\right)$ descends to a homeomorphism from $\Lambda^u = \wt\Lambda^u / \pi_1 S$ to $\Lambda^+ = \left( \union_{\gamma \in \Gamma} \gamma \cdot \wt\Lambda^u\right) / F_n$, but this homeomorphism is precisely the restriction to $\Lambda^u$ of the map $\beta \from \B(\pi_1 S) \to \B$.

This completes the proof of Proposition~\ref{PropGeomLams}.
\end{proof}

\subsubsection{Application: The inclusion lattice of attracting laminations.} 
The inclusion partial order on subsets of $\B$ restricts to an inclusion partial order on~$\L(\phi)$. An attracting lamination is said to be \emph{topmost} if it is maximal with respect to inclusion on $\L(\phi)$. It was proved in \BookOne\ Corollary~6.0.1 that a lamination in $\L(\phi)$ is topmost if and only if its dual lamination in $\L(\phi^\inv)$ is topmost. We extend this result to show that the duality relation respects inclusion; the proof uses Proposition~\ref{PropGeomLams}.

\begin{lemma}\label{containmentSymmetry} 
If $\Lambda^\pm_i $ and $\Lambda^\pm_j$ are dual lamination pairs for $\phi \in \Out(F_n)$ then $ \Lambda^+_i \subset \Lambda_j^+$ if and only $\Lambda^-_i \subset \Lambda^-_j$. 
\end{lemma}

\begin{proof} Passing to a power we may assume $\phi$ is rotationless. Let $F^i$, $F^j$ denote free factors with associated one component free factor systems: 
\begin{align*}
\{[F^i]\} &= \F_\supp(\Lambda^+_i) = \F_\supp(\Lambda^-_i) \\
\{[F^j]\} &= \F_\supp(\Lambda^+_j) = \F_\supp(\Lambda^-_j)
\end{align*} 
If $[F^j]$ does not carry $[F^i]$ then $\Lambda^+_j \not \supset \Lambda_i^+$ and $\Lambda^-_j \not \supset \Lambda_i^-$. We may therefore assume that $[F^j]$ carries $[F^i]$. Restricting $\phi$ to $F^j$ we may assume that $F^j = F_n$, which implies that $\Lambda^+_j$ is topmost (in $\L(\phi)$) and $\Lambda^-_j$ is topmost (in $\L(\phi^\inv)$). By Corollary~6.0.11 of \BookOne\ $\Lambda^+_i$ is topmost if and only if $\Lambda^-_i$ is topmost. In this topmost case, $\Lambda^+_j \not \supset \Lambda_i^+$ and $\Lambda^-_j \not \supset \Lambda_i^-$. We may therefore assume that neither $\Lambda^+_i$ nor $\Lambda^-_i$ is topmost, in which case it suffices to show that $\Lambda^+_i \subset \Lambda_j^+$ and that $\Lambda^-_i \subset \Lambda_j^-$; the two cases are similar, so it suffices to assume that $\Lambda^+_j \not \supset \Lambda_i^+$ and argue to a contradiction. 

Since $\Lambda^+_i$ is not topmost and is not contained in~$\Lambda_j^+$, there is a topmost lamination $\Lambda^+_k \in \L(\phi)$ such that (A)~$\Lambda^+_i \subset \Lambda_k^+$ and (B)~$\Lambda_k^+\not \subset \Lambda_j^+$. Choose a \ct\ $f \from G \to G$ representing $\phi$ and let $H_r$ and $H_s$ be the \eg\ strata corresponding to $\Lambda^+_i$ and $\Lambda^+_k$ respectively.   By (A),  $\Lambda^+_k$  has nongeneric lines. By Proposition~\ref{PropGeomLams}~\pref{ItemAllGeneric}, the stratum of $G$ corresponding to the lamination $\Lambda_k^+$ is non-geometric. By (B), $\Lambda_j^+$ is not weakly attracted to $\Lambda_k^+$, that is, no generic leaf of $\Lambda_j^+$ is weakly attracted to a generic leaf of $\Lambda_k^+$. Applying the Weak Attraction Theorem 6.0.1 and Remark 6.0.2 of \BookOne\ to the topmost lamination $\Lambda_k^+$ it follows that $\Lambda_j^+$ is carried by a proper free factor of $F_n$, which contradicts our assumption that $F^j = F_n$.
\end{proof}

\subsubsection{Application: The span argument for a geometric stratum}
\label{SectionSpanArgument}

We state and prove Fact~\ref{FactSpanArgument} based on the ``span argument'' from \BookOne\ Section~7. This fact is a generalization of \BookOne\ Corollary 7.0.8, which is restricted to the setting of geometric strata whose associated lamination is topmost. The proof is really no different, we simply notice that it applies in a broader setting. 

The fact we need says that the duality relation amongst geometric laminations, which ordinarily is determined by free factor systems, is also related to a broader class of subgroup systems, namely conjugacy classes of malnormal subgroups. The span argument is applied to ping-pong in \BookOne, and we shall need Fact~\ref{FactSpanArgument} for similar purposes in \PartFour.

\begin{fact} 
\label{FactSpanArgument} 
Consider a rotationless $\phi \in \Out(F_n)$ represented by a \ct\ $f \from G \to G$, and an \eg\ geometric stratum $H_r \subset G$ with corresponding lamination $\Lambda^+ \in \L(\phi)$ and dual lamination $\Lambda^- \in \L(\phi^\inv)$. Let $[\pi_1 S]$ be the surface subgroup system in $F_n$ associated to a geometric model for $f$ with respect to $H_r$ (as in Definition~\ref{DefGeometricStratum}), and let $\B[\pi_1 S] \subset \B$ be the subset of lines carried by $[\pi_1 S]$. For any malnormal finite rank subgroup $A \subgroup F_n$, if its conjugacy class $[A]$ carries $\Lambda^+$ then $[A]$ also carries $\B[\pi_1 S]$, and in particular $[A]$ carries $\Lambda^-$. 
\end{fact}

\begin{proof} Note that $\B[\pi_1 S] \subset \B$ is the image of the map $\B(\pi_1 S) \to S$ induced by the $\pi_1$-injective map $S \xrightarrow{j} Y \subset X$ as described in the paragraph preceding the statement of Proposition~\ref{PropGeomLams}. By applying Proposition~\ref{PropGeomLams}~\pref{ItemStableIsMinus} it follows that $\Lambda^- \subset \B[\pi_1 S]$. Once we have proved that $\B[\pi_1 S]$ is carried by $[A]$ it therefore follows that $\Lambda^-$ is also carried by $[A]$.

Since periodic lines in $\B[\pi_1 S]$ are dense in $\B[\pi_1 S]$, and since the set of lines carried by $[A]$ is closed (Fact~\ref{FactLinesClosed}~\pref{ItemLinesClosedOneGp}), it suffices to prove that each conjugacy class of $F_n$ that is carried by $[\pi_1 S]$ is also carried by $[A]$. The analogous fact in \BookOne\ is Corollary 7.0.8, the proof of which has two steps that we outline here: Lemma~7.0.7, the ``span argument'' which we may apply as stated; and Lemma~7.0.6, the ``lifting argument'', which we shall tailor to our present situation. 

Let $G_r \subset G$ be the filtration element associated to the stratum $H_r$. Consider a conjugacy class $c$ of $F_n$ which we assume to be carried by $[\pi_1 S]$, and so the circuit $\alpha$ in $G$ realizing $c$ is contained in the filtration element $G_r$. The hypothesis of Lemma~7.0.7 is that $\alpha$ is an \emph{$H_r$-geometric circuit in $G_r$}, which---as defined just before Lemma~7.0.7---means in our present terminology that $c$ is carried by~$[\pi_1 S]$, exactly matching our present assumption. The conclusion of Lemma~7.0.7 is that $\alpha$ is in the \emph{span} of a generic leaf $\lambda$ of $\Lambda^+$ which---as defined just before Lemma~7.0.6---means: for each $\ell>0$ there is a double sequence of finite subpaths of $\lambda$ of the form $\nu_1, \mu_1, \nu_2, \mu_2, \ldots, \nu_m, \mu_m$, each of length $\ge \ell$, such that for each $i \in \Z / m\Z$ the path $\nu_i$ is a subpath of $\mu_{i-1}$ and a subpath of $\mu_i$ or $\bar\mu_i$, and such that if we take a point $p_i$ in $\nu_i$, and if we let $[p_i,p_{i+1}]$ denote the corresponding subpath of $\mu_i$ (or $\bar\mu_i$) between the copy of $p_i$ in the $\nu_i$ subpath of $\mu_i$ and the copy of $p_{i+1}$ in the $\nu_{i+1}$ subpath of $\mu_i$, then $\alpha$ is freely homotopic in $G$ to the closed concatenation $[p_m=p_0,p_1]*\cdots*[p_{m-1},p_m]$. 

Consider now the malnormal subgroup $A \subgroup F_n$. Let $K \mapsto G$ be a Stallings graph for $A$, meaning an immersion of a finite core graph such that the image of the induced monomorphism $\pi_1 K \inject \pi_1 G$ is conjugate to $A$; the graph $K$ is just the core of the covering space of $G$ corresponding to $A$. Since $A$ is malnormal it follows that there exists $\ell>0$ such that any path in $G$ of length $\ge \ell$ has at most one lift to $K$: if no such $\ell$ existed then by taking a weak limit of longer and longer paths each with two distinct lifts to $K$ we would obtain a line in $G$ with two distinct lifts to $K$, which by Fact~\ref{FactBoundaries} contradicts malnormality of $A$. Using this $\ell$ and applying the fact that $\alpha$ is in the span of any leaf of $\Lambda^+$ (any leaf is generic, by Proposition~\ref{PropGeomLams}), let $\mu_i, \nu_i, p_i$ be as described above. Since $\mu_i$ is a subpath of a leaf of $\Lambda^+$, by hypothesis $\mu_i$ lifts to a path $\ti\mu_i$ in $K$. Since $\nu_i$ has length $\ge \ell$, the lifts of $\nu_i$ to $\ti\mu_{i-1}$ and $\ti\mu_i$ are identical; denote this lift~$\ti\nu_i$. Let $\ti p_i$ be the point in the path $\ti\nu_i$ corresponding to $p_i$ under this lift. We obtain a subpath $[\ti p_{i},\ti p_{i+1}]$ of $\ti\mu_i$ which is a lift of $[p_{i},p_{i+1}]$. The closed concatenation $[\ti p_0, \ti p_1]*\cdots*[\ti p_{m-1}, \ti p_m]$ is therefore a lift of $[p_0,p_1]*\cdots*[p_{m-1},p_m]$, and so it is freely homotopic to a circuit $\ti\alpha$ in $K$ whose projection to $G$ is freely homotopic to $\alpha$, proving that $\alpha$ is carried by~$[A]$ (in fact $\ti\alpha$ is a lift of $\alpha$, since $K \mapsto G$ is an immersion).

%
%
\end{proof}

\textbf{Remarks.} Corollary 7.0.7, from which we have adapted the lifting argument, is stated only in the narrower situation of an immersed graph $K$ representing the set of lines that are not weakly attracted to a certain topmost attracting lamination corresponding to some geometric \eg\ stratum of some relative train track map, a set of lines that is denoted $\<Z,\rho\>$ in Section~6 of \BookOne. What we have added to Corollary 7.0.7 is the observation that the lifting argument uses only that the subgroup system defined by the immersion of $K$ is malnormal. We will not prove this malnormality formally until Proposition~\refWA{PropVerySmallTree}, but it be can verified directly by combining Lemma~\ref{LemmaLImmersed}~\pref{ItemKInL} with the fact that the set of lines $\<Z,\rho\>$ is equal to the set of lines that lift to a certain subgraph $K$ of the complementary subgraph of an appropriate geometric model.

The proof of Corollary 7.0.8 of \BookOne\ contains an extra invariance argument, which invokes the bounded cancellation lemma to show the span of a line is well-defined independent of the marked graph in which that line is realized. Our proof avoids that step by applying Fact~\ref{FactBoundaries}.

\subsection{Geometricity is an invariant of a dual lamination pair}
\label{SectionLaminationGeometricity}

Given $\phi \in \Out(F_n)$, the following proposition shows that geometricity is an invariant of a lamination $\Lambda \in \L(\phi)$, not just an invariant of a particular \ct\ representing $\phi$ and \eg\ stratum corresponding to $\Lambda$. This is accomplished by exhibiting a property of $\Lambda$ that is evidently independent of the choice of a \ct, namely condition~\pref{ItemBoundaryExists}. Furthermore this property shows that geometricity is an invariant of dual lamination pairs $\Lambda^\pm \in \L^\pm(\phi)$, not just of the two individual laminations. 

After the statement of the proposition, we apply this invariance to extend the terminology ``geometric'' beyond the setting of \eg\ strata to apply to appropriate elements of $\L(\phi)$ and of $\L^\pm(\phi)$.

 
\begin{proposition}\label{PropGeomEquiv}
For any $\phi \in \Out(F_n)$, any $\Lambdapmp \in \L^\pm(\phi)$, and any \ct\ $f \from G \to G$ representing a rotationless iterate of $\phi$, if $H_r \subset G$ is the \eg\ stratum associated to $\Lambdapp$ then the following are equivalent:
\begin{enumerate}
\item \label{ItemStratumIsGeom}
$H_r$ is a geometric stratum.
\item \label{ItemBoundaryExists}
There exists a finite $\phi$-invariant set $\C = \{[c_0],[c_1],\ldots,[c_m]\}$ of  conjugacy classes  such that $\F_\supp(\C) = \F_\supp(\Lambda^\pm_\phi)$. 
\end{enumerate}
Furthermore, if \pref{ItemStratumIsGeom}, \pref{ItemBoundaryExists} hold then for any \ct\ $f' \from G' \to G'$ representing a positive or negative rotationless power of~$\phi$ the following holds:
\begin{enumeratecontinue}
\item \label{ItemGeomInv}
If $H'_s \subset G'$ is the \eg\ stratum associated to $\Lambda^+_\phi$ or $\Lambda^-_\phi$ then $H'_s$ is geometric. 
\item \label{ItemSameNielsen}
If $[G_r] = [G'_s]$ and $[G_{r-1}] = [G'_{s-1}]$, and if $\rho_r$, $\rho'_s$ are the closed indivisible Nielsen paths in $G$, $G'$ of height $r,s$, respectively, then the conjugacy classes $[\rho_r]$, $[\rho'_s]$ are the same up to inverse.
\item \label{ItemBoundarySupportCarriesLambda} Adopting the notation of the geometric model for $H_r$, we have 
$$\F_\supp[\bdy S] = \F_\supp[\pi_1 S] = \F_\supp(\Lambda^\pm_\phi)
$$
and we may choose $\C$ to be the peripheral conjugacy classes, that is,
$$\C = \{[\bdy_0 S],\ldots,[\bdy_m S]\}
$$
\end{enumeratecontinue}
\end{proposition}

Before proving the proposition, we apply it to formulate some important definitions:

\begin{definition}
\label{DefGeometricLamination}
Given $\phi \in \Out(F_n)$ and $\Lambda^\pm_\phi \in \L^\pm(\phi)$, we say that $\Lambda^+$ is a \emph{geometric lamination}, or that $\Lambda^\pm$ is a \emph{geometric lamination pair}, if Proposition~\ref{PropGeomEquiv}~\pref{ItemBoundaryExists} holds.  
\end{definition}

Proposition \ref{PropGeomEquiv} can therefore be reworded to say that geometricity of the lamination $\Lambda^+_\phi$ and of the pair $\Lambda^\pm_\phi$, are equivalent to geometricity of some (any) \eg\ stratum corresponding to $\Lambda^+_\phi$ or $\Lambda^-_\phi$, for some (any) \ct\ representing some (any) rotationless power of $\phi$ or $\phi^\inv$.

\begin{proof}[Proof of Proposition \ref{PropGeomEquiv}] Item \pref{ItemGeomInv} follows from the equivalence of  \pref{ItemStratumIsGeom} and \pref{ItemBoundaryExists}, the fact that $\F_\supp(\Lambda^+_\phi) = \F_\supp(\Lambda^-_\phi)$, and the fact that \pref{ItemBoundaryExists} is independent of the choice of positive or negative power of~$\phi$ and choice of \ct\ $f \from G \to G$ representing that power. To prove~\pref{ItemSameNielsen}, by Fact~\ref{FactEGNielsenCrossings} the conjugacy classes $[\rho_r]$, $[\rho'_s]$ have the same characterization and are therefore equal (up to inverse): it is the unique (up to inverse) root-free conjugacy class that is carried by $[G_r]=[G'_s]$ but not by $[G_{r-1}]=[G'_{s-1}]$ and is fixed by~$\phi$. 

To prove \pref{ItemBoundaryExists}$\implies$\pref{ItemStratumIsGeom}, suppose $H_r$ is not geometric. By Fact~\ref{FactEGNielsenCrossings}, no conjugacy class of height~$r$ is fixed by $\phi$, and so every fixed conjugacy class of height $\le r$ is supported by $[G_{r-1}]$; however, $\Lambda^+_\phi$ itself is not supported by $[G_{r-1}]$.

Finally, note that \pref{ItemStratumIsGeom}$\implies$\pref{ItemBoundarySupportCarriesLambda} follows from Lemma~\ref{LemmaScaffoldFFS} and Proposition~\ref{PropGeomLams}~\pref{ItemLamSurfSupport}, and \pref{ItemBoundarySupportCarriesLambda}$\implies$\pref{ItemBoundaryExists} using~$\C = [\bdy S]$.
\end{proof}

\subsection{Stabilizing its complement also stabilizes the surface}
\label{SectionPreservingSurface}

We continue with the fixed notation of Section~\ref{SectionGeomModelComplement}: a rotationless $\phi \in \Out(F_n)$, a representative \ct\ $f \from G \to G$, a geometric stratum~$H_r$, and a geometric model $X$ for $H_r$, and all the notations of Definition~\ref{DefGeomModel} associated to~$X$. We also adopt the notation of the complementary subgraph $L \subset X$ from Definition~\ref{DefComplSubgraph}, and the notation of the peripheral graph of spaces $\wh X$ from Definition~\ref{DefPeriphGraphOfSpaces}.

\smallskip

Consider a closed oriented surface $S$, and a curve system $C \subset S$ separating $S$ into two components whose closures are $V,W$, neither of which is a disc. Any homeomorphism of $S$ that stabilizes $V$ up to isotopy must also stabilize $W$ up to isotopy, and vice versa. It follows that subgroups $\Stab[V],\Stab[W] \subgroup \MCG(S) = \Out(\pi_1 S)$ that stabilize the isotopy classes of $V,W$, respectively, are equal. This kind of thing fails in $\Out(F_n)$ when one tries to use a free factorization $F_n = A * B$ in place of the decomposition $S = V \union W$: an outer automorphism that preserves the conjugacy class $[A]$ need not preserve the conjugacy class $B$, and vice versa, so neither of the subgroups $\Stab[A],\Stab[B]$ is contained in the other.
 
There is however an $\Out(F_n)$ context involving geometric models in which we do obtain complementary behavior of this sort. Recalling from Section~\ref{SectionSSAndFFS} the stabilizer of a subgroup system, the following result says that the subgroup $\Stab[\pi_1 L] \subgroup \Out(F_n)$ is contained in the subgroup $\Stab[\pi_1 S] \subgroup \Out(F_n)$, in fact it is contained in the $\MCG(S)$ subgroup of $\Stab[\pi_1 S]$, and that we obtain a well-defined induced homomorphism $\Stab[\pi_1 L] \mapsto \MCG(S)$. We will need this statement in greater generality, replacing $L$ by certain subgraphs~$K \subset L$.

For the statement of the theorem, recall that by the Dehn-Nielsen-Baer Theorem \cite{FarbMargalit:primer}, the subgroup of $\Out(\pi_1 S)$ consisting of all $\theta \in \Out(\pi_1 S)$ that preserve the set of conjugacy classes associated to all oriented components of $\bdy S$ is naturally isomorphic to $\MCG(S)$, and we identify $\MCG(S)$ with this subgroup.

\begin{proposition} 
\label{PropVertToFree} 
For any subcomplex $K \subset L$ such that $j(\bdy S) \subset K$, and for any $\theta \in \Out(F_n)$ such that $\theta[\pi_1 K]=[\pi_1 K]$, we have:
\begin{enumerate}
\item\label{ItemSubgroupFixed} 
$\theta$ fixes $[\pi_1 S]$.  
\item\label{ItemPeripheralPermuted}
The restriction $\theta \restrict \pi_1 S \in \Out(\pi_1 S)$ is well-defined, it lives in $\MCG(S)$, and the induced function $\Stab[\pi_1 K] \mapsto \MCG(S)$ defined by $\theta \mapsto \theta \restrict \pi_1 S$ is a homomorphism.
\end{enumerate}
\end{proposition}

The proof of Proposition~\ref{PropVertToFree}, most of which is in the following lemma, takes up the rest of the section. Once item~\pref{ItemSubgroupFixed} is proved, the part of item~\pref{ItemPeripheralPermuted} saying that $\theta \restrict \pi_1 S \in \Out(\pi_1 S)$ is well-defined follows from Lemma~\ref{LemmaLImmersed}~\pref{ItemSeparationOfSAndL} and Fact~\ref{FactMalnormalRestriction}; then, once it is proved that $\theta \restrict \pi_1 S$ lives in $\MCG(S)$, the rest of item~\pref{ItemPeripheralPermuted} follows from Fact~\ref{FactMalnormalRestriction}. Item~\pref{ItemSubgroupFixed} itself, and the fact that $\theta \restrict \pi_1 S$ lives in $\MCG(S)$, are each an immediate consequence of the following:  

\begin{lemma}\label{LemmaGeomModelHE}
Under the same hypothesis as Proposition~\ref{PropVertToFree}, and letting $K'$ be the union of the noncontractible components of $K$, for any $\theta \in \Out(F_n)$ such that $\theta[\pi_1 K]=[\pi_1 K]$ there exists a homotopy equivalence of pairs $\Theta_X \from (X,K') \to (X,K')$ and a homeomorphism $\Theta_S \from S \to S$ such that:
\begin{enumerate}
\item\label{ItemThetaXRepresents}
$\Theta_X$ represents $\theta$ with respect to the isomorphism $\pi_1 X \xrightarrow{d_*} \pi_1(G) \approx F_n$.
\item\label{ItemThetaSHomotopy}
The maps $j \composed \Theta_S$ and $\Theta_X \composed j \from S \to X$ are homotopic.
\end{enumerate} 
\end{lemma}

\begin{proof} Since the components of $\bdy S$ are $\pi_1$-injective in~$X$, the $j$-image of each is contained in a component of $K'$. 

List the components of $K'$ as $K_1 \union\ldots\union K_\ell$. From Lemma~\ref{LemmaLImmersed}, the subgroup system $[\pi_1 K]$ has exactly $\ell$ distinct components $[\pi_1 K_1],\ldots,[\pi_1 K_\ell]$, and if $i \ne j \in \{1,\ldots,\ell\}$ then no nontrivial conjugacy class of $F_n$ is carried by both $[\pi_1 K_i]$ and $[\pi_1 K_j]$.

First we construct $\Theta_X \from (X,K') \to (X,K')$ satisfying~\pref{ItemThetaXRepresents}. Start with any homotopy equivalence $\Theta \from X \to X$ that represents~$\theta$. For each component $K_i$ of $K'$ there exists another component $K_{\sigma(i)}$ such that $\Theta_*[\pi_1 K_i] = [\pi_1 K_{\sigma(i)}]$. The function $\sigma \from \{1,\ldots,\ell\} \to \{1,\ldots,\ell\}$ must be a permutation, because otherwise circuits in distinct conjugacy classes will have images in the same conjugacy class. Since $X$ and each $K_i$ are Eilenberg-Maclane spaces, we can homotope each of the restrictions $\Theta \restrict K_i$ to a map $\Theta_i \from K_i \to K_{\sigma(i)}$ which is a $\pi_1$-isomorphism, and each such $\Theta_i$ is a homotopy equivalence. The union of functions $\Theta_1 \union\cdots\union \Theta_\ell \from K' \to K'$ is therefore also a homotopy equivalence. By the homotopy extension property, the homotopy from $\Theta \restrict K'$ to $\Theta_1 \union\cdots\union \Theta_\ell $ extends to a homotopy of all of~$\Theta$.

Next we construct the homeomorphism $\Theta_S \from S \to S$ satisfying~\pref{ItemThetaSHomotopy}. For this purpose we arrange to improve the map $g_1 = \Theta_X \composed j \from S \to X$ by a sequence of homotopies. To start, we already have $g_1(\bdy S) = \Theta_X(j(\bdy S)) \subset \Theta(K) \subset K$.

By homotoping $g_1$ rel $\bdy S$, we shall produce a map $g_2 \from S \to X$ such that $g_2(S) \subset j(S) \union K$, still retaining the property that $g_2(\bdy S) \subset K$. We again use Stallings method from \cite{Stallings:transversality} as in the proof of Lemma~\ref{LemmaScaffoldFFS}. Consider $X \setminus (j(S) \union K)$ which is a union $E_1 \union\cdots\union E_\nu$ of ``naked'' edges of~$X$, meaning edges of $L$ whose interiors are open subsets of~$X$. Letting $M$ denote the set of midpoints of the edges $E_1,\ldots,E_\nu$, we may perturb $g_1$ to be transverse to~$M$, preserving that $g_1(\bdy S) \subset K$. The set $g_1^\inv(M)$ is disjoint union of circles each of which, by $\pi_1$-injectivity, bounds a disc in the interior of $S$. Starting with an innermost disc $D$, by a homotopy of $g_1$ supported on a slightly larger disc we may remove $\bdy D$ from $g^\inv(M)$, and after finitely many of these homotopies we have arranged that $g_1(S) \intersect M = \emptyset$. We may then post-compose $g_1$ with a deformation retraction $(X \setminus (j(S) \union K)) - M \mapsto j(S) \union K$, to produce the desired map $g_2 \from S \to X$.

By the Simplicial Approximation Theorem \cite{Spanier} we may subdivide $S$ and $X$ as simplicial complexes, and perturb $g_2$ by a homotopy, preserving the property that $g_2(\bdy S) \subset K$, so that $g_2$ is a simplicial map. Having done that, consider the subcomplex $R \subset S$ defined by $R = g_2^\inv(L) = g_2^\inv(K)$, where this equation follows from the fact that $g_2(S) \subset j(S) \union K$ and the fact that $j(S) \intersect L \subset K$. Note in particular that $\bdy S \subset R$. 

We claim that any (homotopically) nontrivial closed curve $\gamma$ in $S$ that is contained in $R$ is peripheral in $S$, meaning that $\gamma$ is homotopic to a closed curve in~$\bdy S$. To see why, note that $\Theta(j(\gamma))$ is a nontrivial closed curve in~$X$ homotopic to the closed curve $g_2(\gamma)$ in~$K$. Since $\Theta$ restricts to a homotopy equivalence of $K$, there exists a nontrivial closed curve $\gamma'$ in $K$ such that $\Theta(\gamma')$ is homotopic to $\Theta(j(\gamma))$. Since $\Theta$ is a homotopy equivalence of $X$, the closed curves $\gamma'$ and $j(\gamma)$ are homotopic. Applying 
Lemma~\ref{LemmaLImmersed}~\pref{ItemSeparationOfSAndL} it follows that $\gamma$ is peripheral in~$S$, proving the claim.

Next, by further homotopy of $g_2$, we produce a map $g_3 \from S \to X$ such that $g_3(S) \subset j(\interior(S)) = X-L$. Let $N(R)$ be a regular neighborhood of $R$ in the surface $S$, and so $N(R)$ is a compact subsurface of $S$ containing $\bdy S$. We shall prove that $N(R)$ can be engulfed in a regular neighborhood of $\bdy S$. From the claim of the previous paragraph it follows that any closed curve contained in $N(R)$, in particular any component of $\bdy N(R)$, is either homotopically trivial or peripheral in~$S$. Let $N^+(R)$ be the union of $N(R)$ with every disc bounded by a homotopically trivial component of $\bdy N(R)$ and every annulus bounded by two peripheral components of $\bdy N(R)$. Again $N^+(R)$ is a compact subsurface of $S$ containing $R$ and so containing $\bdy S$, but now each component of $N^+(R)$ is either a disc or an annulus neighborhood of $\bdy S$. Let $D_1,\ldots,D_J$ be the disc components of $N^+(R)$. Let $\alpha_1,\ldots,\alpha_J$ be a pairwise disjoint collection of arcs properly embedded in the subsurface $\overline{S - N^+(R)}$ such that $\alpha_j$ has one endpoint on $\bdy D_j$ and another endpoint on an annulus component of $N^+(R)$. Let $N^{++}(R)$ be a regular neighborhood of $N^+(R) \union (\alpha_1 \union\cdots\union \alpha_J)$. It follows that $N^{++}(R)$ is a disjunion union of annulus neighborhoods of the components of $\bdy S$, that is, a regular neighborhood of $\bdy S$. By the Regular Neighborhood Theorem, $N^{++}(R)$ is ambiently isotopic in $S$ to any regular neighborhood, and so the surface $S$ is therefore isotopic in itself to its subsurface $\closure(S - N^{++}(R))$. Postcomposing this isotopy with the map $g_2$, the resulting composed map $g_3 \from S \to X$ given by $S \mapsto \closure(S - N^{++}(R)) \subset S \xrightarrow{g_2} X$ has image contained in $X - L = j(\interior(S))$, as desired.

Since $g_3(S) \subset j(\interior(S))$ and since $j$ restricts to an embedding on $\interior(S)$, we may lift $g_3$ to a self map $\ti g_3 \from S \to S$. By construction it follows that the maps $j \composed \ti g_3$ and $\Theta_X \composed j \from S \to X$ are homotopic, and so it suffices to prove that $\ti g_3$ is homotopic to a homeomorphism of $S$. 

The induced homomorphism of fundamental groups $(\ti g_3)_* \from \pi_1 S \to \pi_1 S$ is an automorphism, because with respect to the embedding $\pi_1 S \to \pi_1(X)$ induced by $j$ the map $(\ti g_3)_*$ is the restriction of an automorphism of $F_n$ representing $\theta \in \Out(F_n)$, and by a theorem of Peter Scott found in Lemma~6.0.6 of \BookOne\ the restriction of an automorphism of $F_n$ to any invariant finite rank subgroup is an automorphism of that subgroup. The automorphism $(\ti g_3)_*$ preserves the conjugacy classes of the boundary components of $S$, because by applying Lemma~\ref{LemmaLImmersed}~\pref{ItemSeparationOfSAndL}, for any conjugacy class $C$ in $\pi_1 S$ the following are equivalent: $C$ is represented by a closed curve in $\bdy S$; the image conjugacy class $[j(C)]=[g_3(C)]$ in $F_n$ is carried by both $[\pi_1 K]$ and $[\pi_1 S]$.
By the Dehn-Nielsen-Baer theorem \cite{FarbMargalit:primer}, $\ti g_3 \from S \to S$ is homotopic to a homeomorphism $\Theta_S$. 
\end{proof}

\subsection{Preserving the free boundary circles.} 
\label{SectionFreeBoundaryInvariant}
We continue with the fixed notation of the previous section. 

For application in \PartTwo\ we need Lemma~\ref{LemmaFreeDefRetr} which gives a stronger conclusion than Lemma~\ref{LemmaGeomModelHE} regarding invariance of \emph{free boundary circles} of~$S$---by definition, these are the components $\bdy_i S$ of $\bdy S$ for which there exists a half-open annulus neighborhood  $A \subset S$ such that the map $S \xrightarrow{j} X$ restricts to a homeomorphism from $A$ onto an open subset of~$X$. 

For the statement of Lemma~\ref{LemmaFreeDefRetr} we will also use a stronger hypothesis than Lemma~\ref{LemmaGeomModelHE}, requiring that the free factor support of the complementary subgroup system $[\pi_1 L]$ is all of~$F_n$. From  that hypothesis several consequences follow:

\begin{fact}
\label{FactNoAttachingPoints}
If $\F_\supp[\pi_1 L] = \{[F_n]\}$ then $X$ has no attaching points and every component of $L$ is noncontractible. Furthermore, every edge group of $\Gamma(X)$ is infinite cyclic as is every edge stabilizer of $T(X)$.
\end{fact}

\begin{proof} There is a subgraph $\wh L \subset \wh X$ which is mapped homeomorphically to $L$ under the homotopy equivalence quotient map $\wh X \mapsto X$ that collapses each of the edges $\eta_k$ to the corresponding attaching point $v_k$. The graph $\wh L$ is disjoint from the interior of each $\eta_k$. So if any $v_k$ and its corresponding $\eta_k$ do exist, then $\wh L$ is contained in a subcomplex of $\wh X$ which represents a proper free factor of $X$, contradicting that $\F_\supp[\pi_1 L] = \{[F_n]\}$ and therefore proving that $X$ has no attaching points.  

To prove that $L$ has no contractible component it suffices to prove that each valence~$1$ vertex $v \in L$ is contained in a noncontractible subgraph of~$L$. First, $v \in H_r$ because if $v \not\in H_r$ then, since $G$ is a core graph, $v$ would have valence~$\ge 2$ in $G \setminus H_r \subset L$. Also, $v \not\in \bdy_0 S$ because then $v$ would have valence~$\ge 2$ in $\bdy_0 S \subset L$; in particular $v \ne p_r$. Also, $v \not\in \int(H_r;G_{r-1})$, the interior of $H_r$ relative to $G_{r-1}$, because, since $v \ne p_r$, it would follow that $v$ is an attaching point. The only remaining possibility is that $v \in H_r - \interior(H_r;G_{r-1}) \subset G_{r-1}$. But each component of $G_{r-1}$ is noncontractible (Fact~\ref{FactGeometricCharacterization}).

By Definition~\ref{DefPeripheralSplitting}, all edge groups are infinite cyclic or trivial, and trivial edge groups are associated only to the edges obtained by blowing up attaching points, of which there are none as we have just shown.
\end{proof}

The following lemma characterizes free boundary circles in terms of the action of $F_n$ on the Bass-Serre $T(X)$. The final characterization \pref{ItemLInS} is purely algebraic and will play an important role in proving invariance of free boundary circles.

\begin{lemma}
\label{LemmaFreeBdyVertex}
For each boundary circle $\bdy_i S$, letting $L_{l(i)}$ be the component of $L$ containing $j(\bdy_i S)$, letting $v \in \Gamma(X)$ be the vertex corresponding to $L_{l(i)}$, and letting $V \in T(X)$ be any vertex projecting to $v$, the following are equivalent:
\begin{enumerate}
\item\label{ItemIsFreeBdy}
$\bdy_i S$ is a free boundary circle.
\item\label{ItemIsDownstairsThorn}
$v$ has valence~$1$ in $\Gamma(X)$, and letting $e \subset \Gamma(X)$ be the edge incident to $v$, the injection from the edge group of $e$ to the vertex group of $v$ is an isomorphism.
\item\label{ItemIsUpstairsThorn}
$V$ has valence~$1$ in $T(X)$.
\end{enumerate}
If further $\F_\supp[\pi_1 L] = \{[F_n]\}$ then the following is also equivalent:
\begin{enumeratecontinue}
\item\label{ItemLInS}
For each subgroup $A$ in the conjugacy class $[\pi_1 L_{l(i)}]$ there exists a unique subgroup $B$ in the conjugacy class $[\pi_1 S]$ such that some finite index subgroup of $A$ is contained in~$B$.
\end{enumeratecontinue}
\end{lemma}

\begin{proof} The equivalence \pref{ItemIsFreeBdy}$\iff$\pref{ItemIsDownstairsThorn} is an immediate consequence of the construction of $\Gamma(X)$, and the equivalence \pref{ItemIsDownstairsThorn}$\iff$\pref{ItemIsUpstairsThorn} is a consequence of Bass-Serre theory.

For proving the equivalence  \pref{ItemIsUpstairsThorn}$\iff$\pref{ItemLInS}, we may rechoose $V$ within its orbit so that $A = \Stab(V)$. The implication \pref{ItemIsUpstairsThorn}$\implies$\pref{ItemLInS} holds without the ``further'' hypothesis, as follows. Let $V \in T(X)$ have valence~$1$. Letting $E_1$ be the unique edge of $T(X)$ incident to $V$ and letting $W_1$ be the opposite vertex of $E_1$, clearly $W_1$ projects to the $\wh S$ vertex. Setting $B = \Stab(W_1)$, we have $A = \Stab(V)=\Stab(E_1) \subgroup \Stab(W_1) = B$, and these subgroups are all nontrivial by Fact~\ref{FactNoAttachingPoints}. Arguing by contradiction, if \pref{ItemLInS} does not hold then there exists a vertex $W_2 \ne W_1 \in T(X)$ also projecting to the $S$ vertex of $\Gamma(X)$ such that the subgroup $\Stab(\overline{VW_2}) = \Stab(V) \intersect \Stab(W_2)$ is nontrivial. The arc $\overline{VW_2}$ contains the edge $E_1$ incident to $V$ and at least one other edge; let $E_2$ be the next edge after $E_1$. The subgroup $\Stab(\overline{VW_2})$ is contained in the stabilizers of each of $E_1,E_2$, and so $\Stab(E_1) \intersect \Stab(E_2)$ is nontrivial. But $E_1,E_2$ are distinct edges incident to $W_1$, and so $\Stab(E_1)$ and $\Stab(E_2)$ are stabilizers of distinct components of the boundary of the universal cover of $\wh S$, which implies that $\Stab(E_1) \intersect \Stab(E_2)$ is trivial, a contradiction.

Assume now that \pref{ItemLInS} holds. Applying \pref{ItemLInS}, there is a unique vertex $W_1 \in T(X)$ projecting to the $S$-vertex such that the subgroup $A_1 = A \intersect \Stab(W_1)$ has finite index in $A$. Letting $E_1$ be the first edge on the arc $\overline{VW_1}$, it follows that $A' = A \intersect B = \Stab(V) \intersect \Stab(W_1) \subgroup \Stab(E_1)$. Since $L_j$ is noncontractible (Fact~\ref{FactNoAttachingPoints}), the subgroups $A$ and $A'$ are nontrivial. By Fact~\ref{FactNoAttachingPoints}, all edges in $T(X)$ have infinite cyclic stabilizer. Since $\Stab(E_1)$ is infinite cyclic it follows that $A'$ is infinite cyclic, and hence so is $A$. Assuming that $V$ does not have valence~$1$, the tree $T(X)$ has another edge $E_2$ incident to $V$; let $W_2$ be the vertex of $E_2$ opposite $V$. Since $\Stab(E_2) \subset \Stab(V)= A$, since $\Stab(E_2)$ is nontrivial, and since $A$ is infinite cyclic, it follows that $\Stab(E_2)$ is infinite cyclic and has finite index in $A$. But since $\Stab(E_2) \subset \Stab(W_2)$ we have contradicted uniqueness of $W_1$, proving that $V$ has valence~$1$.
\end{proof}

Let $\bdy^\free S$ be the union of free boundary circles of $S$. For example, $\bdy_0 S \subset \bdy^\free S$ if and only if $p_r \not\in G \setminus G_r$, equivalently $p_r$ is an interior point of $H_r$; if $H_r$ is the top stratum then this holds automatically. The map $j$ embeds $\bdy^\free S$ into $X$; we identify $\bdy^\free S$ with its image in~$X$. Note that $\bdy^\free S$ is a union of components of $L$; the remaining components are denoted $L^\nonfree = L - \bdy^\free S$. Let $\bdy^\nonfree S = \bdy S - \bdy^\free S$, and notice that $j(\bdy^\nonfree S) \subset L^\nonfree$. Recall the notation $N_i$ for closed regular neighborhoods of the boundary circles $\bdy_i S$, and let $A_i$ be the interior of $N_i$ which is identified by $j$ with an open subset of $X$. Define the \emph{free subsurface} $S^\free$ of $S$ to be the closure of $S - \union\{N_i \suchthat \bdy_i S \subset \bdy^\nonfree S\}$, a compact subsurface to which $S$ deformation retracts. The map $j \from S \to X$ embeds $S^\free$ into $X$ and we identify $S^\free$ with its image. It follows that $L^\nonfree$ and $S^\free$ are disjoint closed sucomplexes of $X$ and that $X - (L^\nonfree \union S^\free)$ is the disjoint union of open annuli $A_i = \interior(N_i)$ over all $i$ such that $\bdy_i S \subset \bdy^\nonfree S$. Under the deformation retraction $S \mapsto S^\free$ let $\bdy_i S \leftrightarrow \bdy_i S^\free$ be the induced bijection of boundary circles: if $\bdy_i S \in \bdy^\free S$ then $\bdy_i S = \bdy_i S^\free$; whereas if $\bdy_i S \in \bdy^\nonfree S$ then $\bdy_i S \intersect \bdy_i S^\free = \emptyset$ and $\bdy_i S \union \bdy_i S^\free = \bdy N_i$. The map $j \restrict N_i$ may be regarded as a homotopy in $X$ between between the embedded circle $j(\bdy_i S^\free) \subset X$ and the closed edge path $j(\bdy_i S)$ in~$L^\nonfree$.

\begin{lemma}[Addendum to Lemma \ref{LemmaGeomModelHE}]
\label{LemmaFreeDefRetr} For any $\theta \in \Out(F_n)$ such that $\theta[\pi_1 L]=[\pi_1 L]$, and assuming that $\F_\supp[\pi_1 L] = \{[F_n]\}$, there exists a homotopy equivalence $\Theta_X \from (X,L) \to (X,L)$ representing $\theta$ and satisfying the conclusions of Lemma~\ref{LemmaGeomModelHE} so that the following also hold:
\begin{enumerate}
\item In addition to preserving $L$, also $\Theta_X$ preserves the sets $L^\nonfree$, $\bdy^\free S$, $S^\free$, and $X \setminus S^\free$.
\item The restriction $\Theta_X \restrict S^\free$ is a homeomorphism of $S^\free$.
\end{enumerate}
\end{lemma}

\begin{proof} By Fact~\ref{FactNoAttachingPoints} the set of attaching points is empty and it follows that $L \intersect j(S) = j(\bdy S)$. Consider now the map $\Theta_X \from (X,L) \to (X,L)$ that is produced by Lemma~\ref{LemmaGeomModelHE}.  

We first prove that $\Theta_X(\bdy^\free S)=\bdy^\free S$ which, by Lemma~\ref{LemmaFreeBdyVertex}, is a union of components of~$L$. Since each component of $L$ is noncontractible (Fact~\ref{FactNoAttachingPoints}), and since the subgroup system $[\pi_1 L]$ is malnormal with one component for each component of $L$ (Lemma~\ref{LemmaLImmersed}~\pref{ItemComplementMalnormal}), it follows that the restriction $\Theta_X \from L \to L$ permutes the components of $L$, that $\theta$ permutes the components of the subgroup system $[\pi_1 L]$, and that these permutations are compatible under the bijection between components of $L$ and of $[\pi_1 L]$. Let $\alpha \in \Aut(F_n)$ be any representative of $\theta$. Since $\theta$ preserves the subgroup systems $[\pi_1 L]$ and $[\pi_1 S]$, and since their disjoint union $[\pi_1 L] \union [\pi_1 S]$ is precisely the set of conjugacy classes of nontrivial vertex stabilizers for the action of $F_n$ on $T(X)$, it follows that $\alpha$ induces a unique permutation of the  vertices of $T(X)$ with nontrivial stabilizers so that $\Stab(\alpha \cdot V) = \alpha(\Stab(V))$, and furthermore this action preserves the $L$-vertices and the set of $S$-vertices. Given any component $\bdy_i S$ of $\bdy^\free S$, we must find a component $\bdy_j S$ of $\bdy^\free S$ such that $\Theta_X(\bdy_i S) = \bdy_j S$. Let $w \in \Gamma(X)$ be the vertex corresponding to $\bdy_i S$, and let $W \in T(X)$ be any vertex projecting to $w$. Using that \pref{ItemIsFreeBdy}$\implies$\pref{ItemLInS} in Lemma~\ref{LemmaFreeBdyVertex}, there exists a unique $S$-vertex $V \in T(X)$ such that $\Stab(W) \intersect \Stab(V)$ has finite index in $\Stab(W)$. Since $\alpha$ is an automorphism of $F_n$, it follows that $\alpha \cdot V$ is the unique $S$-vertex such that $\Stab(\alpha \cdot W) \intersect \Stab(\alpha \cdot V)$ has finite index in $\Stab(\alpha \cdot W)$. Using that \pref{ItemLInS}$\implies$\pref{ItemIsFreeBdy} in Lemma~\ref{LemmaFreeBdyVertex} it follows that $\alpha \cdot W$ projects to an $L$-vertex of $\Gamma(X)$ corresponding to some component $\bdy_j S$ of $\bdy^\free S$, and so $\Theta_X(\bdy_i S) = \bdy_j S$.

We know from Lemma~\ref{LemmaGeomModelHE} that $\Theta_X \restrict S^\free$ is homotopic to a self-homeomorphism, and by the homotopy extension lemma that we may extend this to a homotopy of the entire map $\Theta_X$ relative to the disjoint subcomplex~$L^\nonfree$. Note that if $\bdy_i S = \bdy_i S^\free$, $\bdy_j S = \bdy_j S^\free$ are free boundary circles, and if the equation $\Theta_X(\bdy_i S^\free) = \bdy_j S^\free$ holds before the homotopy, then it holds as well after the homotopy. To see why, after the isotopy the right hand side of the equation is some component $\bdy_{k} S^\free$ of $\bdy S^\free$ (not neccessarily a free boundary circle), and it would follow that $\bdy_j S^\free$ and $\bdy_k S^\free$ are homotopic in~$X$. But then by applying Lemma~\ref{LemmaLImmersed}~\pref{ItemComplementMalnormal} which says that $[\pi_1 L]$ is a malnormal subgroup system, and using the fact that any nonfree boundary component of $S^\free$ is homotopic into $L$, it follows that $\bdy_j S^\free = \bdy_k S^\free$.

\medskip

Our remaining task is to homotope the map $\Theta_X$ relative to $S^\free \union L^\nonfree$ so as the guarantee that $\Theta_X$ preserves $X - S^\free$. Note that the components of $X - (S^\free \union L^\nonfree)$ are identified under the map $j \restrict \interior(S)$ with the \emph{nonfree open annuli $A_i = \interior(N_i)$}, where the \emph{nonfree closed annuli $N_i \subset S$} are characterized by $\bdy N_i = \bdy_i S^\free \union \bdy_i S$ and $\bdy_i S \subset \bdy^\nonfree S$. It therefore suffices to consider each nonfree closed annulus $N_i$ and the annulus map $\beta = \Theta_X \composed j \from N_i \to X$, and to homotope $\beta$ relative to $\bdy N_i$ so that $\beta(A_i) \subset X-S^\free$.

We claim that there does exist a map $\gamma \from N_i \to X$ such that $\gamma \restrict \bdy N_i = \beta \restrict \bdy N_i$, and $\gamma(A_i) \subset X-S^\free$. Since $N_i$ is just an annulus, existence of $\gamma$ is equivalent to existence of a free homotopy in $X$ between the two circle maps $\beta \restrict \bdy_i S^\free$ and $\beta \restrict \bdy_i S$, such that the interior of this free homotopy is disjoint from $S^\free$. To do this it suffices to construct appropriate free homotopies between each of these two circle maps and the circle map $\zeta$ defined to be the composition 
$$\zeta \from \bdy_i S^\free \xrightarrow{\Theta_X} \bdy_k S^\free  \mapsto \bdy_k S \xrightarrow{j} X
$$
where the map $\bdy_k S^\free \mapsto \bdy_k S$ is defined by choosing a product structure $N_k \approx S^1 \cross [0,1]$. Clearly the annulus map $N_k \xrightarrow{j} X$ itself provides a free homotopy between the circle map $\beta \restrict \bdy_i S^\free$ and the map $\zeta$, and the image of the interior of this free homotopy is the open annulus $A_k$ which is disjoint from $S^\free$. Also, the two circle maps $\zeta$ and $\beta \restrict \bdy_i S$ each have image in $L$, and they are freely homotopic in $X$, so by Lemma~\ref{LemmaLImmersed}~\pref{ItemComplementCircuits} these two circle maps are freely homotopic in $L$, disjoint from $S^\free$. This proves the claim.

It remains to prove that for any two annulus maps $\beta,\gamma \from N_i \to X$ which agree on $\bdy N_i$, the map $\beta$ is homotopic rel $\bdy N_i$ to a map with the same image as $\gamma$ (but $\beta$ need not be homotopic rel $\bdy N_i$ to $\gamma$ itself). This is proved by doubling the annulus $N_i$ across its boundary to form a torus $T$ and defining a map $\beta \union \gamma \from T \to X$ using $\beta$ on one copy of $N_i$ and $\gamma$ on the other. The image of the induced homomorphism $\Z^2 = \pi_1(T) \xrightarrow{(\beta \union \gamma)_*} \pi_1(X) \approx F_n$ is infinite cyclic. Since $X$ is an Eilenberg-Maclane space, it follows that the map $T \mapsto X$ extends continuously to a map $M \mapsto X$ defined on some solid torus $M$ with boundary~$T$. The two copies of $N_i$ comprising $T=\bdy M$ are homotopic rel boundary through $M$, and composing this homotopy with the map $M \mapsto X$ we obtain the desired homotopy rel boundary from $\beta$ to a map whose image equals that of $\gamma$.
\end{proof}

\section{Vertex groups and vertex group systems}
\label{SectionVertexGroups}

A free factor system in $F_n$ consists of the conjugacy classes of nontrivial vertex stabilizers for some minimal action of $F_n$ on a simplicial tree $T$ such that all edge stabilizers are trivial. We generalize this by introducing the concept of a \emph{vertex group system}, which consists of the conjugacy classes of nontrivial point stabilizers for a minimal action of $F_n$ on some $\reals$-tree $T$ with trivial arc stabilizers.\footnote{In the earlier version \cite{HandelMosher:SubgroupOutF_n} ``vertex group systems'' were defined more broadly in terms of very small $F_n$-trees.} 

There are two main results in this section. Proposition~\ref{PropVDCC} is a chain condition, the proof of which was suggested to us by Mark Feighn, which bounds the length of each linear chain of vertex groups in $F_n$ ordered by inclusion. Proposition~\ref{PropGeomVertGrSys} says that certain subgroup systems arising naturally from consideration of geometric models are vertex group systems. 

\subsection{Vertex group systems} 
\label{SectionVertexGroupSystems}

First we recall some elements of the theory of $F_n$-actions on $\reals$-trees. An \emph{$F_n$ tree} is a minimal, isometric action $F_n \act T$ of the group $F_n$ on an $\reals$-tree $T$ for which no point or end of $T$ is fixed by the whole action; we often just write $T$ for an $F_n$-tree, when the action is understood. The $F_n$-tree is \emph{small} is the subgroup of $F_n$ that stabilizes any nondegenerate arc in $T$ is either trivial or cyclic, in which case the stabilizer of each point has rank~$\le n$ \cite{GaboriauLevitt:rank,GJLL:index}. 

A proper, nontrivial subgroup $A \subgroup F_n$ is a \emph{vertex group} if there exists an $F_n$ tree $T$ with trivial arc stabilizers and an $x \in T$ such that $A = \Stab(x)$, in which case we say that the vertex group $A$ is \emph{realized at $x$ in $T$}. Each free factor is a vertex group, and is realized at some vertex of some simplicial $F_n$-tree with trivial edge stabilizers. 
 
Given an $F_n$-tree $T$ with trivial arc stabilizers, the \emph{vertex group system of $T$}, denoted~$\A_T$, is the subgroup system consisting of the conjugacy classes of all vertex groups that are realized in~$T$.  By definition of vertex group, each $A \in \A_T$ is proper and nontrivial. We note that in addition to the bound $n$ on the ranks of the vertex stabilizers in $T$, the cardinality of $\A_T$ is bounded above by a finite constant depending only on the rank~$n$; see \cite{GaboriauLevitt:rank,GJLL:index}. The vertex group system $\A_T$ is therefore an example of a subgroup system in~$F_n$.

A subgroup system $\A$ is said to be a \emph{vertex group system} if there exists an $F_n$-tree $T$ with trivial arc stabilizers such that $\A = \A_T$; we say in this situation that \emph{$\A$ is realized in~$T$}. Note that since arc stabilizers are trivial, the point stabilized by a subgroup $A$ in the system $\A_T$ is uniquely determined by $A$. Each free factor system is a vertex group system, and is realized in some simplicial $F_n$-tree with trivial edge stabilizers.

\begin{lemma}
\label{LemmaVSElliptics}
Every vertex group system $\A$ is malnormal. Furthermore, $\A$ is uniquely determined by the set $\C$ of conjugacy classes of nontrivial elements of $F_n$ that are carried by $\A$, in the sense that for any subgroup $B \subgroup F_n$ the following are equivalent:
\begin{enumerate}
\item The $F_n$-conjugacy class of each nontrivial element of $B$ is an element of the set $\C$.
\item There exists a subgroup $A \subgroup F_n$ such that $B \subgroup A$ and $[A] \in \A$.
\end{enumerate}
It follows that the subgroups $A \subgroup F_n$ for which $[A] \in \A$ are precisely the maximal subgroups satisfying item (1).
\end{lemma}

\begin{proof} Choose an $F_n$-tree $T$ with trivial arc stabilizers in which $\A$ is realized. The lemma is an immediate consequence of three elementary facts about group actions on trees: the stabilizers of any two distinct vertices of $T$ have trivial intersection; $\C$ consists precisely of the conjugacy classes of nontrivial elements of $F_n$ that are elliptic on~$T$; and for each subgroup $B \subgroup F_n$ whose elements are all elliptic on~$T$, there exists $p \in T$ such that $B \subgroup \Stab(p)$.
\end{proof}

\paragraph{The chain condition on vertex groups.} Grushko's Theorem implies that the collection of free factors satisfies a chain condition: every free factor $A \subset F_n$ satisfies $\rank(A) \le n$, and for any free factors $A < A'$ we have $\rank(A) \le \rank(A')$ with equality if and only if $A=A'$. We need the analogous result about vertex groups, the proof of which was suggested to us by Mark Feighn.


\begin{proposition}\label{PropVDCC} 
Every vertex group $A \subgroup F_n$ satisfies $\rank(A) \le n$. For any vertex groups $A < A'$ we have $\rank(A) \le \rank(A')$ with equality if and only if $A=A'$. It follows that there is an upper bound to the length $L$ of any properly nested sequence of vertex group systems $\A_1 \sqsubset \cdots \sqsubset \A_L$. 
\end{proposition}

\begin{proof} The last sentence follows from the previous ones using the fact from \cite{GaboriauLevitt:rank} that gives an upper bound to the cardinality 

The fact that $\rank(A) \le n$ follows from the index inequality which is the main result of \cite{GaboriauLevitt:rank}. It remains to prove that for any proper inclusion of vertex groups $A < A'$ we have $\rank(A) < \rank(A')$. Choose an $F_n$-tree $T$ in which $A$ is realized at some point $p \in T$. Let $S \subset T$ be a minimal $A'$-subtree of $T$, and so $S$ is either a point or a nondegenerate subtree of $T$. If $p \not\in S$ then $A$ stabilizes the shortest arc of $T$ from $p$ to a point of $S$, contradicting that arc stabilizers in $T$ are trivial. If $\{p\}=S$ then $A' \subgroup \Stab(S) = \Stab(p) = A$ and so $A = A'$, contradicting properness. If $p \in S$ and $\{p\} \ne S$ then $S$ is a small $A'$-tree and so $\rank(A) \le \rank(A')$ by the index inequality of \cite{GaboriauLevitt:rank}. If $\rank(A) = \rank(A')$ then by Theorem~5.2 of \cite{Paulin:AutExt}, which is attributed to Levitt, it follows that $S$ is a simplicial tree and the quotient graph of groups $S / A'$ is a rose each edge of which is labelled by an infinite cyclic group, and so each edge of $S$ has nontrivial stabilizer, contradicting that arc stabilizers in $T$ are trivial.
\end{proof}

\subparagraph{Remark.} In the earlier version of this paper \cite{HandelMosher:SubgroupOutF_n} the reader will find Proposition~\ref{PropVDCC} stated in the broader context of point stabilizers of very small $F_n$-trees. The proof, again suggested to us by Mark Feighn, is more complicated, but depends on the same tools, namely the index inequality of \cite{GaboriauLevitt:rank} and Theorem~5.2 of \cite{Paulin:AutExt}. The additional work needed is to consider the case that $\rank(A_1)=\ldots=\rank(A_L)$ and to show, by a Stallings fold argument, that the first Betti number of the quotient of $A_1$ modulo the normal closure of $A_i$ is a bounded, strictly increasing function of $i=1,\ldots,L$.

\subsection{Geometric models and vertex group systems}

Here is our main construction of vertex group systems.

\begin{proposition} 
\label{PropGeomVertGrSys}
Given a rotationless $\phi \in \Out(F_n)$ represented by a \ct\ $f \from G \to G$, and a geometric stratum $H_r \subset G$, let $X$ be a geometric model for $H_r$ with accompanying notation as in Definition~\ref{DefGeomModel}, in particular the surface $S$ and the map $S \xrightarrow{j} Y \subset X$. Let $L \subset X$ be the complementary subgraph as in Definition~\ref{DefComplSubgraph}. For any subcomplex $K \subset L$ such that $j(\bdy S) \subset K$, the subgroup system $[\pi_1 K]$ is a vertex group system.
\end{proposition}

\begin{proof} We must construct an $F_n$-tree $T$ with trivial arc stabilizers in which $[\pi_1 K]$ is realized. The idea is to use the Morgan-Shalen construction of dual trees to measured geodesic laminations \cite{MorganShalen:FreeActions}, applied to a hyperbolic structure on the surface~$S$. The strategy follows the diagram below: we must first pry $S$ and $L$ apart (the left square of the diagram), apply the Morgan-Shalen construction (the middle square), and then put the pieces back together (the right square).
$$\xymatrix{
\wt X \ar[d]   & \P \ar[d] \ar[l]_{\ti q} \ar[r]^r              & \F \ar[d] \ar[r]     & T \ar[d] \\
X                  & L \disjunion S \ar[l]_q \ar[r] & \biggl( \F / F_n   \ar[r] \biggr)  & \biggl( T / F_n \biggr)     
}$$
The left column --- the universal covering of $X$ --- is known. On the top row all objects will be $F_n$-actions and all arrows will be $F_n$-equivariant maps. All vertical arrows will be quotients modulo~$F_n$. The $\reals$-tree $T$, and the $\reals$-forest $\F$ which precedes it in the order of construction, will be acted on indiscretely by $F_n$, and their quotients are shown below them, with the parentheses to emphasize their non-Hausdorff nature. It remains to describe, in order, $q$, $\P$, $\F$, and $T$ and various arrows.

The ``prying apart'' step is accomplished by the quotient map $q \from L \disjunion S \to X$ that is defined as follows: $q \restrict L$ is just the inclusion $L \subset X$; and $q \restrict S = j \restrict S$. The identifications made by $q$ are: the circles $\bdy_0 S \subset S$ and $j(\bdy_0 S) \subset L$ are identified; for $1 \le i \le m$ each point $x \in \bdy_i S$ is identified with $\gamma_i(x) \in G_{r-1} \subset L$; and for each attaching point $x \in j(\interior(S)) \intersect L$, the copies of $x$ in $L$ and in $\interior(S) \subset S$ are identified.

The space $\P$ is simply a pushout, namely, the set of ordered pairs in the Cartesian product $\wt X \cross (L \disjunion S)$ whose two coordinates have the same image in $X$. Projections to the two Cartesian factors define the leftward and downward arrows out of $\P$. The action $F_n \act \wt X$ induces an action $F_n \act \P$ and the arrow $\ti q \from \P \to \wt X$ is $F_n$-equivariant. Letting $L_1,\ldots,L_A$ be the components of $L$, the arrow $\P \mapsto L \disjunion S$ restricts on each component of $\P$ to a universal covering map over a certain component $L_1,\ldots,L_A,S$ of $L \disjunion S$. To be precise, fix an appropriate choice of base points and paths in $X$ in order to identify $\pi_1 S, \pi_1(L_1), \ldots, \pi_1(L_A)$ with subgroups of $\pi_1(X) \approx F_n$. Letting $\L(S)$ be a choice of left coset representatives of the subgroup $\pi_1 S \subgroup F_n$, the components of $\P$ over $S$ are indexed as $\wt S_c$ for $c \in \L(S)$, such that $\Stab(\wt S_c) = \pi_1 S^c = c^\inv \, \pi_1 S \, c \subgroup F_n$. Similarly, the components of $\P$ over each $L_a$ are indexed as $\wt L_{a,c}$, for $c$ in a set $\L(L_a)$ of left coset representatives of $\pi_1(L_a) \subgroup F_n$, so that $\Stab(\wt L_{a,c}) = \pi_1(L_a)^{c}$ (some components $L_a$ may be contractible in which case $\L(L_a) = F_n$).

In what follows we shall abuse notation by identifying the components $\wt S_c$, $\wt L_{a,c}$ of $\P$ with their embedded images $\ti q(\wt S_c)$, $\ti q ( \wt L_{a,c} )$ in $\wt X$.

The space $\F$ will be an $F_n$-forest obtained as a quotient of $F_n \act \P$ by collapsing each component $\wt S_c \subset \P$ to an $\reals$-tree, using the dual $\reals$-tree construction \cite{MorganShalen:FreeActions}, and by collapsing each component $\wt L_{a,c} \subset \P$ to a simplicial $\reals$-tree with appropriate vertex stabilizers.

Fix a hyperbolic structure with totally geodesic boundary on $S$ (below we shall alter this structure by isotopy in order to establish the important property $(*)$). Fix a geodesic lamination $\Lambda \subset S$ which fills $S$ and a transverse measure $\mu$ with support $\Lambda$. For example, one could take $\Lambda,\mu$ to be the unstable measured geodesic lamination of the pseudo-Anosov mapping class on $S$ that is part of the definition of the geodesic model. The outcome of the construction in \cite{MorganShalen:FreeActions} is a $\pi_1 S$-tree $T(\Lambda)$ dual to~$\Lambda$, as follows. Lift $\Lambda$ to a geodesic lamination $\wt\Lambda \subset \wt{S}$ on the universal covering space $\wt{S}$ of $S$. Lift $\mu$ to a transverse measure $\ti\mu$ with support~$\wt\Lambda$. There is a continuous pseudo-metric $\ti d(x,y) = \int_{[x,y]} \mu$ where $[x,y]$ is the geodesic segment connecting $x,y \in \wt{S}$ (to construct $\ti d$ we use, in part, that $\Lambda$ fills, for if $\Lambda$ had a compact leaf that that leaf would be an atom of $\mu$ and this definition of $\ti d$ would fail to produce a continuous pseudo-metric). Recalling that a principal region of $\wt\Lambda$ in $\wt{S}$ is the closure of a component of $\wt{S} - \wt \Lambda$, the value of $\ti d(x,y)$ depends only on the principal regions or nonboundary leaves of $\wt \Lambda$ containing $x,y$ respectively, and furthermore $\ti d(x,y) = 0$ if and only if $x,y$ are contained in the same principal region or nonboundary leaf. Let $\wt{S} \mapsto T(\Lambda)$ be the metric quotient of the pseudo-metric $\ti d$, collapsing to a point each principal region and each nonboundary leaf of $\wt\Lambda$, inducing a metric $d$ on $T(\Lambda)$ such that the quotient map $\wt{S} \mapsto T(\Lambda)$ is an isometry from $\ti d$ to $d$. These objects in $\wt{S}$ --- namely, $\wt\Lambda$, $\ti\mu$, $\ti d$ --- are all equivariant with respect to the deck transformation action $\pi_1 S \act \wt{S}$, and therefore there is an induced isometric action $\pi_1 S \act T(\Lambda)$. As proved in \cite{MorganShalen:FreeActions}, the space $T(\Lambda)$  with the metric $d$ is an $\reals$ tree, the action $\pi_1 S \act T(\Lambda)$ is minimal and has trivial arc stabilizers. Furthermore, the nontrivial point stabilizers can be described in any of the following ways: they are the subgroups of $\pi_1 S$ conjugate to the fundamental groups of crown principal regions of $\Lambda$ in $S$; they are the infinite cyclic groups conjugate to the fundamental groups of the boundary circles $\bdy_0 S,\ldots,\bdy_m S$; and they are the stabilizers of the components of $\bdy\wt{S}$.

Define a continuous pseudometric to each component of $\P$ as follows. On each component $\wt S_c$ over $S$ we use the pseudo-metric described above, arising from the lift of the measured folation $(\Lambda,\mu)$ on $S$. Also, for each component $L_a$, $a=1,\ldots,A$, we choose a path pseudo-metric which assigns length~$0$ to each edge of $L_a \intersect K$ and length~$1$ to each edge of $L_a \setminus K$, and then on each component $\wt L_{a,c}$ of $\P$ over $L_{a,c}$ we lift the path pseudometric of $L_a$. These pseudo-metrics are equivariant with respect to the deck transformation action $F_n \act \P$, in other words the action restricts to pseudometric isometries amongst the components of~$\P$.

Define $r \from \P \to \F$ to be the $F_n$-equivariant quotient which on each component $\wt{S}_c$ or $\wt L_{a,c}$ of $\P$ restricts to the metric quotient of the pseudo-metric on that component, producing a component of $\F$ which is an $\reals$-tree denoted $r(\wt{S}_c)$ or $r(\wt L_{a,c})$, respectively. The action $F_n \act \P$ descends via $r$ to an action $F_n \act \F$. The metrics on the components of $\F$ are equivariant with respect to this action, i.e.\ the action restricts to metric isometries amongst the components of~$F$. By the Morgan-Shalen theorem, each action $\pi_1(S)^c \act r(\wt{S}_c)$ is a minimal $\reals$-tree action whose vertex stabilizers are the conjugates in $F_n$ of the subgroups $\pi_1(\bdy_0 S),\ldots,\pi_1(\bdy_m S)$ that are contained in $\pi_1(S)^c$. Also, each action $\pi_1(L_a)^c \act r(\wt L_{a,c})$ is a Bass-Serre tree for the graph of groups structure on the quotient of $L_a$ obtained by collapsing to a point each component of $L_a \intersect K$, and the vertex stabilizers of this action are the subgroups of $F_n$ contained in $\pi_1(L_a)^c$ that are conjugate to $\pi_1(K_j)$, for each each component $K_j$ of $K$ that is contained in $L_a$. Putting this altogether, it follows that the nontrivial point stabilizers of the $\reals$-forest action $F_n \act \F$ are precisely the subgroups conjugate to the infinite cyclic groups $\pi_1(\bdy_0 S),\ldots,\pi_1(\bdy_m S)$ and the subgroups in the subgroup system $[\pi_1 K]$.

We shall refer to the components of $\F$ of the form $r(\wt S_c)$ as the \emph{$S$-components}, and to their points as \emph{$S$-points} of $\F$. Also, we refer to components of the form $r(\wt L_{a,c})$ as \emph{$L$-components} and their points as \emph{$L$-points} of $\F$. In this language, under the action $\F_n \act \F$, the subgroup system of stabilizers of all $L$-points of $\F$ is exactly $[\pi_1 K]$, and the subgroup system of stabilizers of all $S$ points of $\F$ is exactly $\{[\pi_1 \bdy_0 S],\ldots,[\pi_1 \bdy_m S]\}$. 

The $F_n$-tree $T$ will be constructed from $\F$ by identifying certain pairs of points in an $F_n$-equivariant manner. But we must not identify two $L$-points of $\F$ with nontrivial stabilizers for that will create a point of $T$ whose stabilizer is larger than any subgroup in the system $[\pi_1 K]$. We shall avoid identifying $L$-points altogether by imposing an additional constraint on our construction.

We may assume, after possibly isotoping the hyperbolic structure on $S$ by an isotopy supported on a compact subset of $\interior(S)$, that the following holds:
\begin{description}
\item[Property $(*)$] No attaching point lies on the $j$-image of a crown principal region of $\Lambda$ in~$S$. No two attaching points lie on the $j$-image of same nonboundary leaf of $\Lambda$ or the same ideal polygon principal region in $S$. 
\end{description} 
This is possible because $\Lambda$ has infinitely many leaves but only finitely many boundary leaves in $S$. The finite subset of attaching points may then be put into bijective correspondence with some finite subset of $j(\interior(S))$ that does satisfy $(*)$, and there is an isotopy of $S$ as desired taking one of these subsets to the other.

The tree $T$ is the quotient of $\F$ modulo the $F_n$-invariant partition generated by the following $F_n$-invariant relation: given $\sigma,\lambda \in \F$, an $S$-point and an $L$-point respectively, we say that $\sigma,\lambda$ are \emph{incident} if the subset $q(r^\inv(\sigma)) \intersect q(r^\inv(\lambda)) \subset \wt X$ is nonempty. Observe that as a consequence of Property~$(*)$, every $S$-point is incident to at most one $L$-point; the quotient map $\F \mapsto T$ is therefore injective on the set of $L$-points. Furthermore, from Property~$(*)$ it follows that if $\sigma,\lambda$ as above are incident then one of the following two cases holds. In the first case, $q(r^\inv(\sigma)) \intersect q(r^\inv(\lambda)) \subset \wt X$ is a single point covering an attaching point of $X$, which occurs only if there exist $c \in \L(S)$ and $c' \in \L(L_a)$, for some $a=1,\ldots,A$, such that $r^\inv(\sigma) \subset \wt S_c$ is either an ideal polygon principal region or nonboundary leaf, and $r^\inv(\lambda) \subset \wt L_{a,c'}$ is either a universal covering of a component of $L_a \intersect K$ or a single point lying over a point of $L_a - (L_a \intersect K)$; furthermore, in this first case the stabilizer of $\sigma$ is trivial. In the second case, $q(r^\inv(\sigma)) \intersect q(r^\inv(\lambda))$ is a component of $\bdy \wt S_c$ for some $c \in \L(S)$, which occurs only if $r^\inv(\sigma)$ is a universal covering over some crown principal region of $S_c$ and $r^\inv(\lambda) \subset \wt L_{a,c'}$ is a universal covering of some component of $L_a \intersect K$; furthermore, in this second case the stabilizer of $\sigma$ is in the subgroup system $\{[\pi_1 \bdy_0 S],\ldots,[\pi_1 \bdy_m S]\}$, and it is a subgroup of the stabilizer of $\lambda$ which is in the subgroup system $[\pi_1 K]$. Note also that because of the inclusion $j(\bdy S) \subset K$, every $S$-point $\sigma$ with nontrivial stabilizer is incident to some $L$-point $\lambda$ and the second case holds. Putting together this description of incidence and the stabilizers involved, it follows that the stabilizers of points of $T$ are exactly the subgroups in the subgroup system $[\pi_1 K]$, as required.

Note that it is indeed possible for an $L$-point $\lambda \in \F$ to be incident to more than one $S$-point: this happens when $r^\inv(\lambda) \subset \wt L_{a,c}$ is a universal covering of some component of $L_a \intersect K$ which contains at least two attaching points, or the $j$-images of at least two components of $\bdy_i S$, or at least one of each; this, however, is inconsequential to our present purposes. 

So far, $T$ has been described as a set on which $F_n$ acts, with point stabilizers being the required subgroup system $[\pi_1 K]$. Furthermore, the point stabilized by a particular subgroup in the system $[\pi_1 K]$ is unique. It follows that \emph{any} $\reals$-tree action $F_n \act T$ has trivial edge stabilizers. It remains to define an $\reals$-tree metric on $T$ for which the action $F_n \act T$ is isometric and minimal.

The metric on $T$ is defined as follows. There is a partially defined metric $d'$ on $T$ which is the union of the pushforwards via the map $\F \mapsto T$ of the metrics on the components of~$\F$. Define $d(x,y) = \inf \left( \sum_{i=0}^I d'(x_{i-1},x_i) \right)$ where the infimum is taken over all \emph{connecting sequences} from $x$ to $y$, meaning a sequence of the form $x=x_0,x_1,\ldots,x_I=y$ such that for each $i=1,\ldots,I$ the pair $x_{i-1},x_i$ is contained in the image of some component of $\F$. The definition of $d$ is evidently $F_n$-equivariant and so $F_n$ acts by isometries of~$d$. It is evident, at least, that $d$ is a pseudo-metric.

To check that $T$ is an $\reals$-tree we need to see that the components of $\F$ are connected up in a tree-like pattern under the map $\F \mapsto T$, in the following sense. Recalling $H_r \subset G \subset X$ in the data of the geometric model $X$, let $F_n \act U$ be the Bass-Serre tree of the graph of groups obtained from $G$ by collapsing to a point each component of $K \union H_r$, and note that this graph of groups may also be obtained from $X$ by collapsing to a point each component of $K \union j(S)$. It follows that $U$ is obtained from $\wt X$ by collapsing to a point each component of the union of the total lifts of $K$ and $S$. There is an induced map $s \from T \mapsto U$ whose restriction to each $\reals$-tree $r(\wt S_c)$ is a constant map onto the point $s(r(\wt S_c))$ which is the image under $\wt X \mapsto U$ of $\wt S_c \subset \wt X$, and whose restriction to each subtree $r(\wt L_{a,c})$ is a $\pi_1(L_a)^c$-equivariant embedding onto the subtree $s(r(\wt L_{a,c}))$ of $U$ which equals the image under the map $\wt X \to U$ of $\wt L_{a,c} \subset \wt X$. 
The map $s \from T \mapsto U$ is, set theoretically, simply the quotient that collapses each subtree of the form $r(\wt S_c)$ to a point $s(r(\wt S_c))$. The map $s$ therefore has the following effect on the subtrees of $T$ of the form $r(\wt S_c)$ and $r(\wt L_{a,c})$. First, two subtrees $r(\wt S_c), r(\wt L_{a,c'}) \subset T$ have nontrivial intersection in $T$ \emph{if and only} if their $s$-images in $U$ have nontrivial intersection \emph{if and only if} $s(r(\wt S_c))$ is a point of the subtree $s(r(\wt L_{a,c'}))$. Second, two subtrees $r(\wt L_{a,c}), r(\wt L_{a',c'})$ have $s$-images with nontrivial intersection in $U$ \emph{if and only if} they intersect a common subtree $r(\wt S_{c''}) \subset T$ \emph{if and only if} $s(r(\wt S_{c''})) \in s(r(\wt L_{a,c})) \intersect s(r(\wt L_{a',c'}))$, in which case the latter intersection consists solely of the point $s(r(\wt S_{c''}))$.

We now define the arc $[x,y] \subset T$ with any endpoints $x,y \in T$ and verify that $T$ is an $\reals$-tree. Assuming $sx \ne sy$, the arc $[sx,xy] \subset U$ is a unique concatenation of subarcs $[sx,sy] = [\xi_0,\xi_1] * \cdots * [\xi_{J-1},\xi_J]$, $J \ge 1$ such that each $[\xi_{j-1},\xi_j]$ is a maximal subarc contained in some subtree $s(r(\wt L_{k_j,c_j}))$, $j=1,\ldots,J$. Each such subarc pulls back uniquely to a subarc $[x_{j-1},x'_j] \subset r(\wt L_{k_j,c_j})$. Letting $x'_0 = x$ and $x_J = y$, for each $j=0,\ldots,J$ we have two points $x'_j,x_j$ in the same subtree $r(\wt S_{c'_j})$, connected by a unique arc $[x'_j,x_j] \subset r(\wt S_{c'_j})$. We therefore obtain an embedded arc in $T$ which is expressed as a concatenation without backtracking of the form
$$[x,y] = [x'_0,x_0] * [x_0,x'_1] * \cdots * [x_{J-1},x'_J] * [x'_J,x_J]
$$
whose first and last terms may be degenerate but whose other terms are all nondegenerate. This expression for $[x,y]$ makes sense even when $sx=sy$ but $x \ne y$, in which case $J=0$ and $x = x'_0, y=x_0 \in r(\wt S_{c'_0})$. From the form of the metric on $T$ we have 
$$d(x,y) \le d'(x'_0,x_0) + d'(x_0,x'_1) + \cdots + d'(x_{J-1},x'_J) + d'(x'_J,x_J)
$$
whose first and last terms may be zero but whose other terms are all nonzero. But in fact, this inequality is an equation, because \emph{any} connecting sequence from $x$ to $y$ must contain $x=x'_0,x_0,x'_1,\ldots,x_{J-1},x'_J,x_J$ as an increasing subsequence. The arc $[x,y]$ is therefore isometrically embedded in $T$. And it is the unique embedded arc in $T$ connecting $x$ to $y$, because any such arc must contain $x=x'_0,x_0,x'_1,\ldots,x_{J-1},x'_J,x_J$ as an increasing subsequence, and the arcs between successive points are unique in the corresponding subtrees. This shows that $T$ is an $\reals$-tree.

Finally, minimality of the action $F_n \act T$ is an immediate consequence of minimality of the actions $F_n \act U$ and $\pi_1(S)^c \act r(\wt S_c)$ for each $c \in \L(\pi_1 S)$.

This completes the proof of Proposition~\ref{PropGeomVertGrSys}.
\end{proof}
 
\bibliographystyle{amsalpha} 
\bibliography{mosher} 
 
 \printindex
 
 \end{document}